\documentclass{article}[12pt]
\usepackage[english]{babel}
\usepackage{geometry}
\usepackage[utf8]{inputenc}
\usepackage{subcaption} 
\usepackage{tikz}
\usepackage{float}
\usepackage{amsmath,bm}
\usepackage{amssymb}
\usepackage{amsthm}
\usepackage{cite}
\usepackage{verbatim}
\usepackage{tikz-cd}
\usepackage{wrapfig}
\usepackage{indentfirst}
\usepackage{mathtools}
\usepackage{mathrsfs}
\usepackage{authblk}

\usepackage{physics}
\usepackage{mathdots}
\usepackage{yhmath}
\usepackage{cancel}
\usepackage{color}
\usepackage{siunitx}
\usepackage{array}
\usepackage{multirow}
\usepackage{gensymb}
\usepackage{tabularx}
\usepackage{booktabs}
\usetikzlibrary{fadings}
\usetikzlibrary{patterns}
\usetikzlibrary{shadows.blur}
\usetikzlibrary{shapes}
\usepackage{hyperref}

\usepackage{soul}

\newtheorem{lemma}{Lemma}
\newtheorem{thm}{Theorem}
\newtheorem{remark}{Remark}

\title{Skeleton-stabilized divergence-conforming B-spline discretizations for highly advective incompressible flow problems}
\author[1]{Guoxiang Grayson Tong}
\author[2]{David Kamensky}
\author[3]{John A. Evans}
\affil[1]{Department of Applied and Computational Mathematics and Statistics, University of Notre Dame, Notre Dame, IN, USA, 46556}
\affil[2]{Department of Mechanical and Aerospace Engineering, University of California San Diego, La Jolla, CA, USA, 92093 }
\affil[3]{Ann and H.J. Smead Aerospace Engineering Sciences, University of Colorado Boulder, Boulder, CO, USA, 80303}

\date{}
\begin{document}

\maketitle
\begin{abstract}
We consider a stabilization method for divergence-conforming B-spline discretizations of the incompressible Navier--Stokes problem wherein jumps in high-order normal derivatives of the velocity field are penalized across interior mesh facets.  We prove that this method is pressure robust, consistent, and energy stable, and we show how to select the stabilization parameter appearing in the method so that excessive numerical dissipation is avoided in both the cross-wind direction and in the diffusion-dominated regime.  We examine the efficacy of the method using a suite of numerical experiments, and we find the method yields optimal $\textbf{L}^2$ and $\textbf{H}^1$ convergence rates for the velocity field, eliminates spurious small-scale structures that pollute Galerkin approximations, and is effective as an Implicit Large Eddy Simulation (ILES) methodology.\\

\noindent {\it{Keywords}}: Isogeometric analysis; Divergence-conforming discretizations; Edge stabilization; Skeleton stabilization; Pressure robustness; Incompressible Navier--Stokes equations
\end{abstract}

\section{Introduction}
Over the past decade, divergence-conforming B-splines have emerged as an attractive candidate for the spatial discretization of incompressible fluid flow problems \cite{BA11,evans2013isogeometric,EV12,EV13,VA17,EV18}.  When applied to the Galerkin approximation of the incompressible Navier--Stokes problem, divergence-conforming B-splines produce pointwise divergence-free velocity fields and thus exactly satisfy mass conservation.  Divergence-conforming B-spline Galerkin approximations additionally conserve linear and angular momentum, energy, vorticity, enstrophy (in the two-dimensional setting), and helicity (in the three-dimensional setting) in the inviscid limit, and they yield velocity fields whose error is independent of the pressure field (such methods are referred to as pressure robust).  Much like other isogeometric analysis methodologies, divergence-conforming B-splines exhibit a better approximation behavior per degree of freedom versus classical finite elements \cite{evans2009n}, and they are superior to classical finite elements in approximating both advective \cite{bazilevs2007variational} and diffusive \cite{evans2012discrete} processes.  The multi-level and algebraic structure of divergence-conforming B-splines has even enabled the development of scalable solution methodologies for divergence-conforming B-spline discretizations \cite{coley2018geometric}.  Divergence-conforming B-splines have also found application in multi-physics applications for which pointwise mass conservation is a desirable attribute.  For instance, they have recently been leveraged to overcome the issue of poor mass conservation in immersed methods for computational fluid-structure interaction \cite{KA17,casquero2018non,casquero2021divergence}.

While divergence-conforming B-spline Galerkin approximations have many attractive properties, like other Galerkin approximations, they are susceptible to advective instabilities.  A popular approach to address this problem is to use the Streamline Upwind Petrov--Galerkin/Pressure Stabilizing Petrov--Galerkin (SUPG/PSPG) method \cite{BR82,hughes1986new} in conjunction with grad-div stabilization \cite{jenkins2014parameter}.  While this provides robustness in the advection-dominated limit, it unfortunately destroys both the exact mass conservation and pressure robustness properties of divergence-conforming B-spline Galerkin approximations.  An alternative is to ignore the PSPG terms in the SUPG/PSPG method \cite{gelhard2005stabilized}, but this too destroys pressure robustness and additionally is not provably robust with respect to advection.  Building on the work of \cite{ten2018correct,TENEIKELDER20181135}, we recently developed a stabilized method that both preserves discrete (and for divergence-conforming discretizations, exact) mass conservation and is provably robust in the advection-dominated limit \cite{EV20}, but this comes with a loss of pressure robustness and the inclusion of an additional fine-scale pressure variable.  In another recent paper, a stabilized method was proposed for divergence-conforming discretizations of the linearized incompressible Navier--Stokes problem that preserves exact mass conservation, is pressure robust, and is provably robust with respect to advection \cite{ahmed2021pressure}, but it includes a residual-based least squares stabilization of the vorticity equation that is complicated and potentially expensive to evaluate.

In this paper, we propose a simple stabilization strategy for divergence-conforming B-spline approximations of the incompressible Navier--Stokes problem that alleviates advective instabilities without sacrificing exact mass conservation or pressure robustness.  Our strategy extends so-called edge stabilization methodologies \cite{BU04,BU06,BU0606,Bu07,Burman2007,BU08} to divergence-conforming B-spline approximations.  In particular, we penalize jumps in high-order normal derivatives of the velocity field across mesh facets.  A similar strategy was proposed in \cite{HO18,HO19} to alleviate pressure instabilities associated with isogeometric velocity/pressure pairs that do not satisfy a Babu\v{s}ka--Brezzi inf-sup condition, so following \cite{HO18,HO19}, we refer to our strategy as skeleton stabilization.  Due to the special structure of divergence-conforming B-splines, our proposed stabilization acts only on the components of the velocity field that are tangential to interior mesh facets.  This property is not shared by other edge stabilization methodologies, but it does apply to upwinding strategies used in $\textbf{H}$(div)-conforming discontinuous Galerkin methodologies \cite{schroeder2018divergence}.  We prove in this paper that skeleton-stabilized B-spline discretizations with weakly enforced tangential Dirichlet boundary conditions \cite{BA0702} are consistent, pressure robust, and energy stable and balance both linear and angular momentum, and we further show how to select the stabilization parameter appearing in our strategy to avoid excessive numerical dissipation in the diffusion-dominated limit as well as in the cross-wind direction.  We demonstrate numerically our strategy yields optimal convergence rates for the velocity field in both the $\textbf{L}^2$-norm and the $\textbf{H}^1$-norm, and we also show our strategy is capable of eliminating spurious small-scale structures that pollute divergence-conforming B-spline Galerkin approximations.  Finally, we examine the efficacy of our stabilization strategy as an Implicit Large Eddy Simulation (ILES) methodology using the well-known 3D Taylor--Green vortex benchmark problem.  All numerical experiments reported on in this paper were carried out using the tIGAr library \cite{KA19} which extends the popular FEniCS finite element automation software \cite{LO12} to isogeometric analysis.  The source code for each of these experiments is available in a public GitHub repository\footnote{\url{https://github.com/Grayson3455/SS-Div-Conforming-B-splines-Incompressible-NS}}.

An outline of this paper is as follows.  In Section \ref{NT}, we establish notation used throughout the paper and provide a review of divergence-conforming B-splines.  In Section \ref{NS}, we introduce the incompressible Navier--Stokes problem and our skeleton-stabilized scheme for solving this problem.  In Section \ref{ANALYSIS}, we provide guidance for how to select the stabilization parameter appearing in our skeleton-stabilized scheme, and we also prove our method is pressure robust, consistent, and energy stable.  In Section \ref{NE}, we examine the effectiveness of our skeleton-stabilized scheme using a steady manufactured solution, the lid-driven cavity problem, and the 3D Taylor--Green vortex problem.  We provide concluding remarks in Section \ref{cl}.

\section{Notation}
\label{NT}
We begin by establishing some conventions for denoting function spaces and mathematical operations.
\subsection{B-spline function spaces}
The core idea of isogeometric analysis (IGA) \cite{HU05,CoHuBa09} is to approximate partial differential equation (PDE) solutions using the spline function spaces popular in computer aided design (CAD).  Among the most popular spline spaces are B-splines.  We briefly review B-splines here, for the sake of establishing notation, but refer the unfamiliar reader to \cite{PiegTil97} for a more comprehensive definition.  The piecewise polynomial sections of B-splines defined on a $d$-dimensional parametric space are delimited in each direction by knot vectors $\boldsymbol{\xi}_i = [\xi_{i,1}, \cdots, \xi_{i,m_i}]$ for $i\in\{1,\ldots,d\}$, where $m_i$ is the number of knots in direction $i$.  We take $k_i$ to be the polynomial degree in direction $i$ for $i\in\{1,\ldots,d\}$.  We assume that each of the knot vectors is an open knot vectors, that is, we assume that the first and last knots of the knot vector $\boldsymbol{\xi}_i$ are repeated $p_i+1$ times.  B-spline functions can be used on non-rectangular domains by introducing a mapping $\mathcal{F}$ from the parametric domain to a physical domain over which PDEs are posed, i.e., $\mathcal{F}(\hat{\Omega}) = \Omega$, where $\hat{\Omega}$ is the parametric domain and $\Omega$ is the physical domain. 

Basis functions of a $d$-variate B-spline are defined as tensor products of univariate B-spline basis functions, which are computed recursively by the famous Cox--de Boor algorithm \cite{HU05}.  Smoothness at interfaces between elements is controlled by the multiplicities of the corresponding knots.  Let $\boldsymbol{\zeta}_i = [\zeta_{i,1}, \cdots, \zeta_{i,n_i}]$ be the vector of unique knot values in direction $i$ for $i\in\{1,\ldots,d\}$, where $n_i$ is the number of unique knots in direction $i$, and let $\beta_{i,j}$ be the multiplicity of the unique knot $\zeta_{i,j}$.  Then univariate basis functions in direction $i$ are $C^{k_i - \beta_{i,j}}$ differentiable at $\zeta_{i,j}$.  We encode this regularity in a set of regularity vectors $\boldsymbol{\alpha}_i = [\alpha_{i,1}, \cdots, \alpha_{i,n_i}]$ for $i\in\{1,\ldots,d\}$ where $\alpha_{i,j} = k_i - \beta_{i,j}$.  By construction, $\alpha_{i,1} = \alpha_{i,n_i} = -1$.  The space of tensor products of functions from $d$ univariate B-spline spaces is fully determined by the polynomial degrees $k_1, \ldots, k_d$, the regularity vectors $\boldsymbol{\alpha}_1, \ldots, \boldsymbol{\alpha}_d$, and the parametric mesh $\hat{\mathcal{T}}_h$ defined by the vectors of unique knot values $\boldsymbol{\zeta}_1, \ldots, \boldsymbol{\zeta}_d$, so we introduce the notation $\mathcal{B}_{\boldsymbol{\alpha}_1,\cdots,\boldsymbol{\alpha}_d}^{k_1,\cdots,k_d}(\hat{\mathcal{T}}_h)$ to denote this space.  For the sake of brevity, we later write $\mathcal{B}_{\boldsymbol{\alpha}_1,\cdots,\boldsymbol{\alpha}_d}^{k_1,\cdots,k_d}$ instead of $\mathcal{B}_{\boldsymbol{\alpha}_1,\cdots,\boldsymbol{\alpha}_d}^{k_1,\cdots,k_d}(\hat{\mathcal{T}}_h)$.  We will also adopt the convention $\boldsymbol{\alpha}_i - 1 = [-1,\alpha_{i,2}-1, \cdots,\alpha_{i,n_i-1}-1, -1]$.

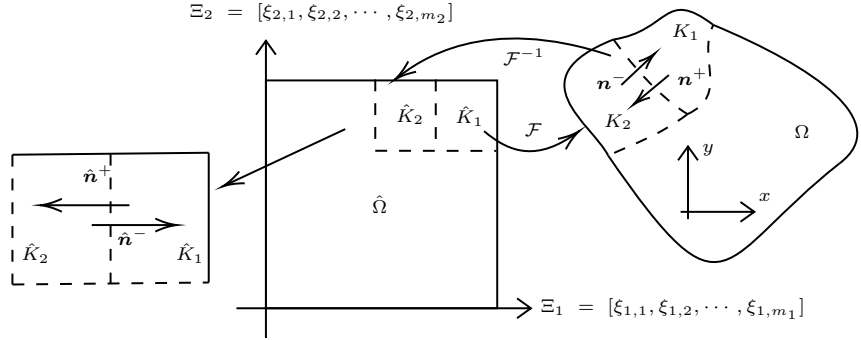
\begin{figure}[b!]
\centering
\tikzset{every picture/.style={line width=0.75pt}} 

\begin{tikzpicture}[x=0.75pt,y=0.75pt,yscale=-1,xscale=1]

\draw  (207,247) -- (355.5,247)(221.85,112) -- (221.85,262) (348.5,242) -- (355.5,247) -- (348.5,252) (216.85,119) -- (221.85,112) -- (226.85,119)  ;
\draw  [color={rgb, 255:red, 0; green, 0; blue, 0 }  ,draw opacity=1 ] (221.85,132) -- (338.5,132) -- (338.5,247) -- (221.85,247) -- cycle ;
\draw    (331.86,157.57) .. controls (345.51,169.27) and (359.16,172.42) .. (378.37,158.66) ;
\draw [shift={(379.86,157.57)}, rotate = 503.13] [color={rgb, 255:red, 0; green, 0; blue, 0 }  ][line width=0.75]    (10.93,-3.29) .. controls (6.95,-1.4) and (3.31,-0.3) .. (0,0) .. controls (3.31,0.3) and (6.95,1.4) .. (10.93,3.29)   ;
\draw    (395.86,119.57) .. controls (331.6,97.84) and (298.59,122.54) .. (288.41,130.39) ;
\draw [shift={(286.86,131.57)}, rotate = 323.13] [color={rgb, 255:red, 0; green, 0; blue, 0 }  ][line width=0.75]    (10.93,-3.29) .. controls (6.95,-1.4) and (3.31,-0.3) .. (0,0) .. controls (3.31,0.3) and (6.95,1.4) .. (10.93,3.29)   ;
\draw  (431,198.33) -- (468.17,198.33)(434.72,165.33) -- (434.72,202) (461.17,193.33) -- (468.17,198.33) -- (461.17,203.33) (429.72,172.33) -- (434.72,165.33) -- (439.72,172.33)  ;
\draw [color={rgb, 255:red, 0; green, 0; blue, 0 }  ,draw opacity=1 ] [dash pattern={on 4.5pt off 4.5pt}]  (397.17,111.33) .. controls (412,132) and (426,144) .. (434.5,149) ;
\draw [color={rgb, 255:red, 0; green, 0; blue, 0 }  ,draw opacity=1 ] [dash pattern={on 4.5pt off 4.5pt}]  (307.5,131.54) -- (307.5,167.54) ;
\draw [color={rgb, 255:red, 0; green, 0; blue, 0 }  ,draw opacity=1 ]   (397.17,111.33) .. controls (425.5,107) and (422.5,85) .. (447.5,104) ;
\draw [color={rgb, 255:red, 0; green, 0; blue, 0 }  ,draw opacity=1 ]   (394.5,170) .. controls (386.5,154) and (347.5,131) .. (397.17,111.33) ;
\draw [color={rgb, 255:red, 0; green, 0; blue, 0 }  ,draw opacity=1 ] [dash pattern={on 4.5pt off 4.5pt}]  (434.5,149) .. controls (420,160) and (407,164) .. (394.5,170) ;
\draw [color={rgb, 255:red, 0; green, 0; blue, 0 }  ,draw opacity=1 ] [dash pattern={on 4.5pt off 4.5pt}]  (434.5,149) .. controls (456,139) and (439,121) .. (447.5,104) ;
\draw [color={rgb, 255:red, 0; green, 0; blue, 0 }  ,draw opacity=1 ]   (511,174) .. controls (450,222) and (455,258) .. (394.5,170) ;
\draw [color={rgb, 255:red, 0; green, 0; blue, 0 }  ,draw opacity=1 ]   (511,174) .. controls (539,153) and (498,125) .. (447.5,104) ;
\draw [color={rgb, 255:red, 0; green, 0; blue, 0 }  ,draw opacity=1 ] [dash pattern={on 4.5pt off 4.5pt}]  (277,131.5) -- (277,167.5) ;
\draw [color={rgb, 255:red, 0; green, 0; blue, 0 }  ,draw opacity=1 ] [dash pattern={on 4.5pt off 4.5pt}]  (277,167.5) -- (338,167.57) ;
\draw    (261.86,156.57) -- (201.67,184.72) ;
\draw [shift={(199.86,185.57)}, rotate = 334.93] [color={rgb, 255:red, 0; green, 0; blue, 0 }  ][line width=0.75]    (10.93,-3.29) .. controls (6.95,-1.4) and (3.31,-0.3) .. (0,0) .. controls (3.31,0.3) and (6.95,1.4) .. (10.93,3.29)   ;
\draw [color={rgb, 255:red, 0; green, 0; blue, 0 }  ,draw opacity=1 ] [dash pattern={on 4.5pt off 4.5pt}]  (94,169.5) -- (94,235) ;
\draw [color={rgb, 255:red, 0; green, 0; blue, 0 }  ,draw opacity=1 ] [dash pattern={on 4.5pt off 4.5pt}]  (94,235) -- (193,234) ;
\draw [color={rgb, 255:red, 0; green, 0; blue, 0 }  ,draw opacity=1 ] [dash pattern={on 4.5pt off 4.5pt}]  (143.5,169) -- (143.5,234.5) ;
\draw [color={rgb, 255:red, 0; green, 0; blue, 0 }  ,draw opacity=1 ]   (193,168.5) -- (193,234) ;
\draw [color={rgb, 255:red, 0; green, 0; blue, 0 }  ,draw opacity=1 ]   (94,169.5) -- (193,168.5) ;
\draw    (153,195) -- (110,195) ;
\draw [shift={(108,195)}, rotate = 360] [color={rgb, 255:red, 0; green, 0; blue, 0 }  ][line width=0.75]    (10.93,-3.29) .. controls (6.95,-1.4) and (3.31,-0.3) .. (0,0) .. controls (3.31,0.3) and (6.95,1.4) .. (10.93,3.29)   ;
\draw    (134,205) -- (175,205) ;
\draw [shift={(177,205)}, rotate = 180] [color={rgb, 255:red, 0; green, 0; blue, 0 }  ][line width=0.75]    (10.93,-3.29) .. controls (6.95,-1.4) and (3.31,-0.3) .. (0,0) .. controls (3.31,0.3) and (6.95,1.4) .. (10.93,3.29)   ;
\draw    (425.1,129.3) -- (408.64,143.02) ;
\draw [shift={(407.1,144.3)}, rotate = 320.19] [color={rgb, 255:red, 0; green, 0; blue, 0 }  ][line width=0.75]    (8.74,-2.63) .. controls (5.56,-1.12) and (2.65,-0.24) .. (0,0) .. controls (2.65,0.24) and (5.56,1.12) .. (8.74,2.63)   ;
\draw    (401.1,134.3) -- (416.64,119.67) ;
\draw [shift={(418.1,118.3)}, rotate = 496.74] [color={rgb, 255:red, 0; green, 0; blue, 0 }  ][line width=0.75]    (8.74,-2.63) .. controls (5.56,-1.12) and (2.65,-0.24) .. (0,0) .. controls (2.65,0.24) and (5.56,1.12) .. (8.74,2.63)   ;

\draw (182,91.4) node [anchor=north west][inner sep=0.75pt]  [font=\scriptsize]  {$\Xi _{2} \ =\ [ \xi _{2,1} ,\xi _{2,2} ,\cdots ,\xi _{2,m_{2}}]$};
\draw (351,152.4) node [anchor=north west][inner sep=0.75pt]  [font=\scriptsize]  {$\mathcal{F}$};
\draw (340,114.4) node [anchor=north west][inner sep=0.75pt]  [font=\scriptsize]  {$\mathcal{F}^{-1}$};
\draw (274,188.4) node [anchor=north west][inner sep=0.75pt]  [font=\scriptsize]  {$\hat{\si{\ohm}}$};
\draw (487,153.4) node [anchor=north west][inner sep=0.75pt]  [font=\scriptsize]  {$\si{\ohm}$};
\draw (425,102.4) node [anchor=north west][inner sep=0.75pt]  [font=\scriptsize]  {$K_{1}$};
\draw (469,187.4) node [anchor=north west][inner sep=0.75pt]  [font=\scriptsize]  {$x$};
\draw (440.5,162.4) node [anchor=north west][inner sep=0.75pt]  [font=\scriptsize]  {$y$};
\draw (359,240.4) node [anchor=north west][inner sep=0.75pt]  [font=\scriptsize]  {$\Xi _{1} \ =\ [ \xi _{1,1} ,\xi _{1,2} ,\cdots ,\xi _{1,m_{1}}]$};
\draw (391,147.4) node [anchor=north west][inner sep=0.75pt]  [font=\scriptsize]  {$K_{2}$};
\draw (286,142.4) node [anchor=north west][inner sep=0.75pt]  [font=\scriptsize]  {$\hat{K}_{2}$};
\draw (316,143.4) node [anchor=north west][inner sep=0.75pt]  [font=\scriptsize]  {$\hat{K}_{1}$};
\draw (97,212.4) node [anchor=north west][inner sep=0.75pt]  [font=\scriptsize]  {$\hat{K}_{2}$};
\draw (175,212.4) node [anchor=north west][inner sep=0.75pt]  [font=\scriptsize]  {$\hat{K}_{1}$};
\draw (128,173.4) node [anchor=north west][inner sep=0.75pt]  [font=\scriptsize]  {$\hat{\boldsymbol{n}}^{+}$};
\draw (145.5,205.15) node [anchor=north west][inner sep=0.75pt]  [font=\scriptsize]  {$\hat{\boldsymbol{n}}^{-}$};
\draw (428,124.4) node [anchor=north west][inner sep=0.75pt]  [font=\scriptsize]  {$\boldsymbol{n}^{+}$};
\draw (387,126.4) node [anchor=north west][inner sep=0.75pt]  [font=\scriptsize]  {$\boldsymbol{n}^{-}$};

\end{tikzpicture}  
\caption{Diagram illustrating our notation for B-spline spaces, with emphasis on the mesh skeleton.}
\label{fig:iga}
\end{figure}

When generalizing ideas from classical finite element analysis (FEA) to IGA, it can be helpful to describe spline spaces in the language of FEA.  We can define a physical mesh
\begin{equation}
\mathcal{T}_h := \left\{ \mathcal{F}(\hat{K}): \hat{K} \in \hat{\mathcal{T}}_h \right\}
\end{equation}
such that $\Omega = \bigcup_{K \in \mathcal{T}_h} \overline{K}$.  The size of each element $K \in \mathcal{T}_h$ in the physical mesh can be calculated as
\begin{equation}
h_K \coloneqq \text{diam}(K) = \sup_{\boldsymbol{x,y} \in K} |\boldsymbol{x} - \boldsymbol{y}|\text{ ,}
\end{equation}
and the mesh $\mathcal{T}_h$ is assumed to be quasi-uniform with global mesh size $h$, i.e., $\exists C,h > 0$ such that $\forall K\in\mathcal{T}_h$
\begin{equation}
    Ch \leq h_K \leq h/C\text{ .}
\end{equation}
In this paper, the facets between elements and on the boundary of the mesh---forming the mesh skeleton, of  (topological) dimension $d-1$---are of particular interest.  The sets of interior and boundary facets are denoted
\begin{align}
\mathcal{E}_0  &\coloneqq \{e ~|~ e=\partial K_1 \cap \partial K_2, e \neq \emptyset, K_1 \neq K_2, K_1, K_2 \in \mathcal{T}_h\}\text{ ,}\\
\mathcal{E}_{\partial}  &\coloneqq \{e ~|~ e=\partial K \cap \partial \Omega, e \neq \emptyset, K \in \mathcal{T}_h \}\text{ .}
\end{align}
Each interior facet $e \in \mathcal{E}_0$ is shared between a ``positive'' element $K_1$ and a ``negative'' element $K_2$.  We define $\boldsymbol{n}^+$ to be the outward facing normal to $K_1$ and $\boldsymbol{n}^-$ to be the outward facing normal to $K_2$.  We employ the convention $\boldsymbol{n} = \boldsymbol{n}^+ = -\boldsymbol{n}^-$.  For a field $\phi$ that is defined over both $K_1$ and $K_2$, we define $\phi^+(\boldsymbol{x}) = \lim_{\epsilon \rightarrow 0^+} \phi(\boldsymbol{x} - \epsilon\boldsymbol{n}^+)$, $\phi^-(\boldsymbol{x}) = \lim_{\epsilon \rightarrow 0^+} \phi(\boldsymbol{x} - \epsilon\boldsymbol{n}^-)$, and $[\![ \phi(\boldsymbol{x}]\!] = \phi^+(\boldsymbol{x}) - \phi^-(\boldsymbol{x})$ for $\boldsymbol{x} \in e$.  Figure \ref{fig:iga} illustrates the relationship between parametric space and the physical space, as well as how interior facets ($--$), boundary facets ($-$), and facet normal vectors are defined in each element of $\hat{\Omega}$ and $\Omega$.

\subsection{Sobolev spaces and norms}
We will state weak forms of PDEs using the following vector-valued Sobolev spaces:
\begin{align}
\mathbf{H}^1_0(\Omega) &\coloneqq \{\boldsymbol{u} \in \mathbf{H}^1(\Omega)~\vert~\boldsymbol{u} = \boldsymbol{0} \ \text{at} \ \partial \Omega \}\text{ ,}\\
\mathbf{H}_{\boldsymbol{n}}^1(\Omega) &\coloneqq \{\boldsymbol{u} \in \mathbf{H}^1(\Omega)~\vert~\boldsymbol{u} \cdot \boldsymbol{n} = 0 \ \text{at} \ \partial \Omega \}\text{ ,}\\
\mathbf{H}(\text{div}, \Omega)&\coloneqq \{ \boldsymbol{u} \in \mathbf{L}^2(\Omega)~\vert~\nabla \cdot \boldsymbol{u} \in L^2(\Omega) \}\text{ ,}\\
L^2_0 &\coloneqq \left\{ p \in L^2(\Omega)~\left\vert~ \int_{\Omega} p \ d\Omega = 0 \right\}\right.\text{ ,}
\end{align}
where $\Omega\subset\mathbb{R}^d$ is a domain, $\mathbf{H}^1(\Omega) = (H^1(\Omega))^d$, $\mathbf{L}^2(\Omega) = (L^2(\Omega))^d$, $\boldsymbol{n}$ is the outward-facing unit normal to $\partial\Omega$, and boundary conditions are understood in the sense of traces.  We shall consider time-dependent problems in this work, but, because the topic of this paper is principally a method of stabilizing the {\em spatial} discretization of such problems, we will abuse notation by suppressing the time-dependence of function spaces.

For analyzing discretizations, it is also convenient to define the following ``broken'' Sobolev norms, which are sums over norms on mesh entities:
\begin{equation*}
\begin{aligned}
&||\cdot||^2_{\mathbf{L}^2(\mathcal{T}_h)}   \coloneqq  \sum_{K \in \mathcal{T}_h} ||\cdot||^2_{\mathbf{L}^2(K)}\text{ ,}
&&||\cdot||^2_{\mathbf{L}^2(\mathcal{E}_{\partial})}    \coloneqq \sum_{e \in \mathcal{E}_{\partial}} ||\cdot||^2_{\mathbf{L}^2(e)}\text{ ,}
&&||\cdot||^2_{\mathbf{L}^2(\mathcal{E}_{0})}   \coloneqq \sum_{e \in \mathcal{E}_{0}} ||\cdot||^2_{\mathbf{L}^2(e)}\text{ .}
\end{aligned}
\end{equation*}
Due to the prevalence of $L^2$ inner products throughout this work, we reserve the generic notation $(\cdot,\cdot)$ for them, unless otherwise specified, i.e.,
\begin{equation}
    (f,g) = \int_\Omega fg\,d\Omega\quad\text{,}\quad(\boldsymbol{u},\boldsymbol{v}) = \int_\Omega\boldsymbol{u}\cdot\boldsymbol{v}\,d\Omega\quad\text{,}\quad(\boldsymbol{T},\boldsymbol{S}) = \int_\Omega\boldsymbol{T}:\boldsymbol{S}\,d\Omega
\end{equation}
for scalar fields $f$ and $g$, vector fields $\boldsymbol{u}$ and $\boldsymbol{v}$, and tensor fields $\boldsymbol{T}$ and $\boldsymbol{S}$.  The generic norm $\Vert\cdot\Vert$ will similarly refer to the $L^2$ norm of its argument, unless otherwise stated.  Above, we adopt the convention $\boldsymbol{C} : \boldsymbol{D} = C_{ij} D_{ij}$ for two tensors $\boldsymbol{C} = \sum_{i=1}^d \sum_{j=1}^d C_{ij} \boldsymbol{e}_i \otimes \boldsymbol{e}_j$ and $\boldsymbol{D} = \sum_{i=1}^d \sum_{j=1}^d D_{ij} \boldsymbol{e}_i \otimes \boldsymbol{e}_j$.  Later in this paper, we also take the gradient of a vector $\boldsymbol{a} = \sum_{j=1}^d a_{j} \boldsymbol{e}_j$ to be $\nabla \boldsymbol{a} := \sum_{i=1}^d \sum_{j=1}^d \frac{\partial a_j}{\partial x_i} \boldsymbol{e}_i \otimes \boldsymbol{e}_j$ where $\left\{ \boldsymbol{e}_j \right\}_{j=1}^d$ is the standard basis for $\mathbb{R}^d$, and we take the divergence of a tensor $\boldsymbol{A} = \sum_{i=1}^d \sum_{j=1}^d A_{ij} \boldsymbol{e}_i \otimes \boldsymbol{e}_j$ to be $\nabla \cdot \boldsymbol{A} := \sum_{i=1}^d \sum_{j=1}^d \frac{\partial A_{ij}}{\partial x_i} \boldsymbol{e}_j$.

\subsection{Divergence-conforming B-splines}

B-spline spaces can be used to construct discrete subcomplexes of the de Rham complex \cite{AR18}, enabling discretizations that exactly satisfy conservation laws of continuous PDE systems.  In particular, for divergence-conforming discretizations of the Navier--Stokes equations with exact mass conservation, we are concerned with the commuting diagram
\begin{equation}\label{eq:commuting-diagram}
\begin{tikzcd}[]
\mathbf{H}^1_{\boldsymbol{n}}(\Omega)\arrow[d, "\Pi_{\boldsymbol{v}}"]  \arrow[r, "{\nabla \cdot}"] & L^2_0(\Omega)\arrow[d, "\Pi_q"] \\
\pmb{\mathcal{V}}_h  \arrow[r, "{\nabla \cdot}"] & \mathcal{Q}_h 
\end{tikzcd}
\end{equation}
where $\mathbf{H}^1_{\boldsymbol{n}}(\Omega)$ is the velocity space, $L^2_0(\Omega)$ is the pressure space, $\pmb{\mathcal{V}}_h \subset \mathbf{H}^1_{\boldsymbol{n}}(\Omega)$ and $\mathcal{Q}_h \subset L^2_0(\Omega)$ are their discrete counterparts, and $\Pi_{\boldsymbol{v}}$ and $\Pi_q$ are suitable projection operators, as defined in \cite{buffa2011isogeometric}.  In 3D ($d=3$), pulled-back discrete spaces on the B-spline parameter space are defined by
\begin{align}
\hat{\pmb{\mathcal{V}}}_h = \iota_{\pmb{\mathcal{V}}} (\pmb{\mathcal{V}}_h) & \coloneqq  \mathcal{B}^{k_1, k_2-1, k_3-1}_{\boldsymbol{\alpha}_1, \boldsymbol{\alpha}_2-1, \boldsymbol{\alpha}_3-1} \times  \mathcal{B}^{k_1-1, k_2, k_3-1}_{\boldsymbol{\alpha}_1-1, \boldsymbol{\alpha}_2, \boldsymbol{\alpha}_3-1} \times  \mathcal{B}^{k_1-1, k_2-1, k_3}_{\boldsymbol{\alpha}_1-1, \boldsymbol{\alpha}_2-1, \boldsymbol{\alpha}_3}~\cap~\iota_{\pmb{\mathcal{V}}}\left(\mathbf{H}^1_{\boldsymbol{n}}(\Omega)\right)\text{ ,}\\
\hat{\mathcal{Q}}_h = \iota_{\mathcal{Q}}({\mathcal{Q}}_h)  & \coloneqq  \mathcal{B}^{k_1-1, k_2-1, k_3-1}_{\boldsymbol{\alpha}_1-1, \boldsymbol{\alpha}_2-1, \boldsymbol{\alpha}_3-1}~\cap~\iota_{\mathcal{Q}}(L^2_0(\Omega))\text{ ,}
\end{align}
where $\iota_{\mathcal{\pmb{V}}}$ and $\iota_{\mathcal{Q}}$ are pullback mappings defined by
\begin{align}
    \iota_{\pmb{\mathcal{V}}}(\boldsymbol{v}) =& \operatorname{det}(D\mathcal{F})(D\mathcal{F})^{-1}(\boldsymbol{v}\circ\mathcal{F})\text{ ,}\\
    \iota_{\mathcal{Q}}(q) =& \operatorname{det}(D\mathcal{F})(q\circ\mathcal{F})\text{ ,}
\end{align}
The discrete velocity space $\pmb{\mathcal{V}}_h$ is polynomially-complete up to the degree
\begin{equation}
    k' = \min_{i\in\{1,2,3\}}\{k_i-1\}\text{ ,}
\end{equation}
which we identify as the overall ``degree'' of the velocity--pressure space $\pmb{\mathcal{V}}_h\times\mathcal{Q}_h$.  Functions in $\pmb{\mathcal{V}}_h$ are globally $C^{\alpha'}$-continuous where
\begin{equation}
    \alpha' = \min_{i\in\{1,2,3\}} \min_{2 \leq j \leq n_i-1} \{\alpha_{i,j}-1\}\text{ .}
\end{equation}
For maximally-smooth splines, we have $\alpha' = k'-1$.  To guarantee the commuting diagram \eqref{eq:commuting-diagram} holds, we require that $\alpha' \geq 0$.  Commutation of \eqref{eq:commuting-diagram} ensures that the bottom row of \eqref{eq:commuting-diagram} is a subcomplex of the Stokes complex in the first row.  In particular, the divergence of every function from $\pmb{\mathcal{V}}_h$ is in $\mathcal{Q}_h$, so that the discrete mass conservation equation
\begin{equation}
    (\nabla\cdot\boldsymbol{u}_h,q_h) = 0\quad\forall q_h\in\mathcal{Q}_h
\end{equation}
implies strong mass conservation, i.e., $\nabla\cdot\boldsymbol{u}_h = 0$ at every point in $\Omega$.  The velocity--pressure space $\pmb{\mathcal{V}}_h\times\mathcal{Q}_h$ is entirely determined by $k'$, $\alpha'$, and the parametric mesh $\hat{\mathcal{T}}_h$ when the polynomial degrees in each parametric direction and the multiplicities of each internal knot are identical.  We assume this to be the case for the remainder of the paper.

\begin{remark}
The spaces $\pmb{\mathcal{V}}_h$ are $\mathbf{H}^1$-conforming generalizations of the $\mathbf{H}(\operatorname{div})$-conforming Raviart--Thomas spaces \cite{raviart1975mixed} from classical finite element analysis.
\end{remark}

\begin{remark}
The construction of 2D velocity and pressure spaces is similar (cf. \cite{EV12}).
\end{remark}

\section{Problem and spatial discretization}
\label{NS}
We now introduce the incompressible Navier--Stokes problem and our spatial discretization of it using skeleton stabilization.

\subsection{The incompressible Navier--Stokes equations}
This paper concerns the incompressible Navier--Stokes equations.  For simplicity, we will consider this system with pure Dirichlet boundary conditions, which is given in strong form as
\begin{equation}
(\mathcal{S}): 
\begin{cases}
\textrm{Find} \ \boldsymbol{u}: \bar{\Omega}\times (0,T]  \to \mathbb{R}^d, p: \bar{\Omega}\times (0,T]  \to \mathbb{R} \ \textrm{s.t.}\\[7.5pt]
\begin{aligned}
 \frac{\displaystyle \partial \boldsymbol{u}} {\displaystyle \partial t} + \nabla \cdot (\boldsymbol{u} \otimes \boldsymbol{u}) + \nabla p - 2\nu \nabla \cdot (\nabla^s\boldsymbol{u})  &= \boldsymbol{f} \ \  &&\textrm{in} \ \Omega \times (0,T]\\
 \nabla \cdot \boldsymbol{u} &= 0 \ \  &&\textrm{in} \ \Omega \times (0,T]\\
\boldsymbol{u} &= \boldsymbol{u}_D \ \ &&\textrm{on} \ \partial \Omega \times (0,T] \\
\boldsymbol{u} &= \boldsymbol{u}_0 \ \ &&\textrm{in} \ \Omega \ \text{at} \ t=0
\end{aligned}
\end{cases}
\end{equation}
where the unknown solution fields $\boldsymbol{u}$ and $p$ have the interpretations of velocity and pressure, $\boldsymbol{f}:\bar{\Omega}\times(0,T\rbrack\to\mathbb{R}^d$ is a given source term, interpreted as a body force field, $\boldsymbol{u}_D:\partial\Omega\times (0,T\rbrack\to\mathbb{R}^d$ is given Dirichlet boundary data, and $\boldsymbol{u}_0:\bar{\Omega}\to\mathbb{R}^d$ is the given initial velocity field.   The real-valued coefficient $\nu > 0$ is the kinematic viscosity, which can be derived from a Reynolds number $\textrm{Re} = UL/\nu$ for some velocity scale $U>0$ and length scale $L>0$.  In our later numerical experiments, we set $U=1$ and $L=1$.  The operator $\nabla^s$ is the symmetric part of the gradient operator, i.e., $\nabla^s \boldsymbol{u} = \frac{\nabla \boldsymbol{u} + \left(\nabla \boldsymbol{u}\right)^T}{2}$. Finite element and isogeometric methods are typically derived from the weak or variational form of the problem,
\begin{equation}
(\mathcal{W}): 
\begin{cases}
\textrm{Find} \ [\boldsymbol{u}, p] \in \pmb{\mathcal{U}} \times \mathcal{Q}, \ \ \textrm{s.t.} \ \textrm{that} \ \forall[\boldsymbol{v}, q] \in \pmb{\mathcal{V}} \times \mathcal{Q},\\[7.5pt]
\left(\frac{\displaystyle \partial \boldsymbol{u}} {\displaystyle \partial t}, \boldsymbol{v}\right) + A(\boldsymbol{u},\boldsymbol{v})+ C(\boldsymbol{u};\boldsymbol{u},\boldsymbol{v}) + B(\boldsymbol{u},q) - B(\boldsymbol{v},p) = L(\boldsymbol{v}) \\
\end{cases}
\label{weakform}
\end{equation}
where $\pmb{\mathcal{V}} = \mathbf{H}^1_0(\Omega)$ and $\mathcal{Q} = L^2_0(\Omega)$ are velocity and pressure test function spaces, $\pmb{\mathcal{U}}$ is a velocity trial function space satisfying $\boldsymbol{u} = \boldsymbol{u}_D$ on $\partial\Omega$ (in a trace sense), and the variational forms $A$, $B$, $C$, and $L$ are defined by
\begin{align}
    A(\boldsymbol{u},\boldsymbol{v}) &= \left(2\nu\nabla^s\boldsymbol{u},\nabla^s\boldsymbol{v}\right)\text{ ,}\\
    B(\boldsymbol{u},q) &= (\nabla\cdot\boldsymbol{u},q)\text{ ,}\\
    C(\boldsymbol{w};\boldsymbol{u},\boldsymbol{v}) &= -(\boldsymbol{w} \otimes \boldsymbol{u}, \nabla \boldsymbol{v})\text{ ,}\\
    L(\boldsymbol{v}) &= (\boldsymbol{f},\boldsymbol{v})\text{ .}
\end{align}

\subsection{The stabilized semi-discrete problem}\label{sec:semidisc}
The main subject of this paper is a stabilized spatial semi-discretization of the problem $(\mathcal{W})$:
\begin{equation}
(\mathcal{W}_h): 
\begin{cases}
\textrm{Find} \ [\boldsymbol{u}_h, p_h] \in \pmb{\mathcal{U}}_h \times \mathcal{Q}_h, \ \ \textrm{s.t.} \ \textrm{that} \ \forall[\boldsymbol{v}_h, q_h] \in \pmb{\mathcal{V}}_h \times \mathcal{Q}_h,\\[7.5pt]
\left(\frac{\displaystyle \partial \boldsymbol{u}_h} {\displaystyle \partial t}, \boldsymbol{v}_h\right) + A_h(\boldsymbol{u}_h,\boldsymbol{v}_h)+ C(\boldsymbol{u}_h;\boldsymbol{u}_h,\boldsymbol{v}_h) + B(\boldsymbol{u}_h,q_h) - B(\boldsymbol{v}_h,p_h) = L_h(\boldsymbol{v}_h) \\
\end{cases}
\label{eq:weakform-semidisc}
\end{equation}
where divergence-conforming B-splines are used for the discrete spaces $\pmb{\mathcal{U}}_h$, $\pmb{\mathcal{V}}_h$, and $\mathcal{Q}_h$.  Only the normal component of the velocity is enforced strongly in the discrete trial space, and the discrete test space is, correspondingly, a subset of $\mathbf{H}^1_{\boldsymbol{n}}$, rather than $\mathbf{H}^1_0$.  The forms $A_h$ and $L_h$ are modified versions of $A$ and $L$, augmented to include stabilization of advection and Nitsche-type weak enforcement of the tangential component of the Dirichlet boundary condition:
\begin{align}
    A_h(\boldsymbol{u},\boldsymbol{v}) &= A(\boldsymbol{u},\boldsymbol{v}) - (2\nu\boldsymbol{n}\cdot\nabla^s\boldsymbol{u},\boldsymbol{v})_{\partial\Omega} - (2\nu\boldsymbol{n}\cdot\nabla^s\boldsymbol{v},\boldsymbol{u})_{\partial\Omega} + \left(2\nu\frac{C_\text{Nit}}{h}\boldsymbol{u},\boldsymbol{v}\right)_{\partial\Omega} + J_h(\boldsymbol{u},\boldsymbol{v})\text{ ,}\\
    L_h(\boldsymbol{v}) &= L(\boldsymbol{v}) - (2\nu\boldsymbol{n}\cdot\nabla^s\boldsymbol{v},\boldsymbol{u}_D)_{\partial\Omega} + \left(2\nu\frac{C_\text{Nit}}{h}\boldsymbol{u}_D,\boldsymbol{v}\right)_{\partial\Omega}\text{ ,}
\end{align}
where the additional form $J_h$ introduces the skeleton stabilization of advection:
\begin{equation}	
J_h(\boldsymbol{u}_h,\boldsymbol{v}_h) \coloneqq \sum_{e \in \mathcal{E}_0} \left( \eta [\![\partial^{\alpha' + 1}_{\boldsymbol{n}} \boldsymbol{u}_h]\!], [\![ \partial^{\alpha' + 1}_{\boldsymbol{n}} \boldsymbol{v}_h]\!]\right)_e
\label{eq:jj}
\end{equation}
wherein
\begin{equation}
\partial_{\boldsymbol{n}}^{\alpha' + 1} \coloneqq \underbrace{\partial_{\boldsymbol{n}} \cdots \partial_{\boldsymbol{n}}}_{\text{$\alpha' + 1$ times}}
\label{nj}
\end{equation}
and $\partial_{\boldsymbol{n}} := \boldsymbol{n} \cdot \nabla$.  The dimensionless constant $C_\text{Nit}$ is associated with the Nitsche-type weak enforcement of the tangential Dirichlet boundary condition, which we select as $C_{\text{Nit}} = 5(k'+1)$ in this paper, following \cite{EV12}.  A choice for the stabilization parameter $\eta$ appearing in $J_h$ will be proposed based on dimensional analysis in Section \ref{sec:units}, where it will depend on element size and Reynolds number.  
In the lowest-order case of $k'=1$ with $\alpha' = 0$, we penalize jumps in first normal derivatives of velocity, and when $\alpha' > 0$, we penalize jumps in higher order normal derivatives.

\begin{remark}
The Nitsche-based enforcement of the tangential Dirichlet boundary condition can be viewed as stabilization for problems with boundary layers, as discussed further in \cite{BA0702,EV12}.
\end{remark}

\begin{remark}
The skeleton stabilization form $J_h$ reduces to a classical edge stabilization form when $\alpha' = 0$ \cite{Bu07}.
\end{remark}

\begin{remark}
For discrete velocity fields in $\pmb{\mathcal{U}}_h$, the component of velocity normal to interior mesh facets is $C^{\alpha'+1}$-continuous while the components of velocity tangential to interior mesh facets are $C^{\alpha'}$-continuous.  Consequently, the skeleton stabilization form $J_h$ only acts on the components of velocity tangential to interior mesh facets, and we can write
\begin{equation}	
J_h(\boldsymbol{u}_h,\boldsymbol{v}_h) = \sum_{e \in \mathcal{E}_0} \eta\left( [\![\partial^{\alpha' + 1}_{\boldsymbol{n}} \left( \boldsymbol{n} \times \boldsymbol{u}_h \times \boldsymbol{n} \right)]\!], [\![ \partial^{\alpha' + 1}_{\boldsymbol{n}} \left( \boldsymbol{n} \times \boldsymbol{v}_h \times \boldsymbol{n} \right) ]\!]\right)_e\text{ .}
\end{equation}
This property does not hold when edge stabilization is applied to divergence-conforming Scott--Vogelius finite element approximations of the incompressible Navier--Stokes problem \cite{BU08}, but it does hold for upwinding strategies used in $\textup{\textbf{H}}(\textup{div})$-conforming discontinuous Galerkin methodologies \cite{schroeder2018divergence}.
\end{remark}

\begin{remark}
It can be shown that
\begin{equation}	
J_h(\boldsymbol{u}_h,\boldsymbol{v}_h) = \sum_{e \in \mathcal{E}_0} \eta_{\boldsymbol{u}_h} \left( [\![\partial^{\alpha' + 1}_{\boldsymbol{u}_h} \boldsymbol{u}_h ]\!], [\![ \partial^{\alpha' + 1}_{\boldsymbol{u}_h} \boldsymbol{v}_h ]\!]\right)_e
\end{equation}
where $\eta_{\boldsymbol{u}_h} = \eta/|\boldsymbol{u}_h\cdot\boldsymbol{n}|^{2(\alpha'+1)}$,
\begin{equation}
\partial_{\boldsymbol{u}_h}^{\alpha' + 1} \coloneqq \underbrace{\partial_{\boldsymbol{u}_h} \cdots \partial_{\boldsymbol{u}_h}}_{\text{$\alpha' + 1$ times}}\text{ ,}
\end{equation}
and $\partial_{\boldsymbol{u}_h} := \boldsymbol{u}_h \cdot \nabla$.  Consequently, the skeleton stabilization form $J_h$ effectively penalizes jumps in convective derivatives.
When $\alpha' = 0$, the skeleton stabilization form $J_h$ further reduces to
\begin{equation}	
J_h(\boldsymbol{u}_h,\boldsymbol{v}_h) = \sum_{e \in \mathcal{E}_0} \eta_{\boldsymbol{u}_h} \left( [\![\partial_{\boldsymbol{u}_h} \boldsymbol{u}_h \times \boldsymbol{n} ]\!], [\![ \partial_{\boldsymbol{u}_h} \boldsymbol{v}_h \times \boldsymbol{n} ]\!]\right)_e\text{ .}
\end{equation}
This is precisely the jump penalty term appearing in a recent stabilization methodology for divergence-conforming discretizations of the incompressible Navier--Stokes problem based on least squares stabilization of the vorticity equation up to the choice of the stabilization parameter $\eta$ \cite{ahmed2021pressure}.
\end{remark}

\section{Analysis of the scheme}
\label{ANALYSIS}
We now study the properties of the semidiscrete problem ($\mathcal{W}_h$) analytically, to determine an appropriate choice of stabilization parameter (Section \ref{sec:units}) and prove several interesting properties of the method (Section \ref{sec:proof}).

\subsection{Choice of stabilization parameter}
\label{sec:units}
In this section, we motivate a particular choice of the stabilization parameter $\eta$ from \eqref{eq:jj}, using dimensional analysis. First, consider the units of the advection term,
\begin{equation}\label{eq:adv-units}
\left\lbrack (-\boldsymbol{u}_h \otimes \boldsymbol{u}_h , \nabla \boldsymbol{v}_h )\right\rbrack = \frac{\left\lbrack|\boldsymbol{u}_h|^2\right\rbrack}{L} L^d = \left\lbrack|\boldsymbol{u}_h|^2\right\rbrack L^{d - 1}\text{ ,}
\end{equation}
where square brackets are used to denote a ``dimensions of'' operator and $L$ denotes dimensions of length.  The stabilization term must have the same units, for a dimensionally-covariant formulation.  For the stabilization term, we have
\begin{equation}\label{eq:j-units}
    \left\lbrack J(\boldsymbol{u}_h,\boldsymbol{v}_h)\right\rbrack = \lbrack\eta\rbrack  \frac{\left\lbrack |\boldsymbol{u}_h|\right\rbrack}{L^{\alpha' +1}}\frac{1}{L^{\alpha' +1}} L^{d-1} = \lbrack\eta\rbrack \frac{L^{d-1}\left\lbrack |\boldsymbol{u}_h|\right\rbrack}{L^{2\alpha' + 2}}\text{ .}
\end{equation}
Comparing \eqref{eq:adv-units} and \eqref{eq:j-units}, we can determine the appropriate units for $\eta$ to be
\begin{equation}\label{eq:eta-units}
    \lbrack\eta\rbrack = L^{2\alpha'+2}\left\lbrack\vert\boldsymbol{u}_h\vert\right\rbrack\text{ .}
\end{equation}
Other criteria on $\eta$ include that we want it to provide less stabilization at lower Reynolds numbers and avoid excessive crosswind diffusion.  In light of \eqref{eq:eta-units} and these additional considerations, we propose to take
\begin{equation}
\eta \coloneqq  \gamma \text{min}\{\text{Re}_h, 1.0\} h^{2\alpha' + 2} |\boldsymbol{u}_h \cdot \boldsymbol{n}|\text{ ,}
\label{Rea1} 
\end{equation}
where the mesh size $h$ is taken as a length scale, and $\text{Re}_h$ is the element Reynolds number
\begin{equation}
\text{Re}_h = \frac{|\boldsymbol{u}_h| h}{\nu}\text{ ,}
\end{equation}
and $\gamma\geq 0$ is a dimensionless parameter.  In the numerical experiments of this work, we relate $\gamma$ to $\alpha'$ by
\begin{equation}
\gamma = \delta\times 10^{-(\alpha' + 2)}\text{ ,}
\end{equation}
which we find to be effective in practice with $\delta=1$, although better results can be obtained in certain problems by choosing moderately higher values of $\delta$.

\begin{remark}
The stabilization parameter $\eta$ we propose differs considerably from those commonly used in edge stabilization (see, e.g., \cite{BU08}).  First of all, the velocity scaling $| \boldsymbol{u} \cdot \boldsymbol{n}|$ appearing in our stabilization parameter is typically taken to be  $|\boldsymbol{u}|$ instead.  However, we have found this results in excessive numerical dissipation.  As our stabilization parameter scales with $| \boldsymbol{u} \cdot \boldsymbol{n}|$, it is zero when the velocity field is tangent to a mesh facet, just like upwinding strategies used in $\textup{\textbf{H}}(\textup{div})$-conforming discontinuous Galerkin methodologies \cite{schroeder2018divergence}.  Second of all, edge stabilization schemes typically do not depend on element Reynolds number, but we have also found this results in excessive numerical dissipation, especially in the diffusion-dominated limit.  Our form of the stabilization parameter $\eta$ is motivated by standard choices of the stabilization parameter $\tau_M$ in the Streamline Upwind Petrov--Galerkin (SUPG) method \cite{BR82}.  In particular, $\tau_M$ typically scales like $h/|\boldsymbol{u}_h|$ in the advection-dominated limit and $h^2/\nu$ in the diffusion-dominated limit.  In fact, $\tau_M$ must scale like $h^2/\nu$ in the diffusion-dominated limit in order for the SUPG method to maintain coercivity \cite{hughes2018multiscale}.  This is not the case, however, for the skeleton stabilization scheme considered in this paper.
\end{remark}

\subsection{Properties of the scheme}
\label{sec:proof}
For simplicity, our analysis assumes homogeneous Dirichlet boundary conditions, such that $\pmb{\mathcal{U}}_h = \pmb{\mathcal{V}}_h$. We also find it convenient to take advantage of the divergence-conforming discretization, rewriting the semidiscrete problem ($\mathcal{W}_h$) in the following equivalent form:
\begin{equation}
(\mathring{\mathcal{W}}_h): 
\begin{cases}
\textrm{Find} \ \boldsymbol{u}_h \in \mathring{\pmb{\mathcal{V}}}_h, \ \ \textrm{s.t.} \ \textrm{that} \ \forall\boldsymbol{v}_h \in \mathring{\pmb{\mathcal{V}}}_h,\\[7.5pt]
\left(\frac{\displaystyle \partial \boldsymbol{u}_h} {\displaystyle \partial t}, \boldsymbol{v}_h\right) + \mathscr{A}_\text{stb}(\boldsymbol{u}_h,\boldsymbol{v}_h) = L_h(\boldsymbol{v}_h) \\
\end{cases}
\end{equation}
where
\begin{equation}
    \mathscr{A}_\text{stb}(\boldsymbol{u}_h,\boldsymbol{v}_h) = A_h(\boldsymbol{u}_h,\boldsymbol{v}_h)+ C(\boldsymbol{u}_h;\boldsymbol{u}_h,\boldsymbol{v}_h)
\end{equation}
and the space $\mathring{\pmb{\mathcal{V}}}_h$ is the solenoidal subset of $\pmb{\mathcal{V}}_h$.  

Writing the problem in this form, we see immediately that this formulation satisfies an important property, following from the $\mathbf{L}^2$ orthogonality of solenoidal and irrotational vector fields.
\begin{thm}[Pressure robustness]\label{thm:p-robust}
Let $\boldsymbol{f} \in \mathbf{L}^2(\Omega)$ be the source term of the semidiscrete problem ($\mathring{\mathcal{W}}_h$).  If it is modified to $\boldsymbol{f} \to \boldsymbol{f} + \nabla \Phi$, where $\Phi \in H^1(\Omega)$, then the solution $\boldsymbol{u}_h$ will remain unchanged.
\end{thm}
\noindent The main practical ramification of this is that errors in approximating steep pressure gradients will not pollute the velocity solution.  This is of significance in, e.g., immersed boundary methods, where immersed boundaries are represented as large irrotational source terms \cite{peskin1993improved}.  Related issues have also been observed in thermal convection \cite{galvin2012stabilizing}.

Another obvious property of the semidiscrete problem is that, if the exact velocity solution $\boldsymbol{u}$ is sufficiently regular for the jump terms in $J_h$ to be zero and the Nitsche terms to be well-defined, then it will clearly satisfy the semidiscrete problem exactly.
\begin{lemma}[Consistency]\label{lem:consist}
Suppose $\boldsymbol{u}\in \mathbf{H}^{\alpha'+3/2}(\Omega)$.  Then
\begin{equation}
    \left(\frac{\displaystyle \partial \boldsymbol{u}} {\displaystyle \partial t}, \boldsymbol{v}_h\right) + \mathscr{A}_\text{stb}(\boldsymbol{u},\boldsymbol{v}_h) = L_h(\boldsymbol{v}_h)\quad\forall\boldsymbol{v}_h\in\mathring{\pmb{\mathcal{V}}}_h\text{ .}
\end{equation}
\end{lemma}

Next, we observe that $\mathscr{A}_\text{stb}$ satisfies a nonlinear coercivity property in the mesh-dependent norm
\begin{equation}
    |||\cdot|||^2   \coloneqq 2\nu (\nabla^s \cdot,\nabla^s \cdot) + 2\nu\frac{\displaystyle C_{\normalfont{\text{Nit}}}}{\displaystyle h} ||\cdot||^2_{\mathbf{L}^2(\mathcal{E}_{\partial})} +  || \eta^{1/2} [\![\partial^{\alpha' + 1}_{\boldsymbol{n}} (\cdot)]\!] ||^2_{\mathbf{L}^2(\mathcal{E}_0)} 
\end{equation}
defined on the discrete velocity space $\pmb{\mathcal{V}}_h$ if the Nitsche parameter is sufficiently large.
\begin{lemma}[Coercivity]\label{lem:coercivity} Provided that $C_{\normalfont{\text{Nit}}}$ is chosen sufficiently large, there exists a constant $C_\textup{coer} > 0$ such that
\begin{equation}
\mathscr{A}_{\textup{stb}} (\boldsymbol{u}_h,\boldsymbol{u}_h) \geq C_{\textup{coer}}|||\boldsymbol{u}_h|||^2\text{ .}
\end{equation}
\end{lemma}

\begin{proof}
The trilinear form $C$ is skew-symmetric in its second and third arguments, implying that $C(\boldsymbol{u}_h; \boldsymbol{u}_h, \boldsymbol{u}_h) = 0$.  Thus we have
\begin{align}
\begin{split}
\mathscr{A}_{\text{stb}} (\boldsymbol{u}_h,\boldsymbol{u}_h)
& = 2\nu (\nabla^s \boldsymbol{u}_h,\nabla^s \boldsymbol{u}_h) + 2\nu \frac{\displaystyle C_{\textrm{Nit}}}{\displaystyle h}||\boldsymbol{u}_h||^2_{\mathbf{L}^2(\mathcal{E}_{\partial})} - 4\nu(\boldsymbol{n} \cdot \nabla^s \boldsymbol{u}_h, \boldsymbol{u}_h)_{\Gamma_h} + || \eta^{1/2} [\![\partial^{\alpha' + 1}_{\boldsymbol{n}} \boldsymbol{u}_h]\!] ||^2_{\mathbf{L}^2(\mathcal{E}_0)}\\
& \geq 2\nu (\nabla^s \boldsymbol{u}_h,\nabla^s \boldsymbol{u}_h) +2\nu( \frac{\displaystyle C_{\textrm{Nit}}}{\displaystyle h} - \lambda) ||\boldsymbol{u}_h||^2_{\mathbf{L}^2(\mathcal{E}_{\partial})} - \displaystyle {\frac{2\nu} {\lambda}} ||\boldsymbol{n} \cdot \nabla^s \boldsymbol{u}_h||^2_{\mathbf{L}^2(\mathcal{E}_{\partial})} + || \eta^{1/2} [\![\partial^{\alpha + 1}_{\boldsymbol{n}} \boldsymbol{u}_h]\!] ||^2_{\mathbf{L}^2(\mathcal{E}_0)}\text{ ,}
\end{split}
\end{align}
where $\lambda > 0$ is an arbitrary constant associated with applying Young's inequality to the Nitsche consistency and symmetry terms.  Let $C_{\text{tr}} > 0$ be sufficiently large such that the trace inequality
\begin{equation}
||\boldsymbol{n} \cdot \nabla^s \boldsymbol{u}_h||^2_{\mathbf{L}^2(\mathcal{E}_{\partial})} \leq \frac{C_{\text{tr}}}{h} (\nabla^s \boldsymbol{u}_h,\nabla^s \boldsymbol{u}_h)
\end{equation}
holds \cite{Evans2013}.  Then, if $\lambda$ is chosen to be equal to $\frac{2C_{\text{tr}}}{h}$, the desired result immediately follows with $C_\text{coer} = \frac{1}{2}$ if $C_{\text{Nit}} \geq 4C_{\text{tr}}$.
\end{proof}

In a steady, linearized model problem, where time derivatives are zero and the first argument to the convection form $C$ is replaced by a given solenoidal vector field, Lemma \ref{lem:consist} would yield the error orthogonality property that $\mathscr{A}_\text{stb}(\boldsymbol{u}_h-\boldsymbol{u},\boldsymbol{v}_h) = 0$, which can be combined with Lemma \ref{lem:coercivity} to derive error estimates.  Analysis of this type can be found in literature on related methods, e.g., \cite{ahmed2021pressure,EV12,EV13,VJ17}.  In particular, these authors derive bounds on error that are robust with respect to Reynolds number.  However, the present contribution is focused on the nonlinear Navier--Stokes problem.  In the nonlinear setting, the coercivity property formalized in Lemma \ref{lem:coercivity} leads straightforwardly to energetic stability of the semidiscrete solution.
\begin{thm}[Energetic stability]  If $\boldsymbol{u}_h$ satisfies ($\mathring{\mathcal{W}}_h$) and $C_{\normalfont{\text{Nit}}}$ is sufficiently large,
\begin{equation}
\frac{\partial}{\partial t} \frac{1}{2} ||\boldsymbol{u}_h||^2_{\mathbf{L}^2(\mathcal{T}_h)} = - \mathscr{A}_{\text{stb}} (\boldsymbol{u}_h,\boldsymbol{u}_h) + (\boldsymbol{f}, \boldsymbol{u}_h) \leq (\boldsymbol{f}, \boldsymbol{u}_h)\text{ .}
\end{equation}
\end{thm}
\noindent Note, in particular, that, in the absence of external forcing or in the presence of conservative forcing, this implies that kinetic energy of the discrete solution will decay over time.  The same property holds for ordinary Galerkin approximations as well, but we shall demonstrate numerically, in Section \ref{sec:numerical}, that the extra numerical dissipation due to $J_h$ in fact provides robustness in high Reynolds number flows, as one might expect by analogy to the error analysis of similar methods applied to linear model problems.  As a final remark, the skeleton-stabilized methodology presented herein also admits global balance laws for linear momentum on rectilinear domains and axial angular momentum on cylindrical domains.  In particular, the global balance laws presented for ordinary Galerkin approximations in Section 7 of \cite{EV13} trivially extend to the method proposed in this paper.

\section{Numerical examples}\label{sec:numerical}
We now proceed to study the properties of the method empirically, using numerical experiments.  Our numerical experiments use the library tIGAr \cite{KA19}, which extends the FEniCS \cite{LO12} software for finite element automation to IGA.  FEniCS enables users to specify variational problems in a high-level domain-specific language called Unified Form Language (UFL) \cite{Alnaes2014}.  UFL problem descriptions are then compiled \cite{Kirby2006} into efficient numerical routines.  For unsteady problems, the fully-discrete formulation is obtained from the semidiscrete problem ($\mathcal{W}_h$) by using the implicit generalized-$\alpha$ \cite{Hul93} finite difference method to discretize $\partial/\partial t$.  The saddle point problem at each time step is solved using the iterated penalty scheme of \cite{MS18}.  The formulation for steady problems is a straightforward simplification of ($\mathcal{W}_h$).  Because we use Nitsche's method to enforce tangential Dirichlet boundary conditions on the velocity, there is no need for refinement around domain boundaries (cf. \cite{BA0702}), and uniform knot spacing is used to define the div-conforming B-spline discretizations for all examples in this paper.  Maximal continuity splines are used as well, so $\alpha' = k - 1$.

\label{NE}
\subsection{Steady manufactured solution}
\label{mmf}
We begin our numerical experiments with some basic tests using a problem with the known exact solution
\begin{align}
\begin{split}
\begin{cases}
u_1(\boldsymbol{x}) &= 2e^x_1(x_1-1)^2x_1^2(x_2^2-x_2)(2x_2-1)\\
u_2(\boldsymbol{x}) &= -e^x_1(x_1-1)x_1(x_1^2+3x_1-2)(x_2-1)^2x_2^2\\
p(\boldsymbol{x}) &= (-424 + 156e + (x_2^2-x_2)(-456+e^x_1(456 + x_1^2(228-5(x_2^2-y)) \\
& + 2x_1(-228 + (x_2^2-x_2)) + 2x_1^3(-36 + (x_2^2-x_2)) + x_1^4(12+ x_2^2-x_2))))
\end{cases}
\end{split}
\label{p(x)}
\end{align}
on the unit square $\Omega = [0,1]\times[0,1]$, as used previously in \cite{BA11,EV12}.  This solution can be manufactured by deriving a corresponding source term $\boldsymbol{f}$ that is consistent with the strong form of the Navier--Stokes equations.  (We do this automatically, using computer algebra in FEniCS UFL.)

\subsubsection{Convergence}

We first test the method's convergence with respect to $h$-refinement for different polynomial degrees, at a fixed Reynolds number of 10.  The results are collected in Table \ref{tab:conv}.  These results are consistent with optimal convergence rates in both $\mathbf{L}^2$ and $\mathbf{H}^1$, for all polynomial degrees tested.  The high-order convergence for $k' > 1$ is attributable to the consistency property of Lemma \ref{lem:consist}, whose hypotheses are satisfied by the smooth exact solution used here.
\begin{table}[b!]
\begin{center}
\begin{tabular}{SSSSSSSS} \toprule
 {} &  {$h$} & {$\displaystyle{\frac{1}{4}}$} & {$\displaystyle{\frac{1}{8}}$} & {$\displaystyle{\frac{1}{16}}$} & {$\displaystyle{\frac{1}{32}}$} & {$\displaystyle{\frac{1}{64}}$} & {$\displaystyle{\frac{1}{128}}$} \\ \midrule
  {} &  {$||\boldsymbol{u}-\boldsymbol{u}_h||_{\mathbf{L}^2}$}  & {4.110e\text{-}3} & {1.048e\text{-}3} & {2.629e\text{-}4}  & {6.579e\text{-}5} & {1.645e\text{-}5} & {4.113e\text{-}6}  \\
  {$k'=1$} &  {order}  & {$-$}  & 2.784 & 2.579  & 2.462 & 2.384 &  2.330   \\
  {$\gamma = 1\text{e-}2$} &  {$|\boldsymbol{u}-\boldsymbol{u}_h|_{\mathbf{H}^1}$}  & {5.546e\text{-}2}  & {2.788e\text{-}2} & {1.395e\text{-}2}  & {6.978e\text{-}3}  & {3.489e\text{-}3} & {1.745e\text{-}3}  \\
  {} &  {order}  & {$-$}  & 1.727 & 1.542  & 1.433 & 1.361  & 1.309       \\ \midrule

    {} &  {$||\boldsymbol{u}-\boldsymbol{u}_h||_{\mathbf{L}^2}$}  &  {3.873e\text{-}4} & {4.444e\text{-}5} & {5.396e\text{-}6} &  {6.691e\text{-}7} & {8.346e\text{-}8}  & {1.043e\text{-}8}  \\
  {$k'=2$} &  {order}  & {$-$}  & 3.837 & 3.661  & 3.538 & 3.451  & 3.387    \\
  {$\gamma = 1\text{e-}3$} &  {$|\boldsymbol{u}-\boldsymbol{u}_h|_{\mathbf{H}^1}$}  & {9.237e\text{-}3}  & {2.244e\text{-}3} & {5.556e\text{-}4}  & {1.385e\text{-}4} & {3.460e\text{-}5} & {8.649\text{e}-6}  \\
  {} &  {order}  & {$-$}  & 2.387 & 2.303  & 2.245 &  2.205  & 2.176       \\ \midrule
  
    {} &  {$||\boldsymbol{u}-\boldsymbol{u}_h||_{\mathbf{L}^2}$}  &   {3.281e\text{-}5} & {2.354e\text{-}6} & {1.586e\text{-}7} &  {1.027e\text{-}8} & {6.534e\text{-}10}  & {4.119e\text{-}11}  \\
  {$k'=3$} &  {order}  & {$-$}  & 5.001 & 4.699  & 4.537  &  4.439 & 4.372    \\
  {$\gamma = 1\text{e-}4$} &  {$|\boldsymbol{u}-\boldsymbol{u}_h|_{\mathbf{H}^1}$}  & {9.096e\text{-}4}  & {1.228e\text{-}4} & {1.619e\text{-}5}  & {2.085e\text{-}6} & {2.648e\text{-}7} & {3.336e\text{-}8}   \\
  {} &  {order}  & {$-$} & 3.437 & 3.299  & 3.223 & 3.178  &  3.149      \\ \midrule
      \\ \bottomrule
\end{tabular}
\end{center}
 \caption{Convergence of results to a smooth manufactured solution of the steady Navier--Stokes problem.}\label{tab:conv}
 \end{table}

\subsubsection{Reynolds number robustness}
Our chosen manufactured solution is independent of Reynolds number, which makes it useful for comparing errors at different Reynolds numbers.  Ideally, the skeleton stabilization introduced by the form $J_h$ should render error insensitive to Reynolds number.  To verify this, we now consider fix the mesh size at $h=1/16$ and compute discrete solutions at Reynolds numbers of $10^i$ for $i\in\{0,1, 2,3\}$.  The results in Figure \ref{fig:reynolds-robust} show that, as hoped, $\mathbf{L}^2$ and $\mathbf{H}^1$ errors are essentially independent of Reynolds number, for $k' = 1,2,3$.
\begin{figure}[t!]
\centering
   	 \begin{subfigure}[b]{0.32\textwidth}
   	 	\centering
		\includegraphics[height=2in]{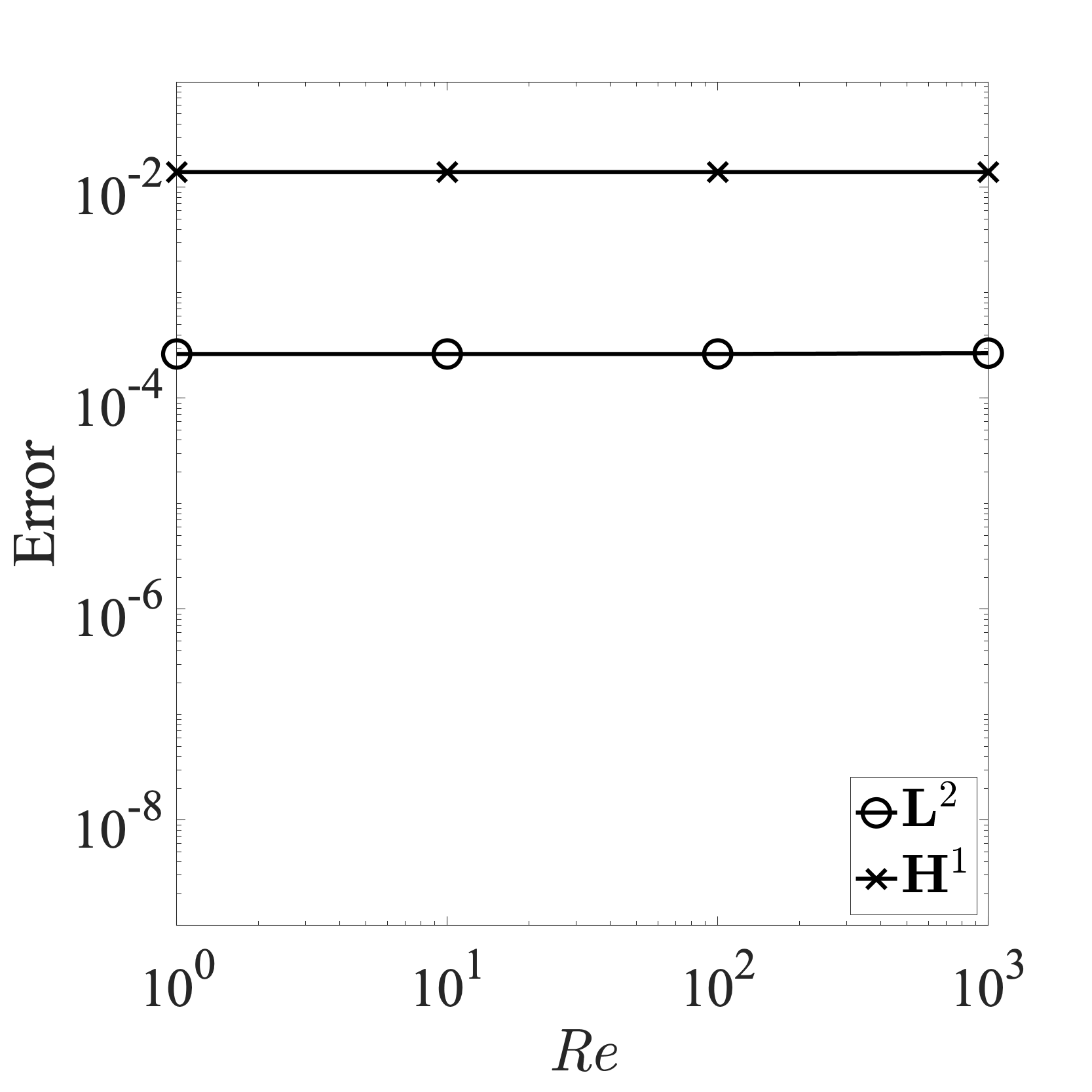}
		\caption{$k' = 1$}
	\end{subfigure}
	\begin{subfigure}[b]{0.32\textwidth}
   	 	\centering
		\includegraphics[height=2in]{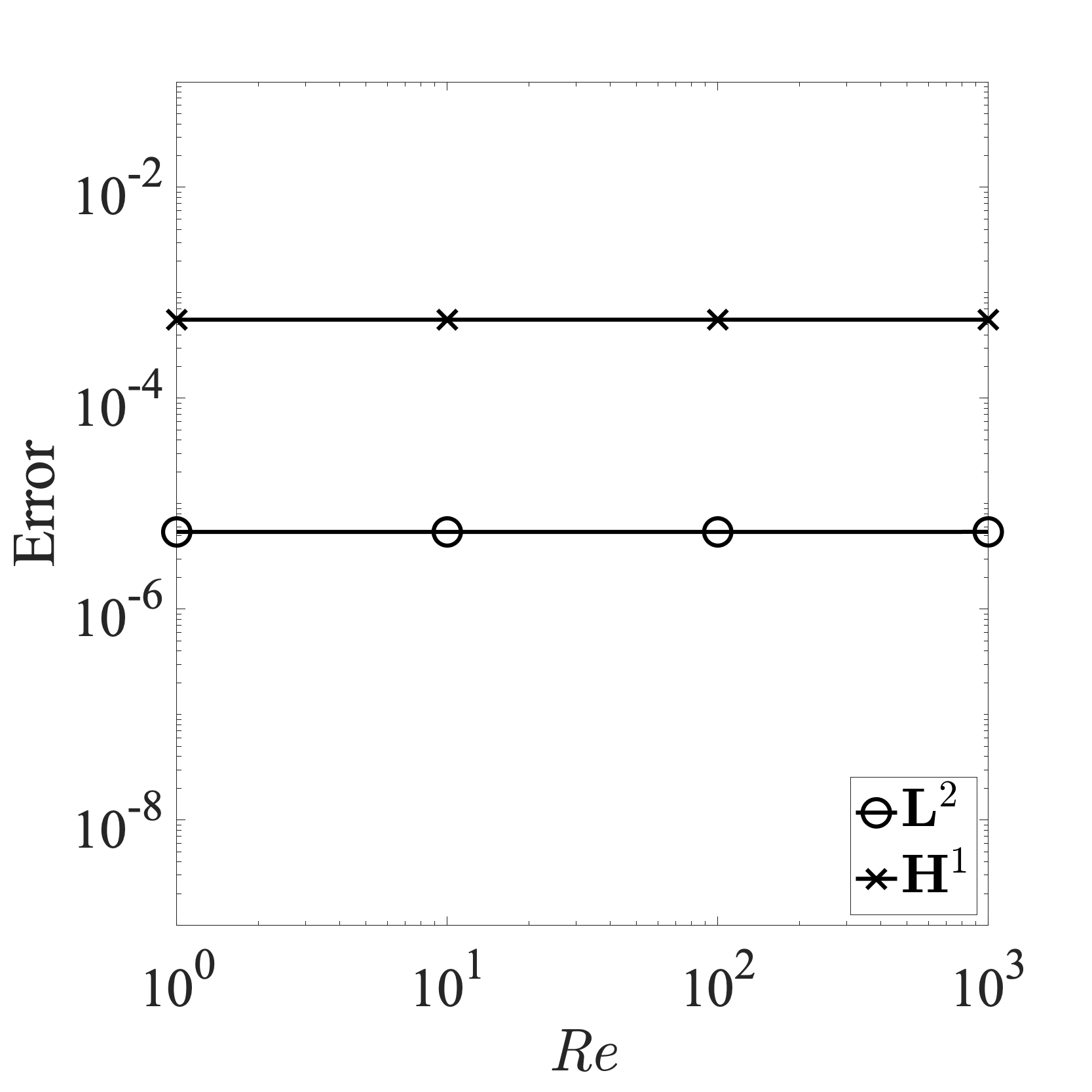}
		\caption{$k' = 2$}
	\end{subfigure}
	 \begin{subfigure}[b]{0.32\textwidth}
   	 	\centering
		\includegraphics[height=2in]{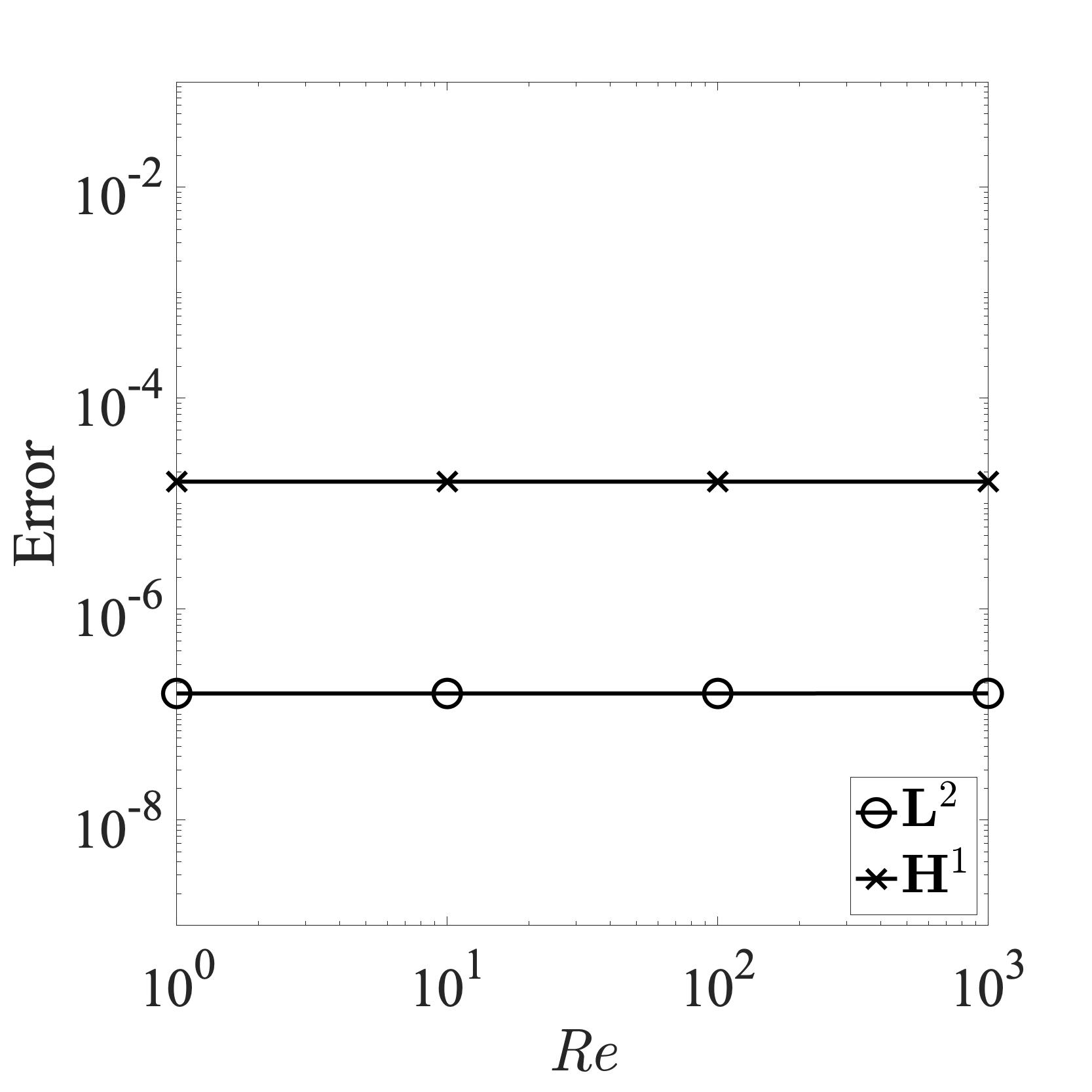}
		\caption{$k' = 3$}
	\end{subfigure}
	\caption{Numerical testing of Reynolds number robustness at $h=1/16$.}\label{fig:reynolds-robust}
\end{figure}

\subsubsection{Pressure robustness}
We now test Theorem \ref{thm:p-robust} numerically.  In particular, we modify the source term $\boldsymbol{f}$ used to manufacture the exact solution, by adding an irrotational component:
\begin{equation}
    \boldsymbol{f}\to \boldsymbol{f} + \nabla\Phi\text{ ,}
\end{equation}
where we choose $\Phi(\boldsymbol{x}) = \text{sin}(\pi x_1 x_2)$.  This modifies the pressure solution, $p$, but we are primarily interested in velocity error.  Table \ref{tab:p-robust} compares the velocity errors for the original and modified source terms, using a mesh with $h=1/16$, a Reynolds number of 10, and divergence-conforming B-spline spaces of degree $k'=1$.  The results show that discrepancies in velocity errors (highlighted in red) are small enough to be reasonably attributed to solver tolerance and accumulation of floating point round-off error.

\begin{table}[t!]
\begin{center}
\begin{tabular}{SSS} \toprule
{} &  {$\boldsymbol{f}$}  &  {$\boldsymbol{f} + \nabla \Phi$}  \\ \midrule 
{$||\boldsymbol{u}-\boldsymbol{u}_h||_{\mathbf{L}^2}$}  &  {0.000262925982{\color {red}6760857}} &{0.000262925982{\color {red} 7626544}}  \\ \midrule
{$|\boldsymbol{u}-\boldsymbol{u}_h|_{\mathbf{H}^1}$}  &  {0.0139525191245439{\color {red}72}} &{0.0139525191245439{\color {red}45}}  \\ \bottomrule
\end{tabular}
\end{center}
\caption{Verification of pressure robustness.}\label{tab:p-robust}
\end{table}

\subsection{Lid-driven cavity problem}
\label{lldd}
We now move on to a more complicated problem without a known closed-form solution, namely, the lid-driven cavity benchmark.  This problem serves as a test of how robust discretizations are in the presence of solution singularities, and has been studied extensively in the literature \cite{BU08,EV12,EV13,GHIA}.  The most common 2D variant of this problem is posed on a unit square, with Dirichlet boundary data $\boldsymbol{u} = [u_D,0]^T$ on the top and no-slip conditions on the remaining sides.  (Thus, the Dirichlet data is discontinuous at the upper corners, and cannot be the trace of an $\mathbf{H}^1$ solution, placing the problem technically outside the scope of the weak problem ($\mathcal{W}$) and the consistency result of Lemma \ref{lem:consist}, which requires even more regularity.)  We refer to the top as $\Gamma_{\textrm{lid}}$ and the union of other sides as $\Gamma_{\textrm{wall}}$.  We focus on the steady state of flow in this configuration, where the solution is characterized by a primary center vortex with a few secondary corner or near-wall vortices, known as Moffatt eddies \cite{mo64}. We compute approximate solutions at $\text{Re} = 7500$ and $\text{Re}=10000$ and compare our results with those of Ghia et al. \cite{GHIA}, considering both qualitative flow behavior (Section \ref{sec:ldc-qualitative}) and quantitative solution data (Section \ref{sec:ldc-quantitative}).

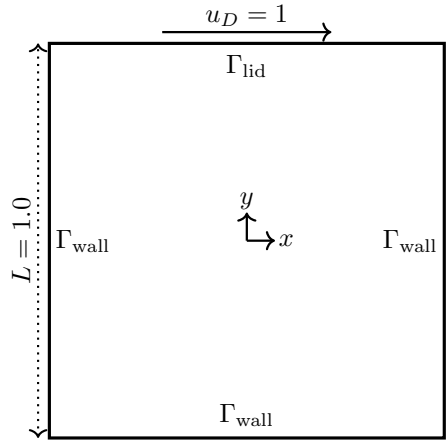
\begin{figure}[t!]
 \begin{center}
\begin{tikzpicture}[thick, scale=0.75]
\draw [->][line width=0.3mm, black] (3.5,3.5) -- (3.5,4) ;
\node at (3.5,4.2) {$y$};
\draw [->][line width=0.3mm, black] (3.5,3.5) -- (4,3.5) ;
\node at(4.2,3.5) {$x$};
\draw[black, very thick] (0,0) rectangle (7,7);
\draw [<->][line width=0.5mm, black, thick, dotted ] (-0.2,0) -- (-0.2,7) ;
\node [rotate=90] at (-.5,3.5) {$L = 1.0$};
\node at (3.5,7.5) {$u_D = 1$};
\draw [->][line width=0.3mm, black]  (2.0,7.2)--  (5.0,7.2) ;
\node [align=left] at (.6,3.5) {$\Gamma_{\text{wall}}$};
\node [align=left] at (6.4,3.5) {$\Gamma_{\text{wall}}$};
\node [align=left] at (3.5,0.4) {$\Gamma_{\text{wall}}$};
\node [align=left] at (3.5,6.6) {$\Gamma_{\text{lid}}$};
\end{tikzpicture}
\end{center}
  	\caption{Sketch of the lid-driven cavity problem.}
  		\label{lid}
\end{figure}
\subsubsection{Near-wall eddies}\label{sec:ldc-qualitative}
The variational multiscale (VMS) concept has clarified, over the past few decades, the connection between stabilization and modeling the effects of solution features that are too small to be resolved by finite meshes.  In particular, the connection between skeleton stabilization and VMS was discussed in \cite{Burman2007}.  Without any form of stabilization, unresolved fine-scale solution features can lead to spurious qualitative features at resolved scales.  In the lid-driven cavity problem, the near-wall Moffatt eddies are much smaller than the overall size of the domain, and serve as an effective test of how well stabilization schemes can model their influence on the coarse-scale solution. 
\begin{figure}[t!]
\centering
   	 \begin{subfigure}[b]{0.49\textwidth}
   	 	\centering
		\includegraphics[height=2.5in]{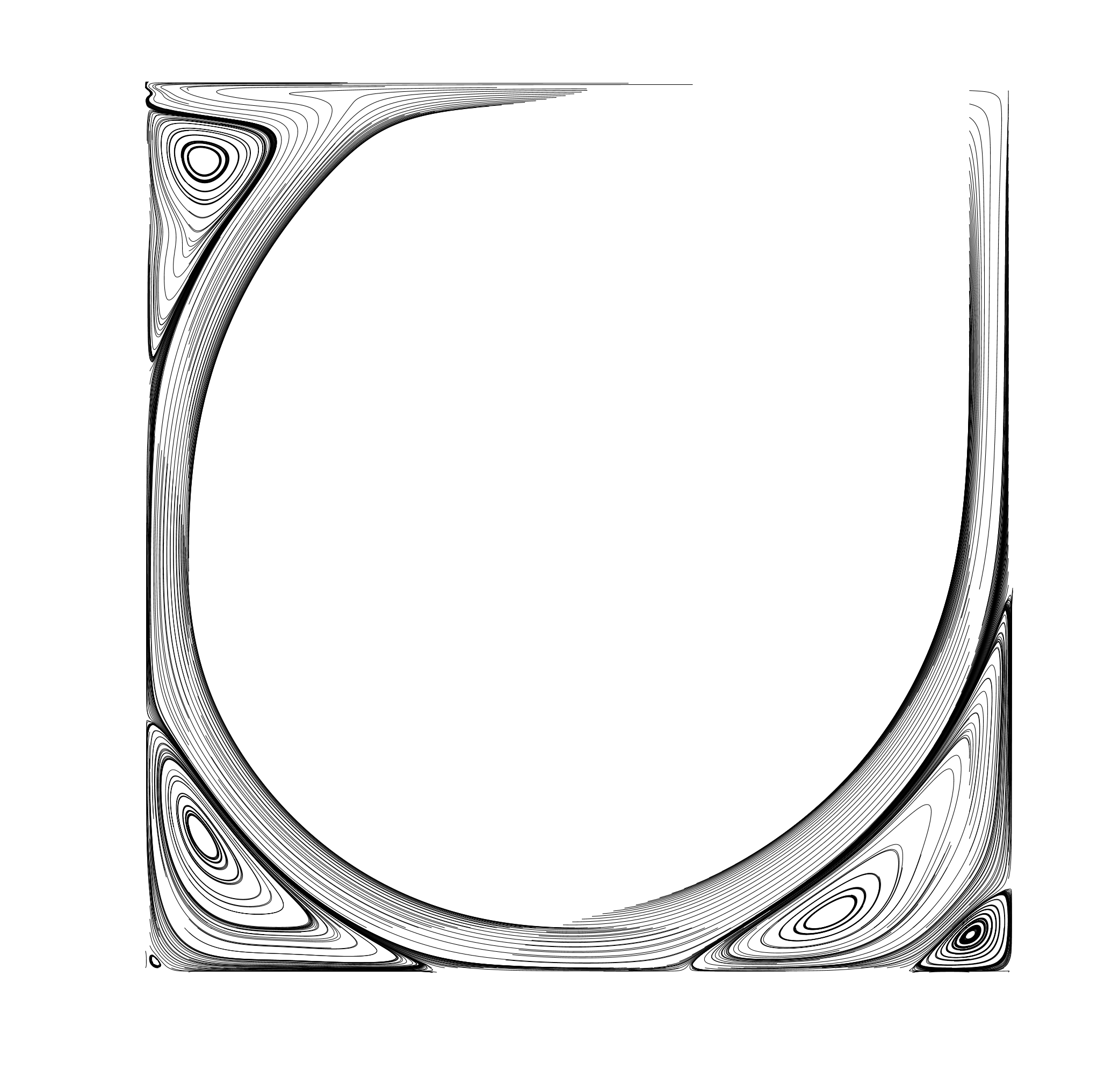}
		\caption{$\text{Re} = 7500$}
	\end{subfigure}
	 \begin{subfigure}[b]{0.49\textwidth}
   	 	\centering
		\includegraphics[height=2.5in]{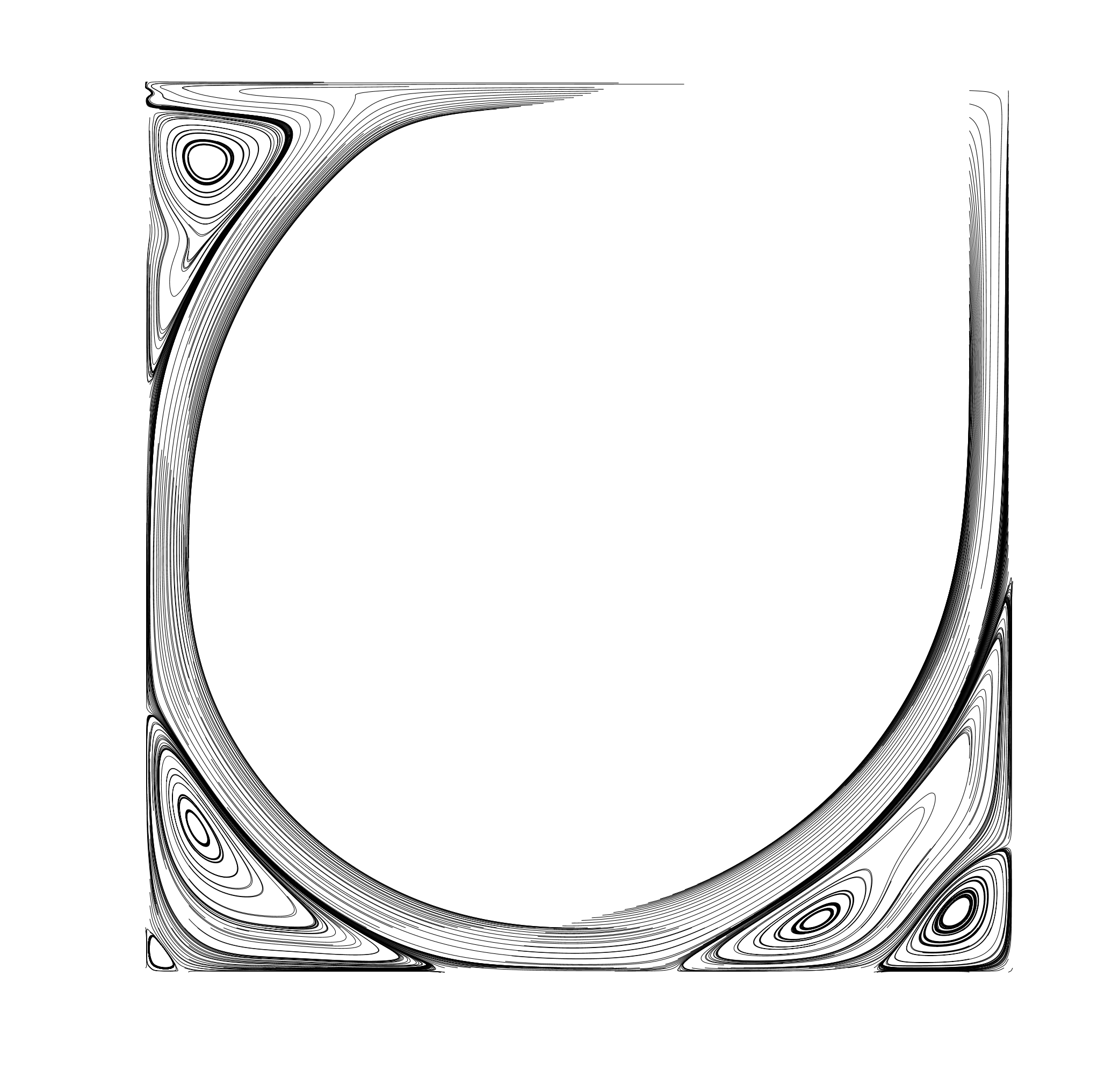}
		\caption{$\text{Re} = 10000$}
	\end{subfigure}
	\caption{Near-wall streamline contours of a well-resolved solution for the lid-driven cavity problem.}\label{fig:resolved-solutions-ldc}
\end{figure}
For reference, we first compute a well-resolved solution, to illustrate the qualitative behavior of the exact solution.  Solutions at both Reynolds numbers of interest are computed on a $128\times128$ element mesh, with a skeleton-stabilized divergence-conforming B-spline discretization of degree $k'= 3 $.  Figure \ref{fig:resolved-solutions-ldc} shows streamlines traced from locations near the bottom two corners, to clearly illustrate the structure of the Moffatt eddies.  The number, location, and pattern of the Moffatt eddies seen in Figure \ref{fig:resolved-solutions-ldc} match the results of \cite{GHIA}. However, for meshes with insufficient resolution, solutions using $\gamma = 0$
(i.e., Galerkin's method on the interior and Nitsche-based enforcement of tangential Dirichlet boundary conditions),
referred to as  ``unstabilized'' here, develop spurious eddies around walls and corners, which do not even qualitatively resemble reference solutions.  Figures \ref{fig:unstab-7500}--\ref{fig:stab-10000} compare these unstabilized solutions with those of our proposed scheme on a very coarse $16\times 16$ element mesh, at Reynolds numbers of $7500$ and $10000$ and varying B-spline degrees.  While the coarse mesh is fundamentally incapable of resolving the smallest near-wall eddies, the stabilized solutions have a clearer qualitative resemblance to the well-resolved reference solutions in Figure \ref{fig:resolved-solutions-ldc}, while the unstabilized solutions develop strange, unphysical patterns of vortices that disrupt the overall qualitative flow pattern outside of the main central vortex.  

\begin{figure}[t!]
\centering
   	 \begin{subfigure}[b]{0.329\textwidth}
   	 	\centering
		\includegraphics[height=2.25in]{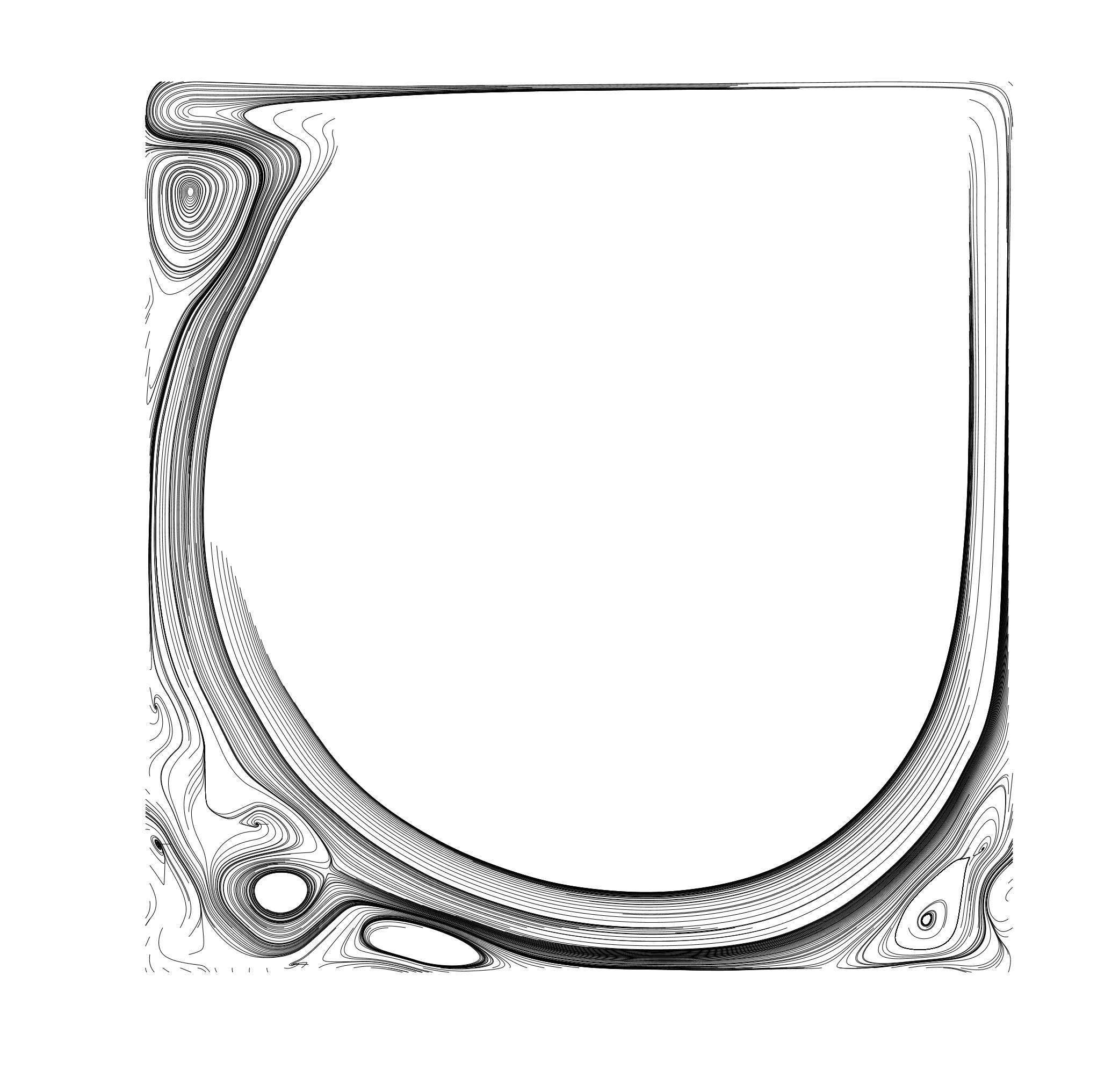}
		\caption{\centering $k'=1$}
	\end{subfigure}
	\begin{subfigure}[b]{0.329\textwidth}
   	 	\centering
		\includegraphics[height=2.25in]{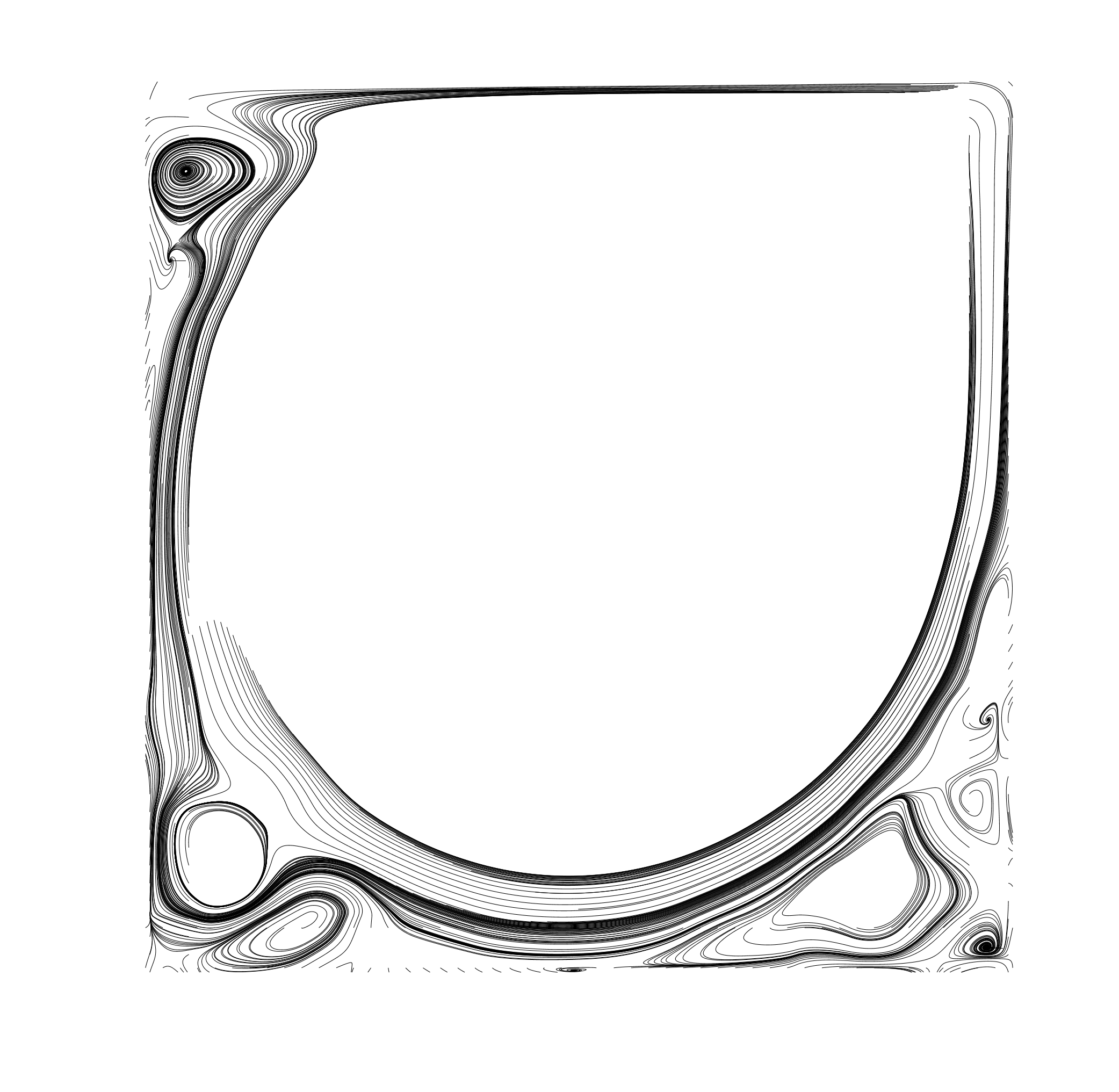}
		\caption{\centering $k'=2$}
	\end{subfigure}
\begin{subfigure}[b]{0.329\textwidth}
   	 	\centering
		\includegraphics[height=2.25in]{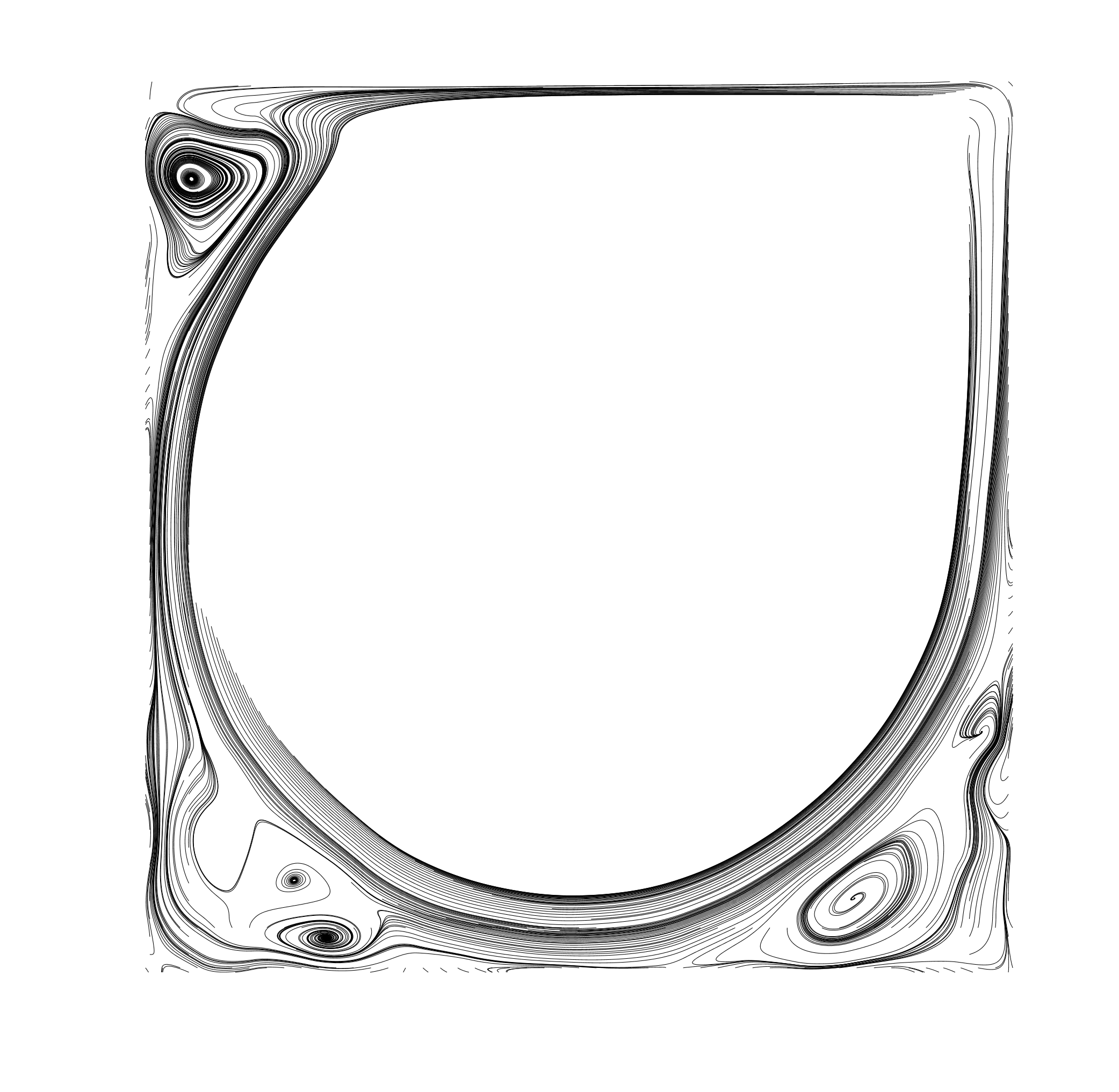}
		\caption{\centering  $k'=3$}
	\end{subfigure}
		\caption{Near wall streamline contours attained using a severely underresolved unstabilized divergence-conforming discretization on a $16\times16$ element mesh for the lid-driven cavity problem at $Re = 7500$.}\label{fig:unstab-7500}
\end{figure}

\begin{figure}[t!]
\centering
   	 \begin{subfigure}[b]{0.329\textwidth}
   	 	\centering
		\includegraphics[height=2.25in]{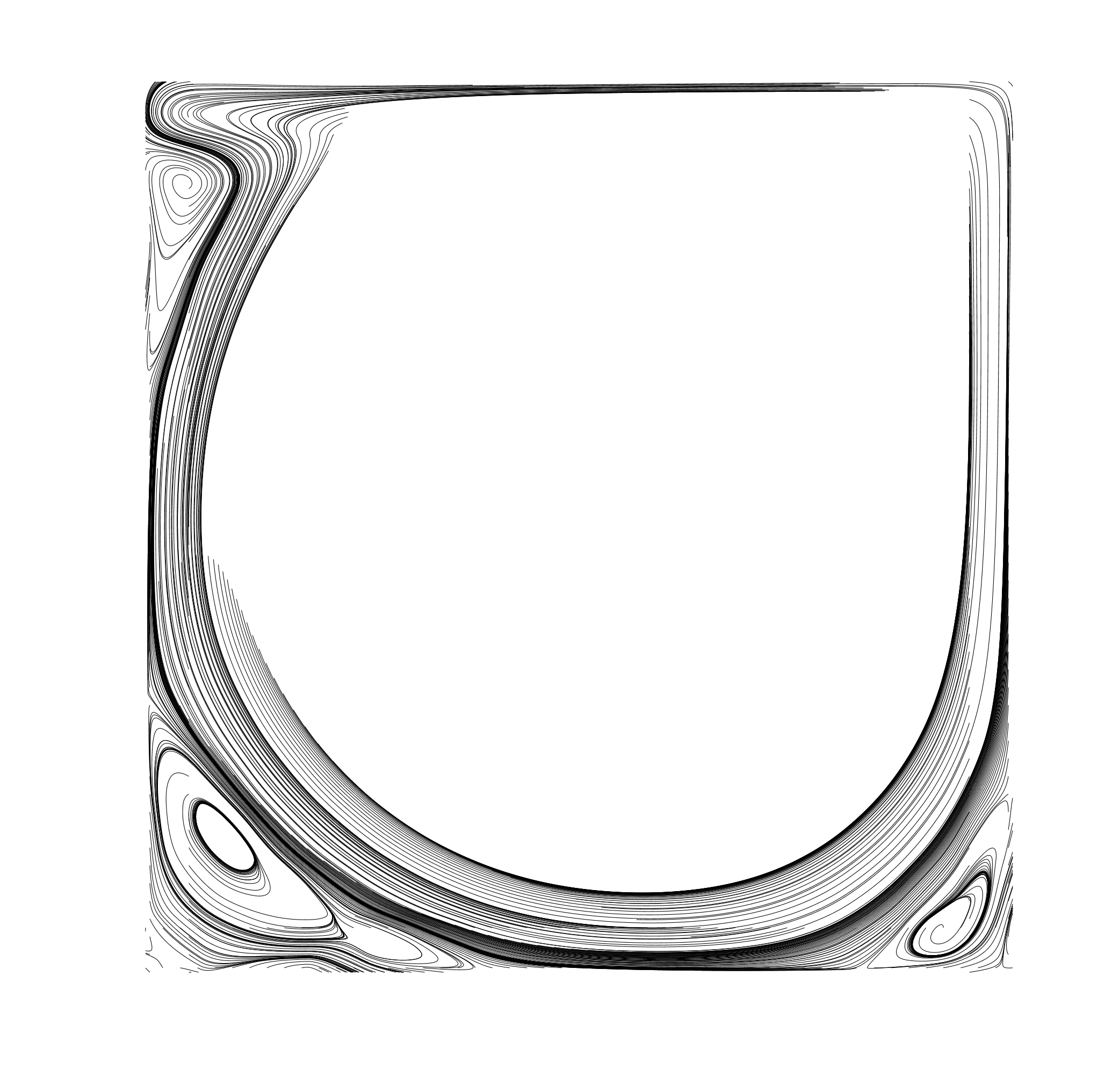}
		\caption{\centering $k'=1$}
	\end{subfigure}
	\begin{subfigure}[b]{0.329\textwidth}
   	 	\centering
		\includegraphics[height=2.25in]{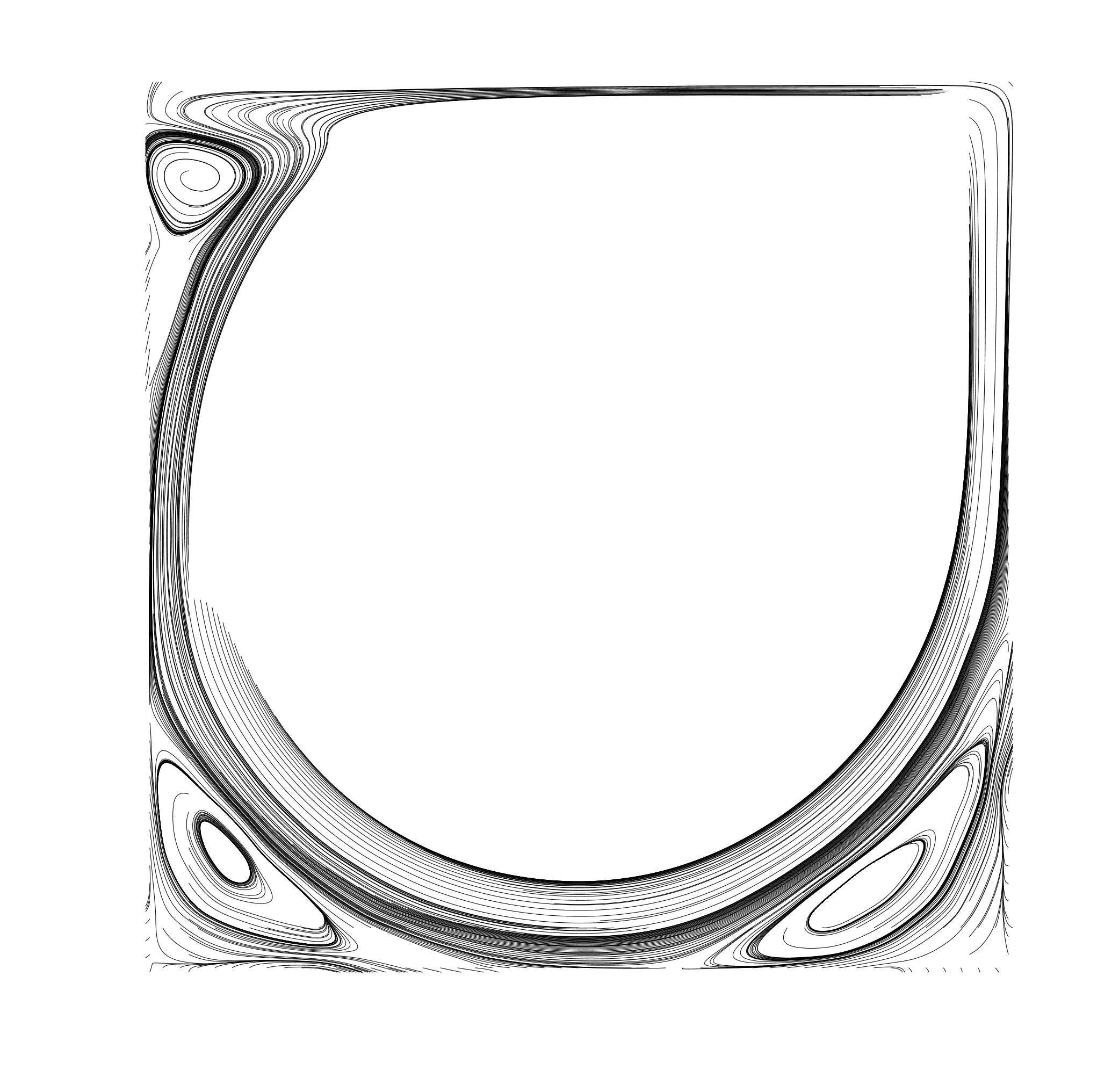}
		\caption{\centering $k'=2$}
	\end{subfigure}
\begin{subfigure}[b]{0.329\textwidth}
   	 	\centering
		\includegraphics[height=2.25in]{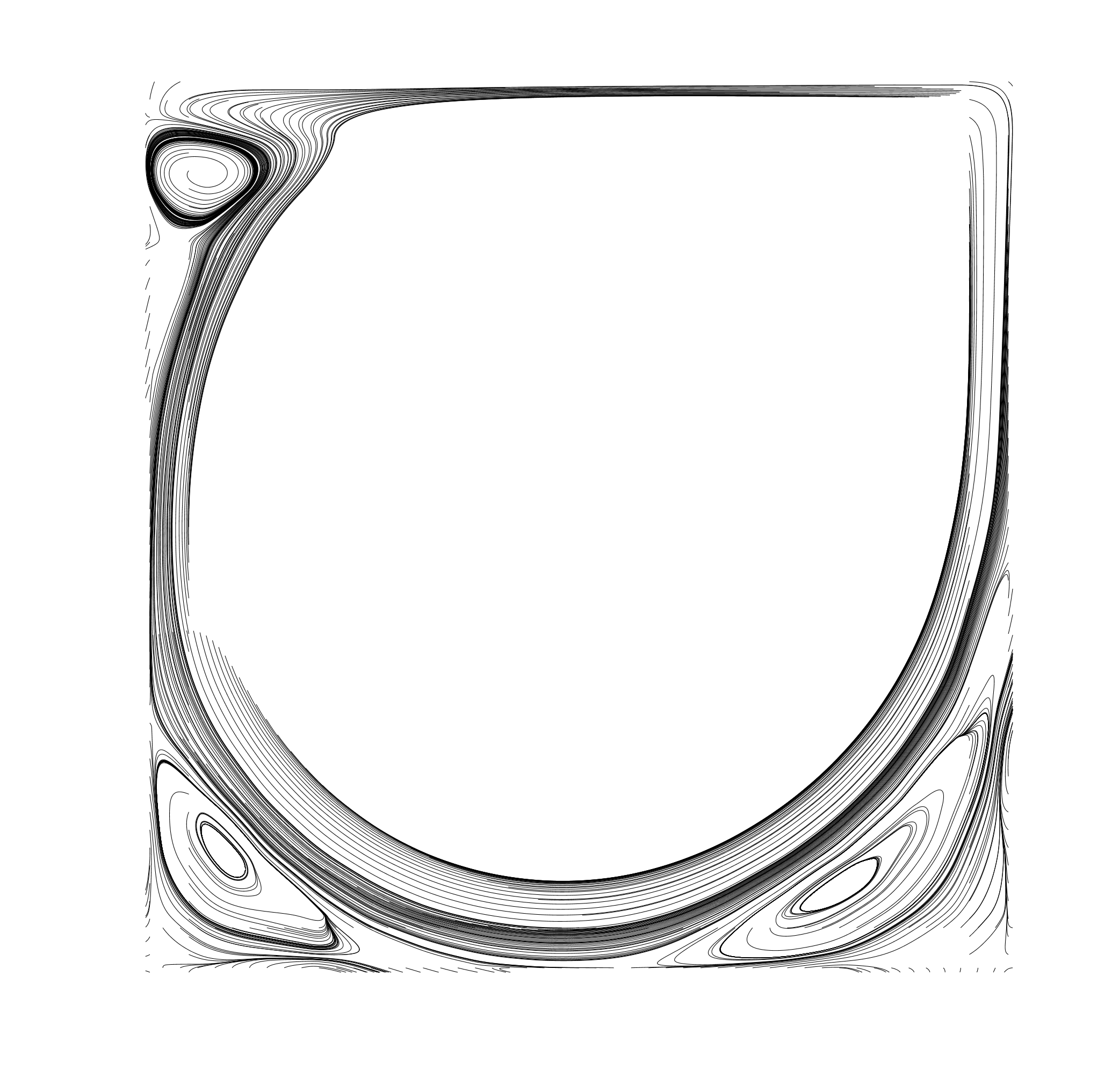}
		\caption{\centering  $k'=3$}
	\end{subfigure}
		\caption{Near wall streamline contours attained using a severely underresolved stabilized divergence-conforming discretization on a $16\times16$ element mesh for the lid-driven cavity problem at $Re = 7500$.}
\end{figure}

\begin{figure}[t!]
\centering
   	 \begin{subfigure}[b]{0.329\textwidth}
   	 	\centering
		\includegraphics[height=2.25in]{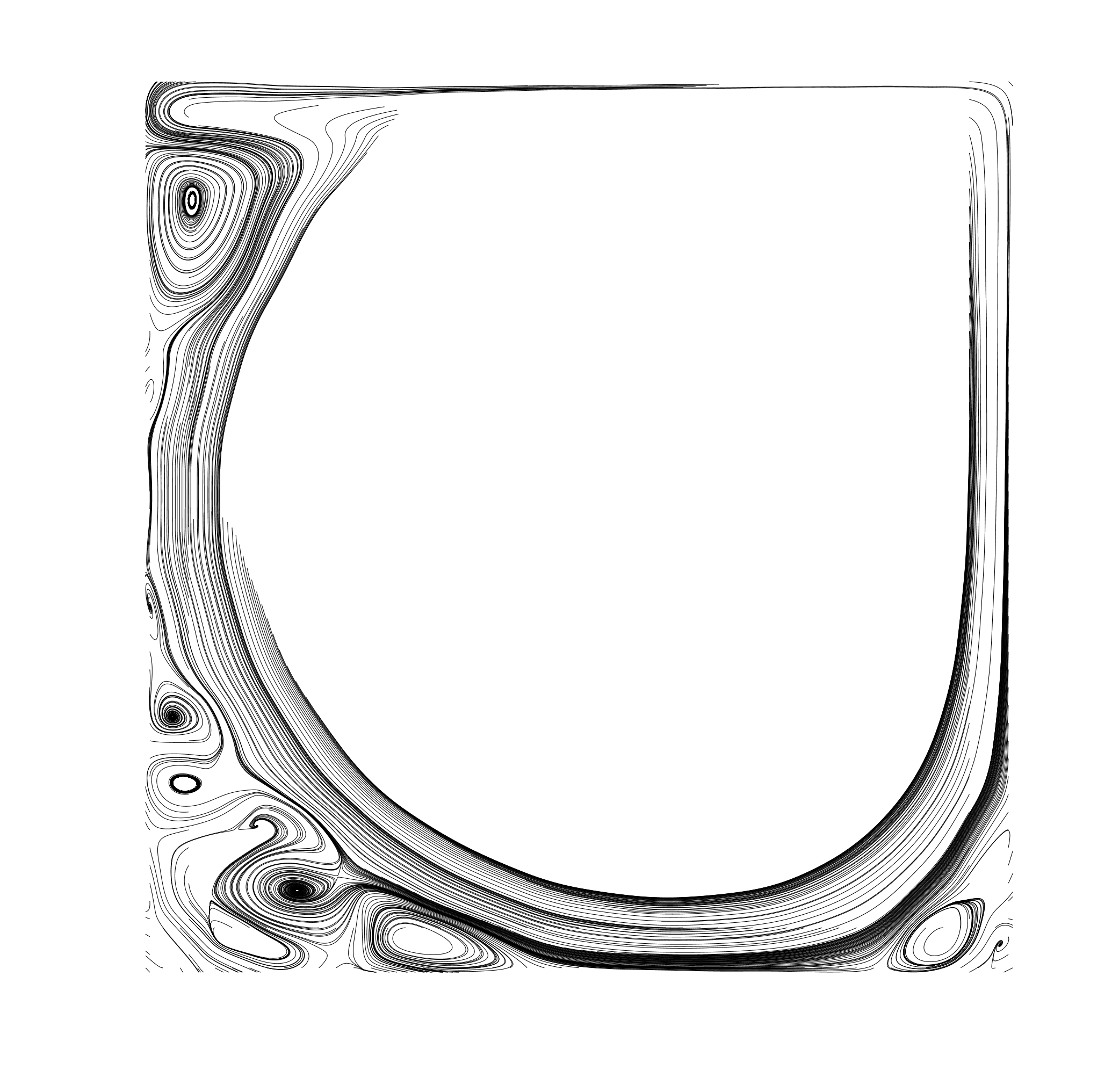}
		\caption{\centering $k'=1$}
	\end{subfigure}
	\begin{subfigure}[b]{0.329\textwidth}
   	 	\centering
		\includegraphics[height=2.25in]{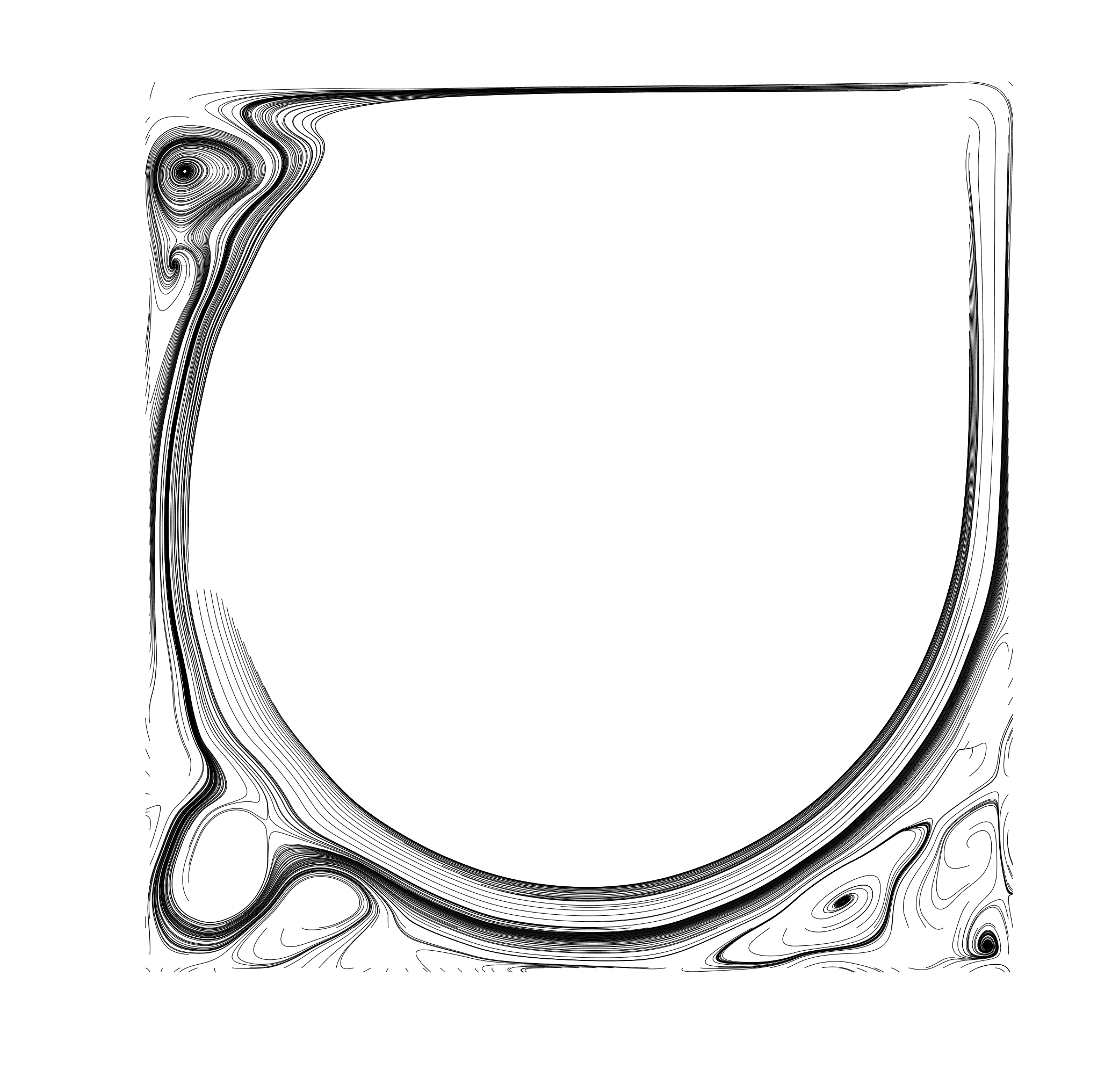}
		\caption{\centering $k'=2$}
	\end{subfigure}
\begin{subfigure}[b]{0.329\textwidth}
   	 	\centering
		\includegraphics[height=2.25in]{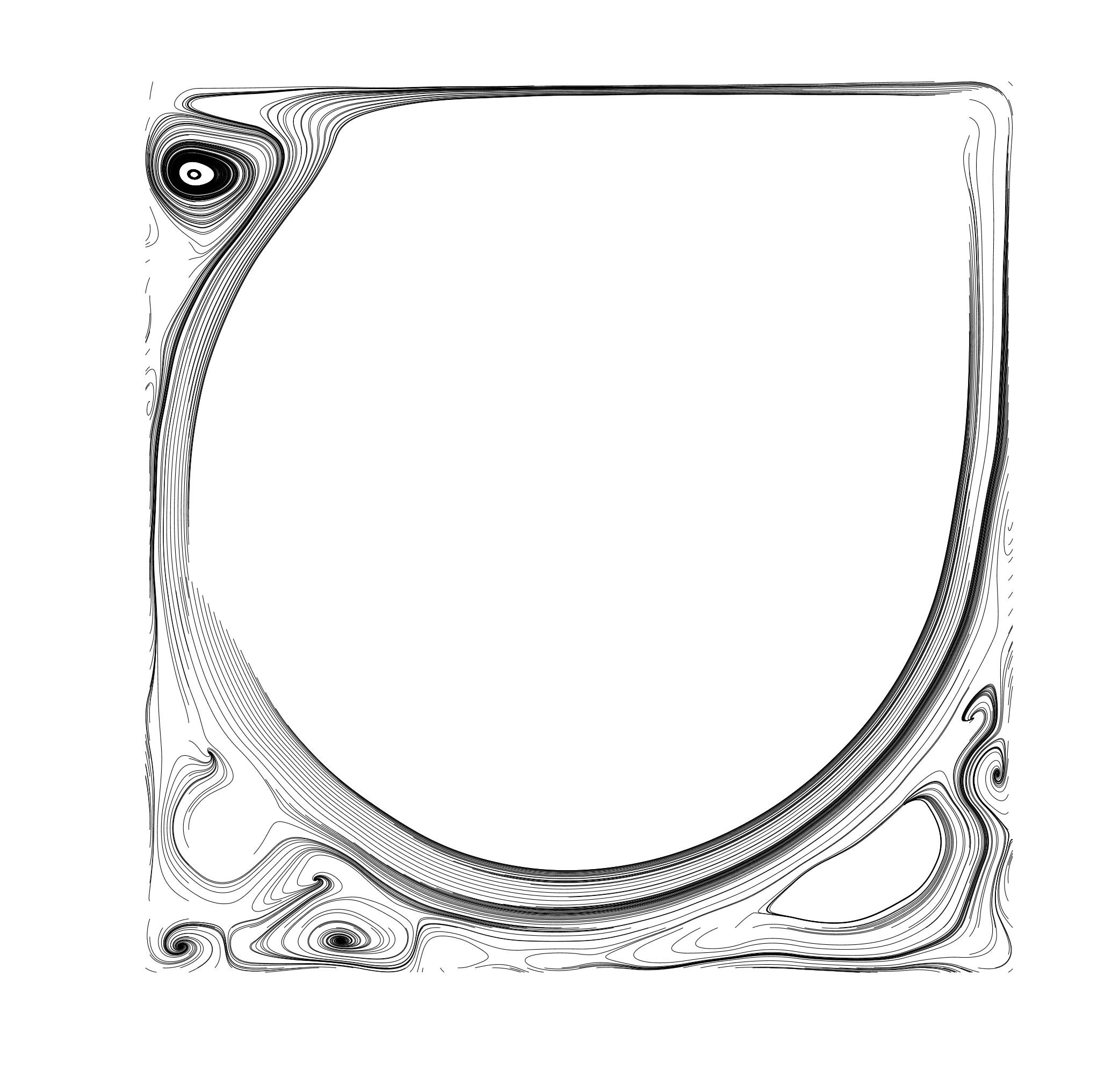}
		\caption{\centering  $k'=3$}
	\end{subfigure}
		\caption{Near wall streamline contours attained using a severely underresolved unstabilized divergence-conforming discretization on a $16\times16$ element mesh for the lid-driven cavity problem at $Re = 10000$.}
\end{figure}

\begin{figure}[t!]
\centering
   	 \begin{subfigure}[b]{0.329\textwidth}
   	 	\centering
		\includegraphics[height=2.25in]{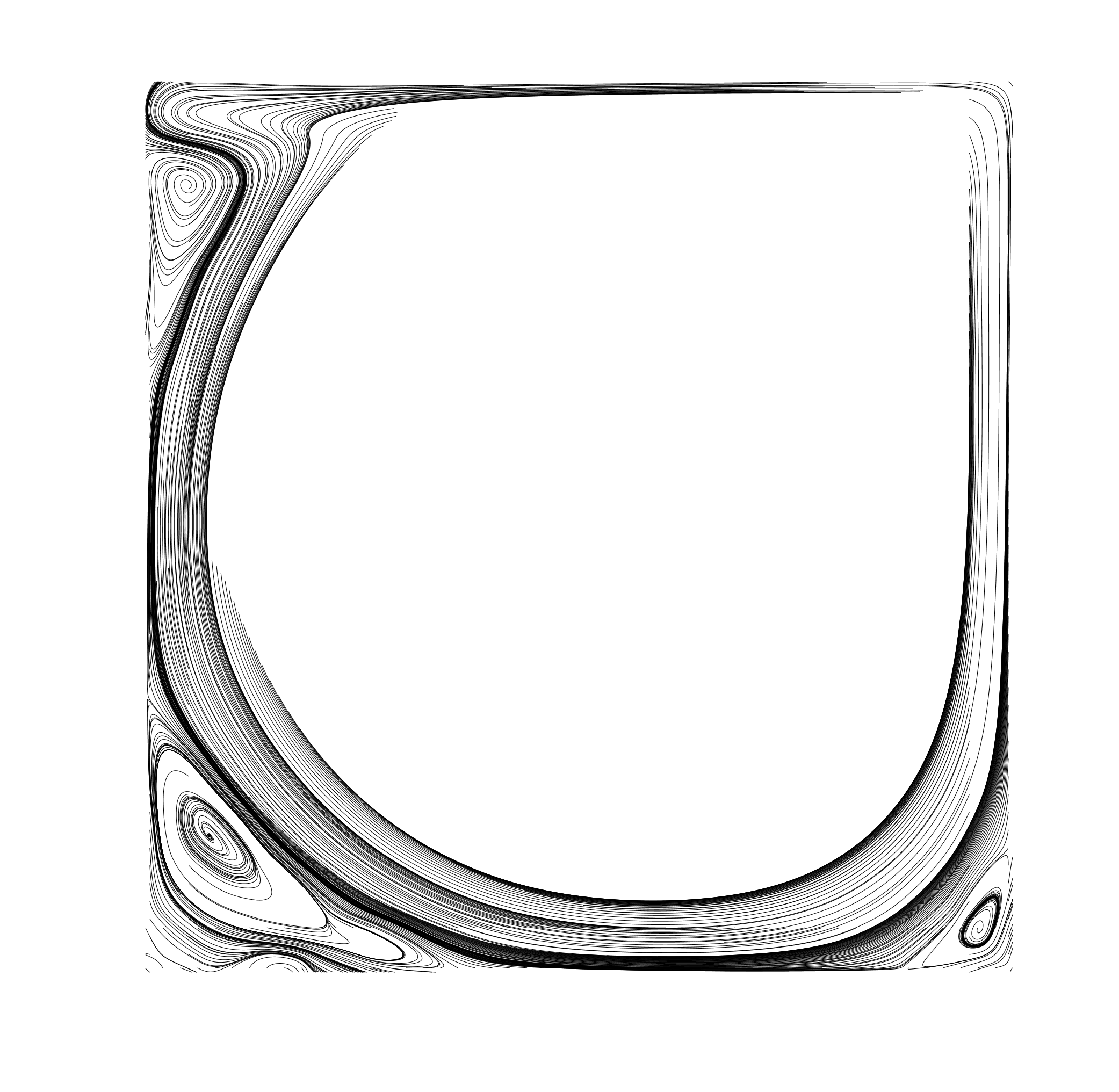}
		\caption{\centering $k'=1$}
	\end{subfigure}
	\begin{subfigure}[b]{0.329\textwidth}
   	 	\centering
		\includegraphics[height=2.25in]{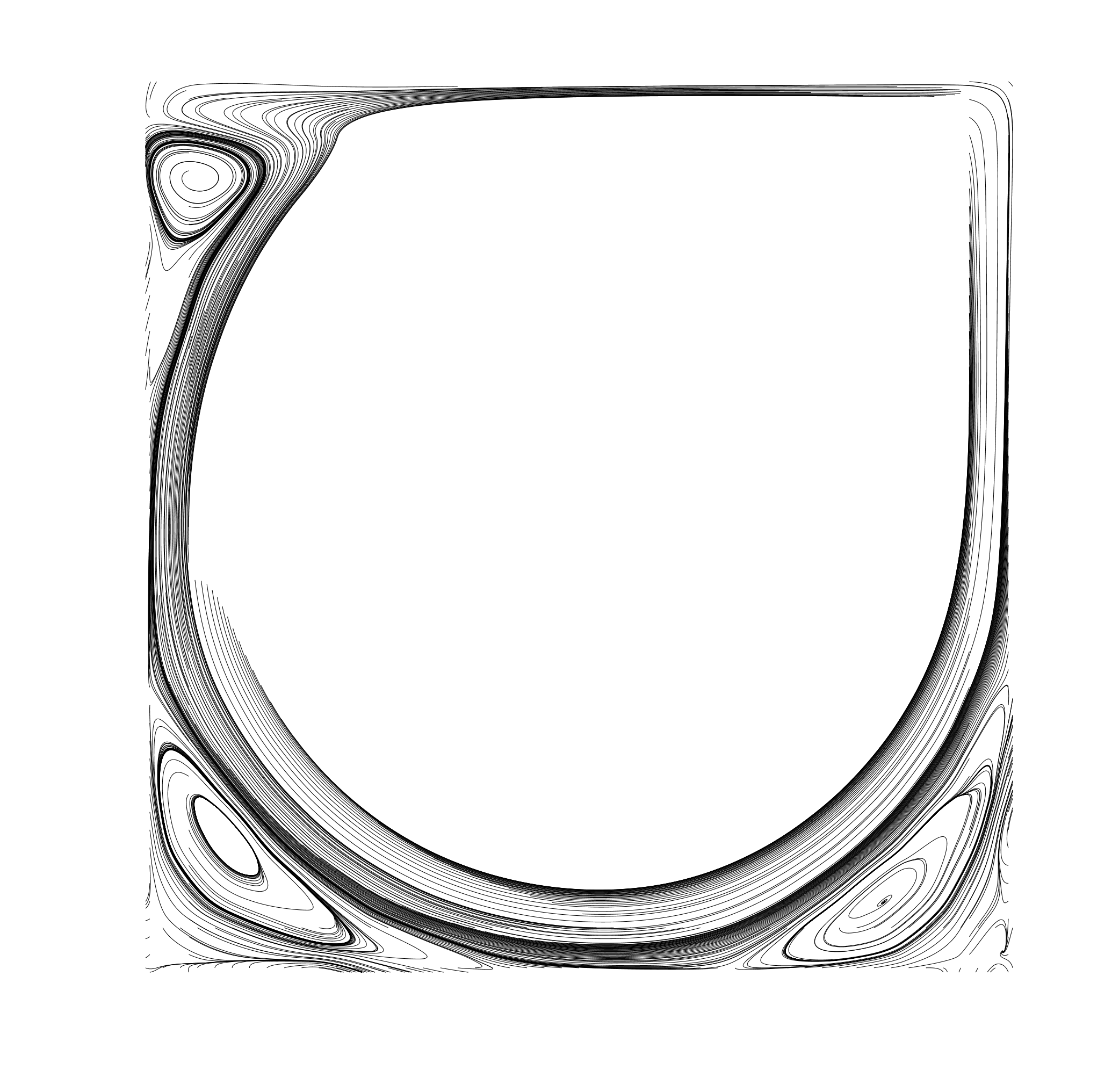}
		\caption{\centering $k'=2$}
	\end{subfigure}
\begin{subfigure}[b]{0.329\textwidth}
   	 	\centering
		\includegraphics[height=2.25in]{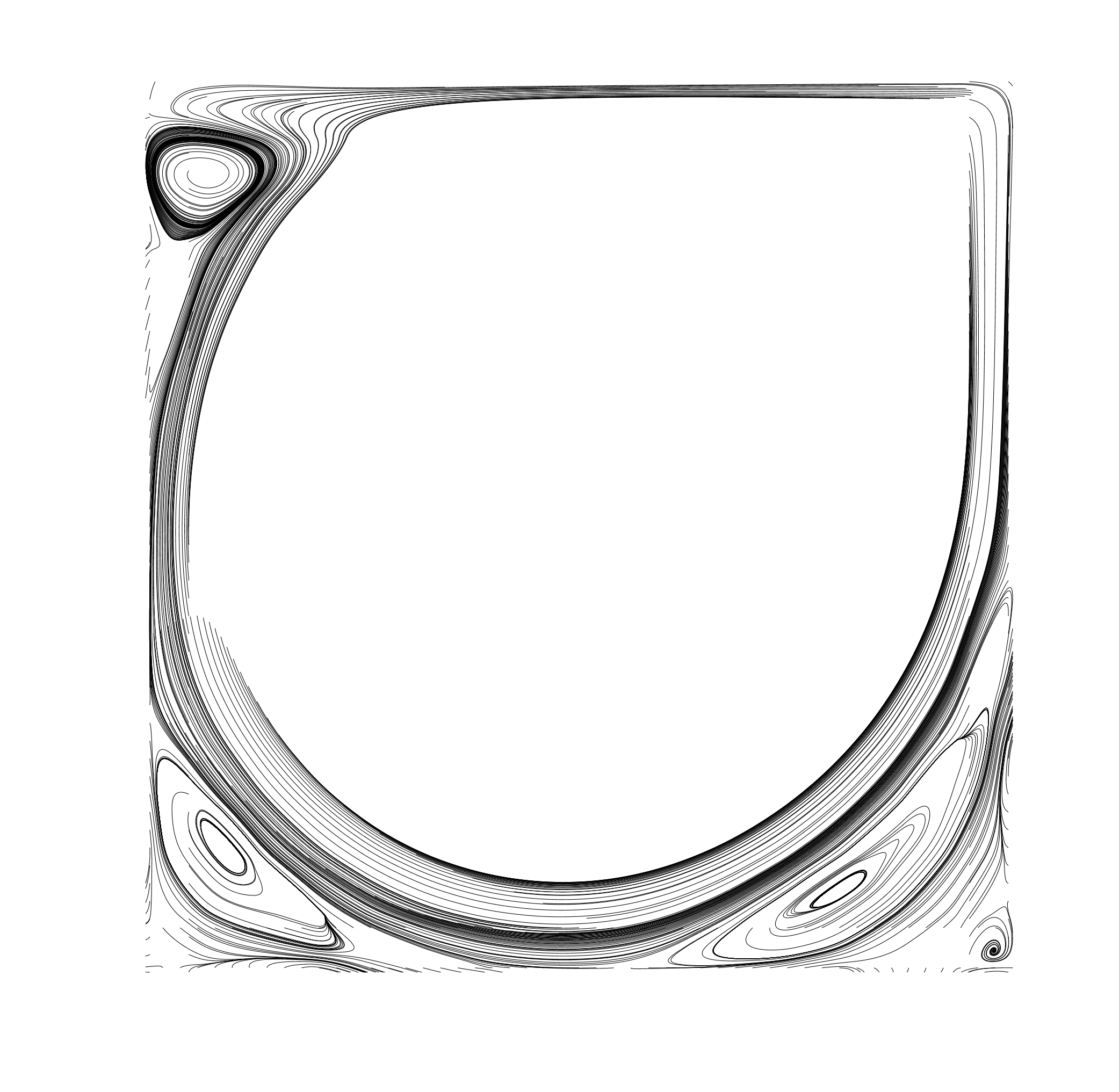}
		\caption{\centering  $k'=3$}
	\end{subfigure}
		\caption{Near wall streamline contours attained using a severely underresolved unstabilized divergence-conforming discretization on a $16\times16$ element mesh for the lid-driven cavity problem at $Re = 10000$.}\label{fig:stab-10000}
\end{figure}

\subsubsection{Centerline velocity profiles}\label{sec:ldc-quantitative}
We now consider a more quantitative comparison of centerline velocity profiles with the results of \cite{GHIA}.  Figures \ref{fig:ldc-vel-center-7500} and \ref{fig:ldc-vel-center-10000} make this comparison for both horizontal (black) and vertical (blue) velocities at Reynolds numbers of 7500 and 10000, respectively.  We consider both well-resolved results (cf. Figure \ref{fig:resolved-solutions-ldc}) on a $128\times 128$ mesh with splines of degree $k'=3$, alongside stabilized and unstabilized results on a very coarse mesh of $16\times 16$ elements, using various spline degrees.
From these results, we can see that the well-resolved solution matches the results of \cite{GHIA} closely.  This is not surprising, though it should be remarked that a uniform $128 \times 128$ element mesh was used for this solution.  As such, the solution is not fully resolved near the wall.  We believe the high accuracy of the solution field near the wall can be attributed to the weak enforcement of tangential Dirichlet boundary conditions using Nitsche's method.  Both the stabilized and unstabilized solutions for the mesh of $16\times 16$ elements are less accurate than the well-resolved solution for each degree $k' = 1, 2, 3$, but this is not surprising as the stabilized and unstabilized solutions are severely underresolved on this coarse mesh.  However, the unstabilized coarse results tend to exhibit more spurious oscillations than the stabilized ones.  This stabilizing effect is more evident in the horizontal component, which is directly driven by the prescribed tangential velocity at the top of the domain.

\begin{figure}[t!]
\centering
   	 \begin{subfigure}[b]{0.329\textwidth}
   	 	\centering
		\includegraphics[height=2in,width=2.2in]{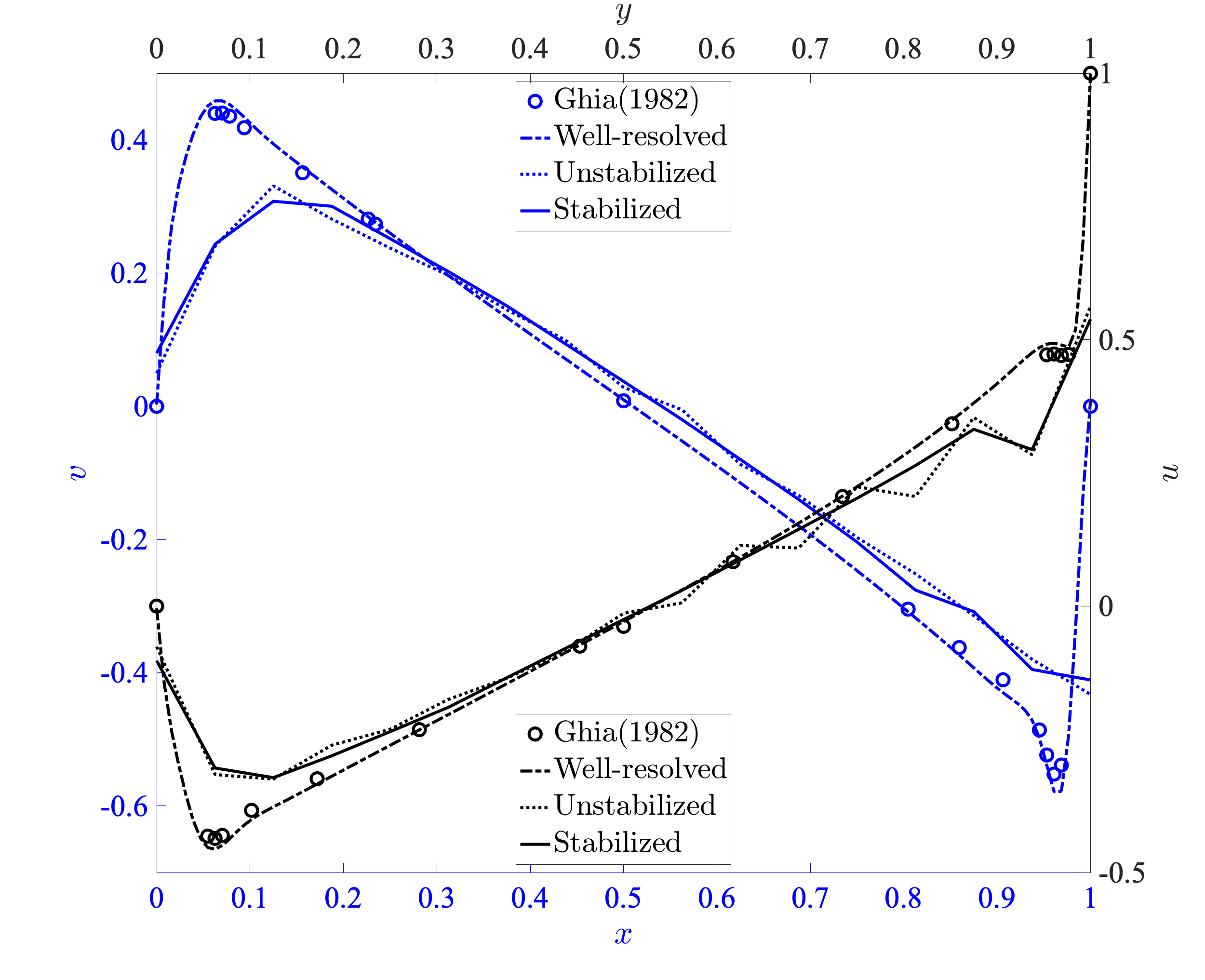}
		\caption{$k'=1$}
	\end{subfigure}
   	 \begin{subfigure}[b]{0.329\textwidth}
   	 	\centering
		\includegraphics[height=2in,width=2.2in]{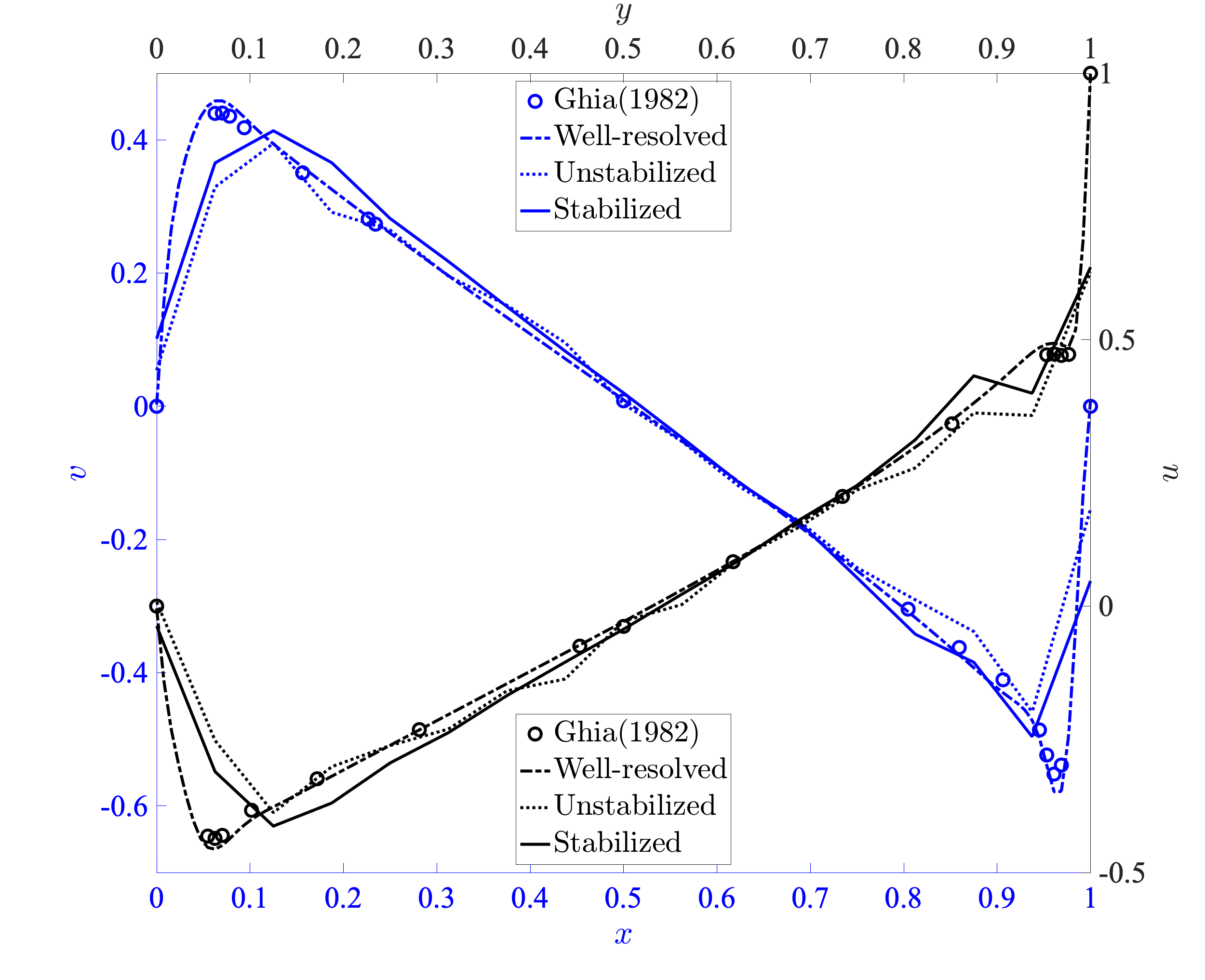}
		\caption{$k'=2$}
	\end{subfigure}
   	 \begin{subfigure}[b]{0.329\textwidth}
   	 	\centering
		\includegraphics[height=2in,width=2.2in]{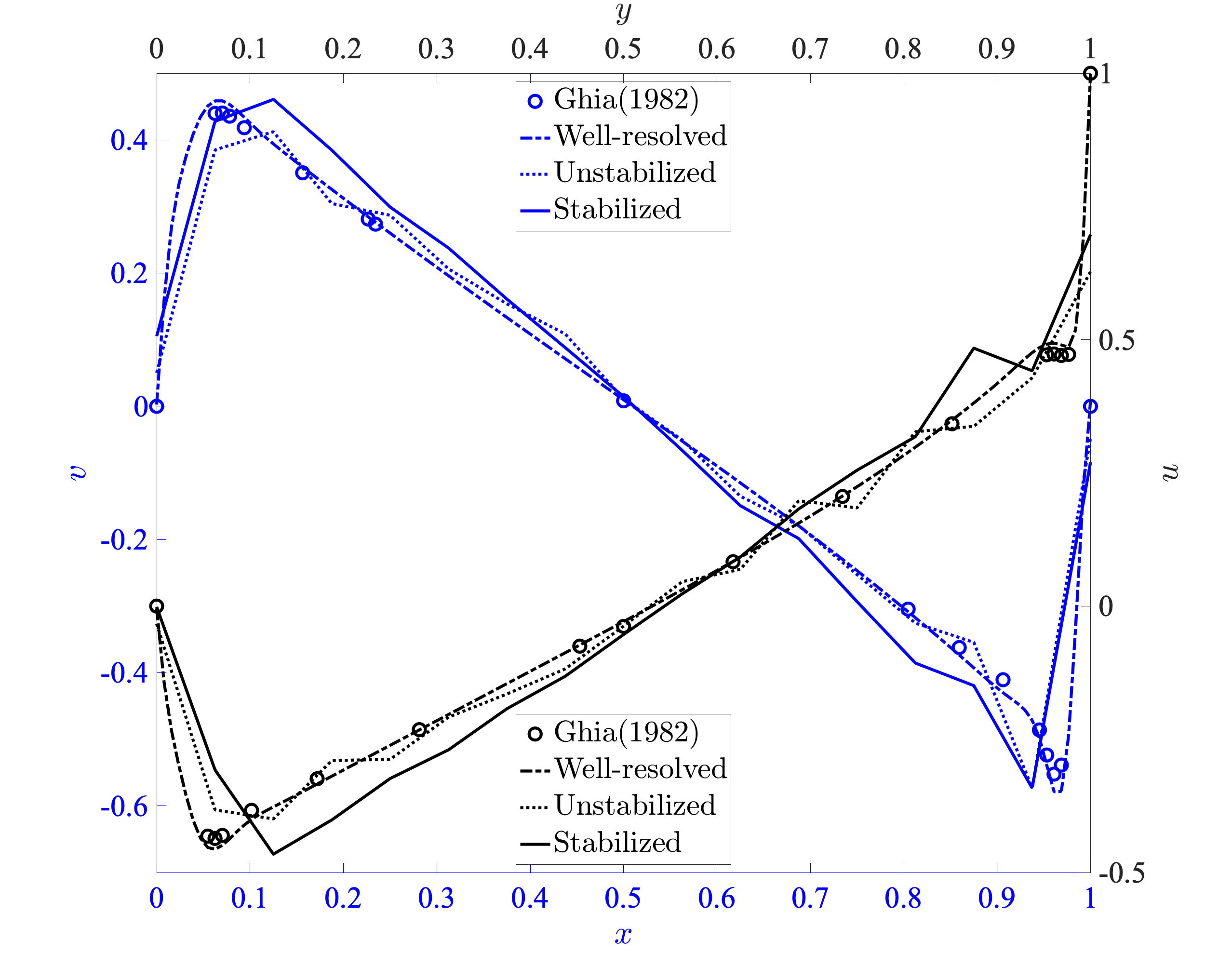}
		\caption{$k'=3$}
	\end{subfigure}
	\caption{Centerline velocity profiles attained using a well-resolved discretization, a severely underresolved unstabilized divergence-conforming discretization on a $16\times16$ element mesh, and a severely underresolved stabilized divergence-conforming discretization on a $16\times16$ element mesh for the lid-driven cavity problem at $Re = 7500$.}\label{fig:ldc-vel-center-7500}
\end{figure}

\begin{figure}[t!]
\centering
   	 \begin{subfigure}[b]{0.329\textwidth}
   	 	\centering
		\includegraphics[height=2in,width=2.2in]{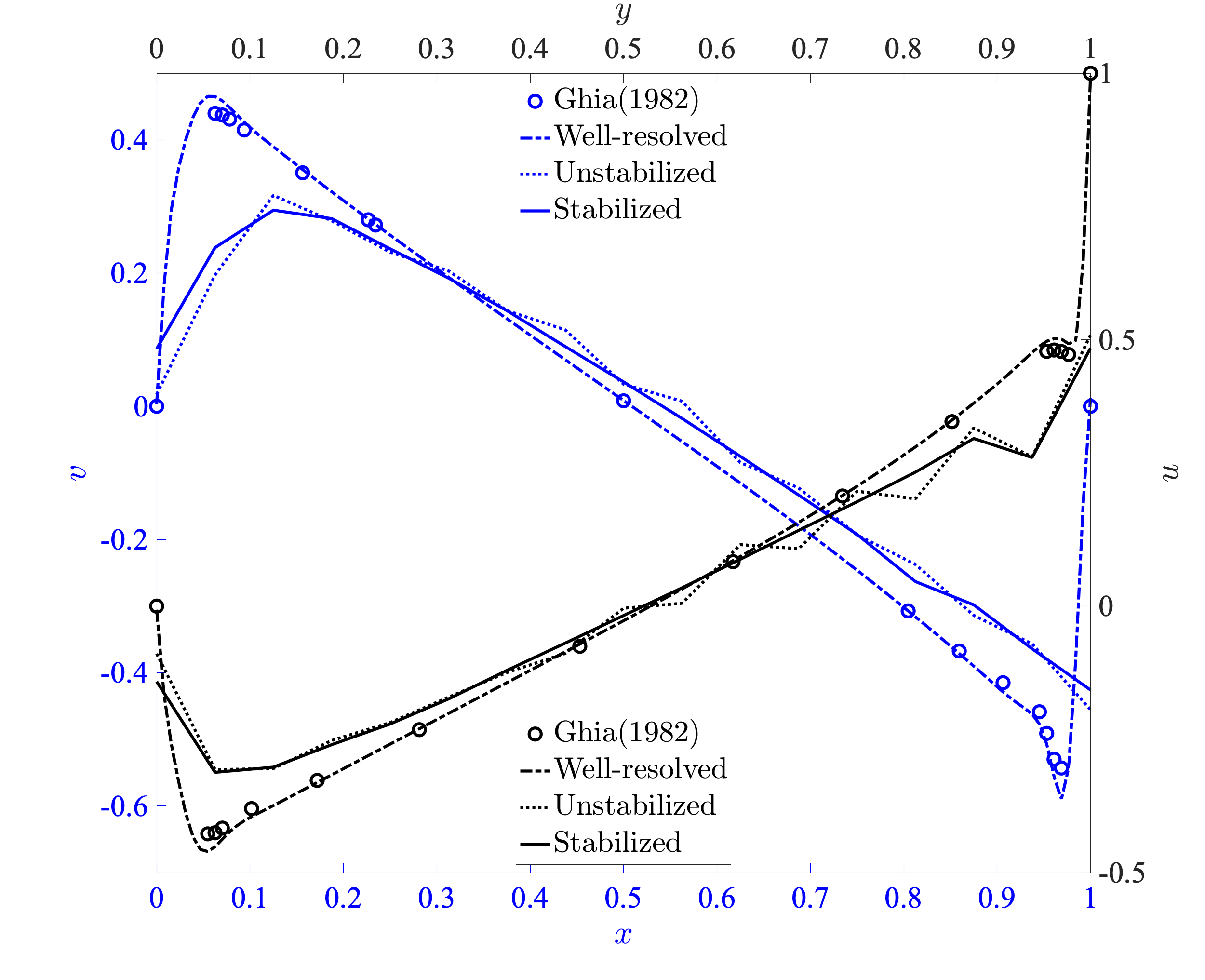}
		\caption{$k'=1$}
	\end{subfigure}
   	 \begin{subfigure}[b]{0.329\textwidth}
   	 	\centering
		\includegraphics[height=2in,width=2.2in]{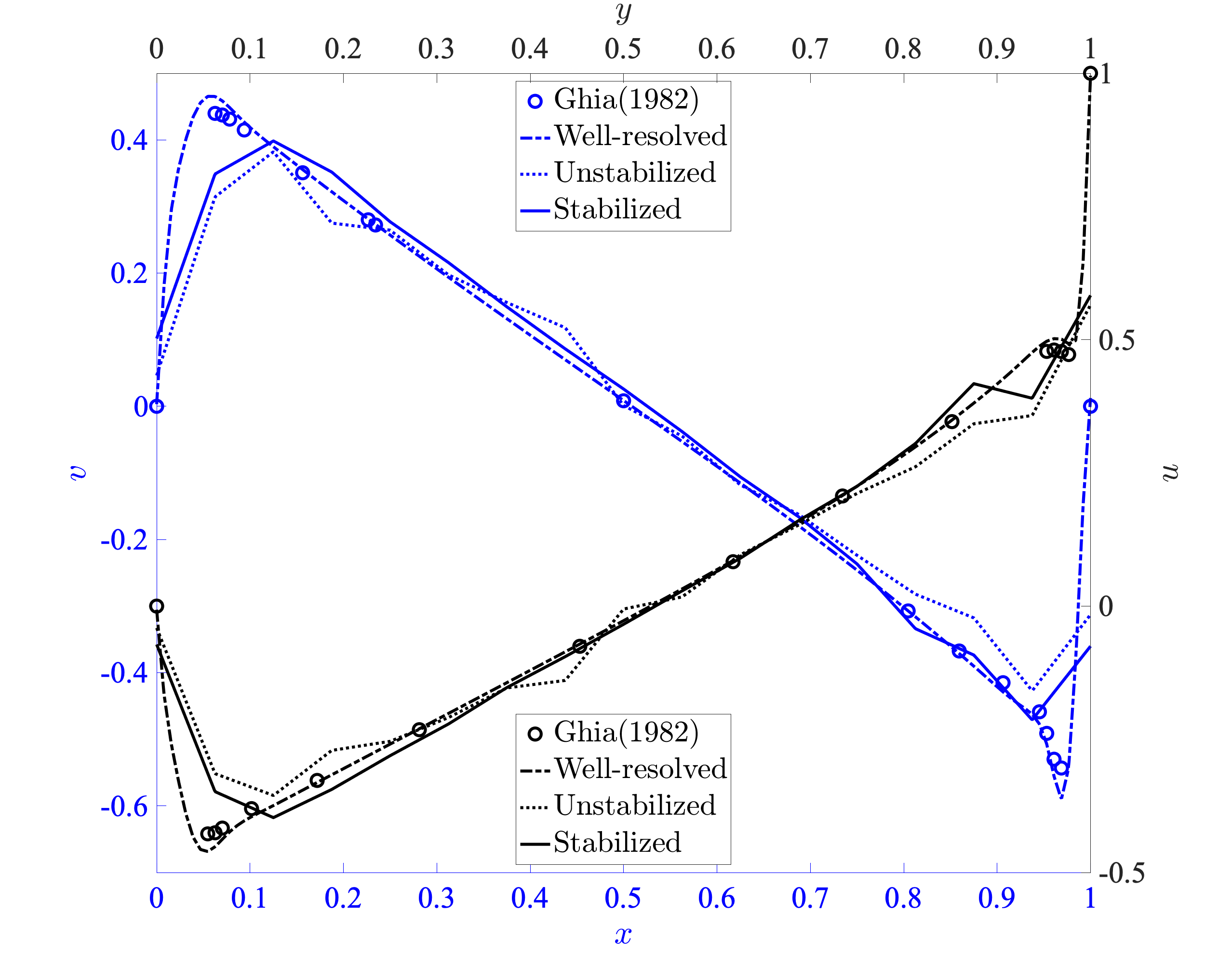}
		\caption{$k'=2$}
	\end{subfigure}
   	 \begin{subfigure}[b]{0.329\textwidth}
   	 	\centering
		\includegraphics[height=2in,width=2.2in]{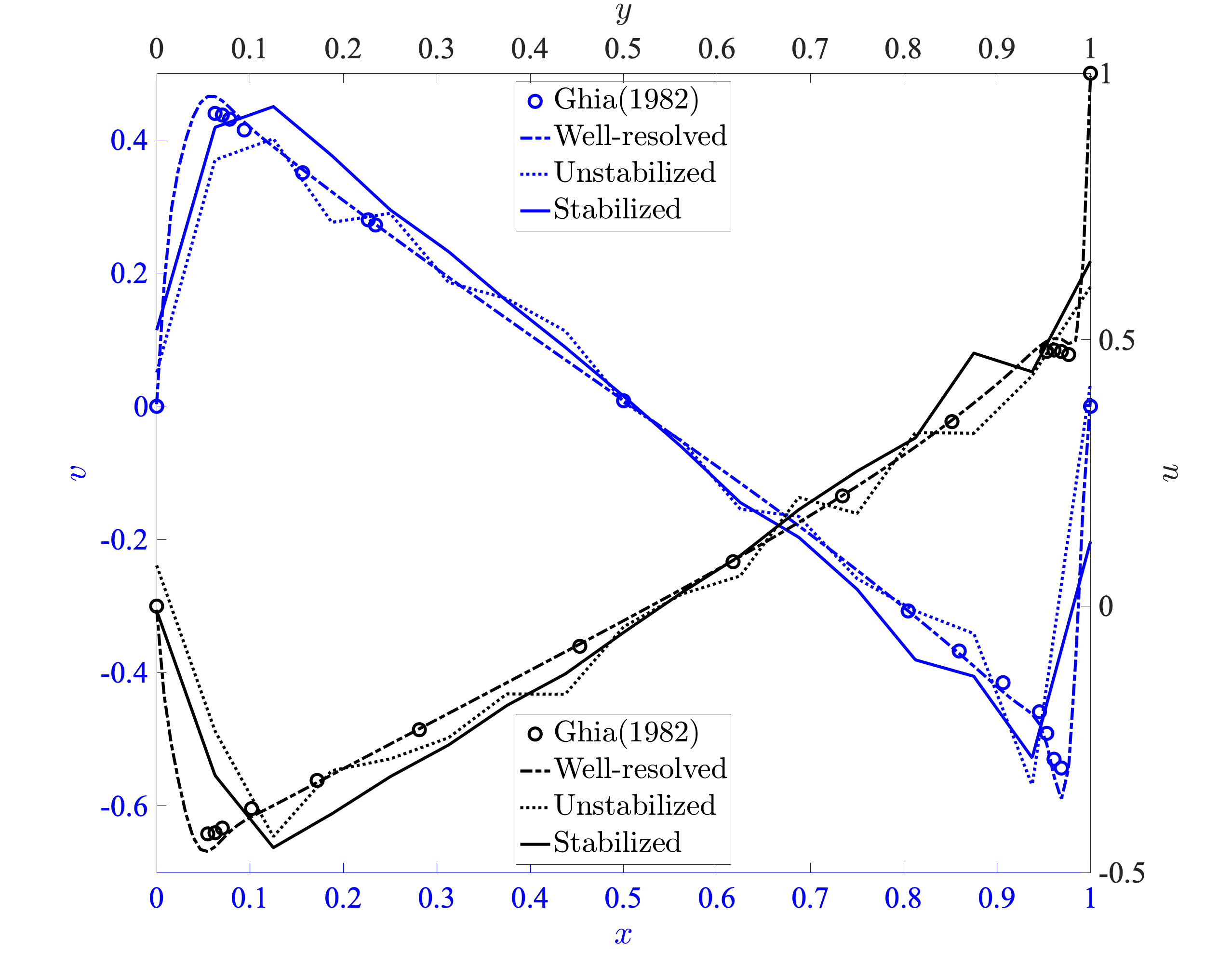}
		\caption{$k'=3$}
	\end{subfigure}
	\caption{Centerline velocity profiles attained using a well-resolved discretization, a severely underresolved unstabilized divergence-conforming discretization on a $16\times16$ element mesh, and a severely underresolved stabilized divergence-conforming discretization on a $16\times16$ element mesh for the lid-driven cavity problem at $Re = 10000$.}\label{fig:ldc-vel-center-10000}
\end{figure}

\subsection{3D Taylor--Green vortex problem}
\label{tg3d}
As a final example, we apply the proposed method in 3D, to the Taylor--Green vortex benchmark, which evolves from the periodic initial velocity
\begin{align}
\boldsymbol{u}_0(\boldsymbol{x}) = \textrm{sin}(x_1) \textrm{cos}(x_2) \textrm{cos}(x_3)\boldsymbol{e}_1  -\textrm{cos}(x_1)\textrm{sin}(x_2) \textrm{cos}(x_3)\boldsymbol{e}_2\text{ .}
\end{align}
This problem can be restricted to the finite cube $\lbrack 0,\pi\rbrack^3$ with free-slip boundary conditions (corresponding to symmetry planes of the problem data).  At sufficiently high Reynolds numbers, the initial laminar flow structure breaks down chaotically into smaller eddies, making it a useful case study in the transition to turbulence \cite{BO83,EV13,EV20,VA17}.  In this paper, we focus on the case of $\textrm{Re}=1600$, over a time interval of length $T=10$.  To explore the behavior of the method in an underresolved setting, we use a mesh of $32^3$ elements.  Our discussion of the results will look at several quantities of interest in turbulent flows.

\subsubsection{Dissipation rate}
The rate $\epsilon$ of energy dissipation in the semidiscrete problem ($\mathcal{W}_h$) can be split into resolved dissipation, $\epsilon_r$, and model dissipation, $\epsilon_m$:
\begin{equation}
    \epsilon = -\frac{dE_k}{dt} = \epsilon_r + \epsilon_m\text{ ,}
\end{equation}
where
\begin{equation}
    E_k = \frac{1}{2V}\Vert\boldsymbol{u}_h\Vert_{\mathbf{L}^2}^2
\end{equation}
is the total kinetic energy, normalized by the volume $V$ of the domain.
The resolved dissipation rate corresponds to directly plugging the resolved velocity $\boldsymbol{u}_h$ into the continuous problem's formula for dissipation:
\begin{equation}
\epsilon_r = \frac{2\nu}{V} (\nabla^s \boldsymbol{u}_h, \nabla^s \boldsymbol{u}_h)\text{ .}
\label{resolved}
\end{equation}
The model dissipation is the additional dissipation due to the skeleton stabilization:
\begin{equation}
\epsilon_{m} =\frac{1}{V} \sum_{e \in \mathcal{E}_0} (\eta [\![\partial^{\alpha' + 1}_{\boldsymbol{n}} \boldsymbol{u}_h]\!], [\![ \partial^{\alpha' + 1}_{\boldsymbol{n}} \boldsymbol{u}_h]\!])_e\text{ .}
\end{equation}
This additional dissipation can be viewed as an Implicit Large Eddy Simulation (ILES) filter, as discussed further in \cite{BU0606}.
%

\begin{figure}[b!]
\centering
   	 \begin{subfigure}[b]{0.49\textwidth}
   	 	\centering
		\includegraphics[height=3in]{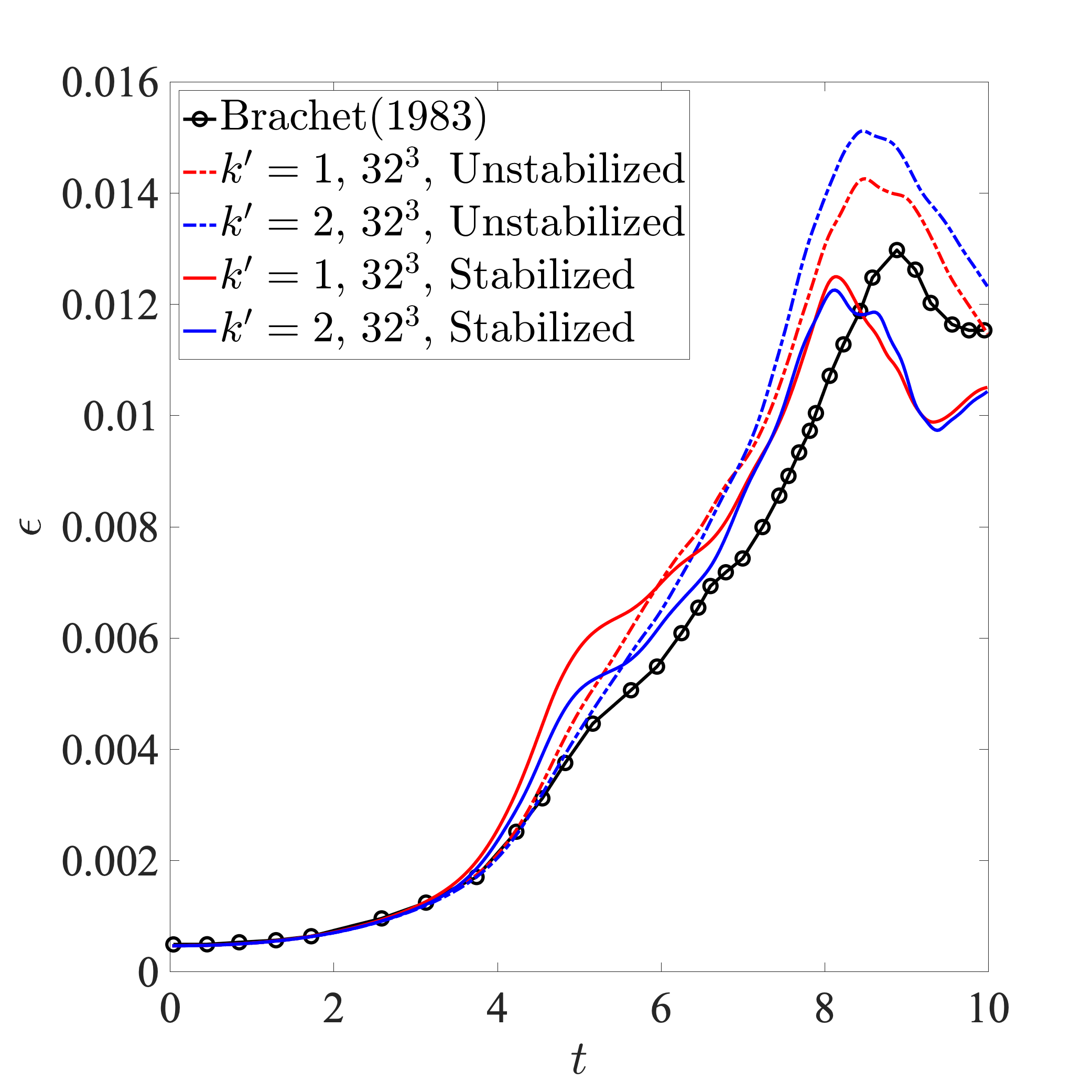}
		\caption{Total dissipation rate}
		\label{18a}
	\end{subfigure}
	\begin{subfigure}[b]{0.49\textwidth}
   	 	\centering
		\includegraphics[height=3in]{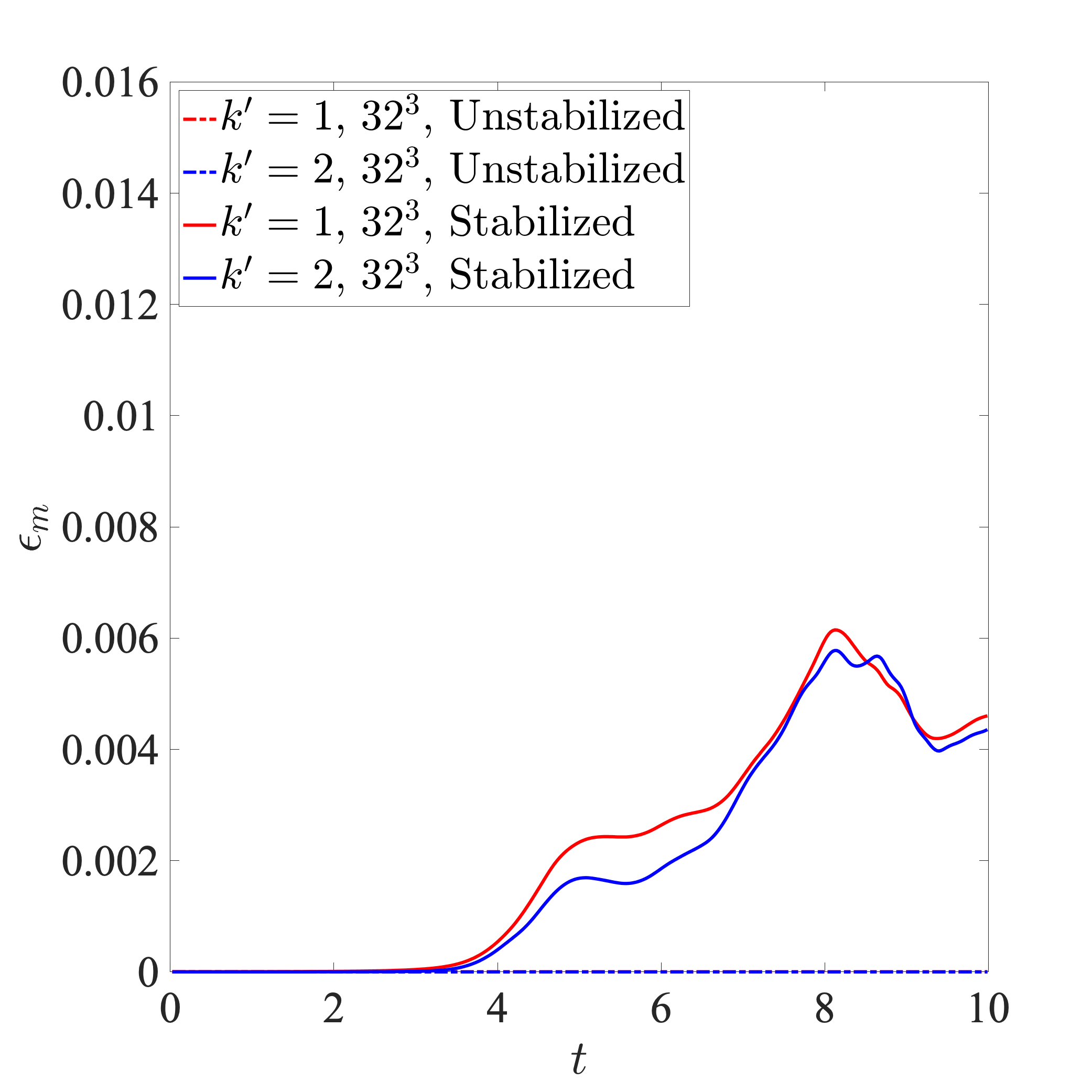}
		\caption{Modeled dissipation rate}
		\label{18b}
	\end{subfigure}
	\caption{Evolution of total and modeled dissipation rate for the 3D Taylor--Green vortex problem at $Re=1600$.}
	\label{fig18}
\end{figure}

Figure \ref{fig18} shows the evolution of the total and model dissipation rates for divergence-conforming skeleton stabilization, alongside the Direct Numerical Simulation (DNS) results of \cite{BO83} (which were computed using a spectral method with $256^3$ Fourier modes).  In the unstabilized computations, all dissipation is due to resolved dissipation, because $\epsilon_m = 0$.  For an $\mathbf{H}^1$-stable method on a coarse mesh, one would expect the resolved dissipation to be substantially lower than the true dissipation, because rapidly-varying components of the full solution would be filtered out.  The unstabilized resolved dissipation is, however, higher than the dissipation in the reference solution, suggesting the presence of large spurious gradients in the coarse discrete velocity.  In the stabilized computations, however, a significant portion of the dissipation is due to $\epsilon_m$.  Figure \ref{f19} compares the resolved dissipation of the stabilized solution to data from \cite{y17}, where a DNS velocity solution is processed with a box filter of width matching the element size in our computations.  This more clearly demonstrates how the skeleton stabilization can be viewed as an implicit filter, and how the absence of stabilization causes energy to be dissipated through spurious coarse-scale velocity gradients.

Comparing results for $k'=1$ and $k'=2$, we see that (as is typically the case with rough or rapidly-varying exact solutions), we gain little from order elevation on a coarse mesh after $t\sim 7$, where the flow reaches its maximum dissipation rate and is dominated by small, unresolved eddies.  We do see a somewhat more accurate dissipation rate with $k'=2$ in the early stages of the computation, during the initial transition to turbulence.  When using B-splines of maximal continuity, this improvement is attained through only a minimal increase in the total degrees of freedom.  However, constant factors in the cost of assembly and solution procedures may increase substantially without specialized quadrature or solver technologies that take advantage of increased smoothness (e.g., \cite{hughes2010efficient}).  

\begin{figure}[t!]
\includegraphics[height = 3in]{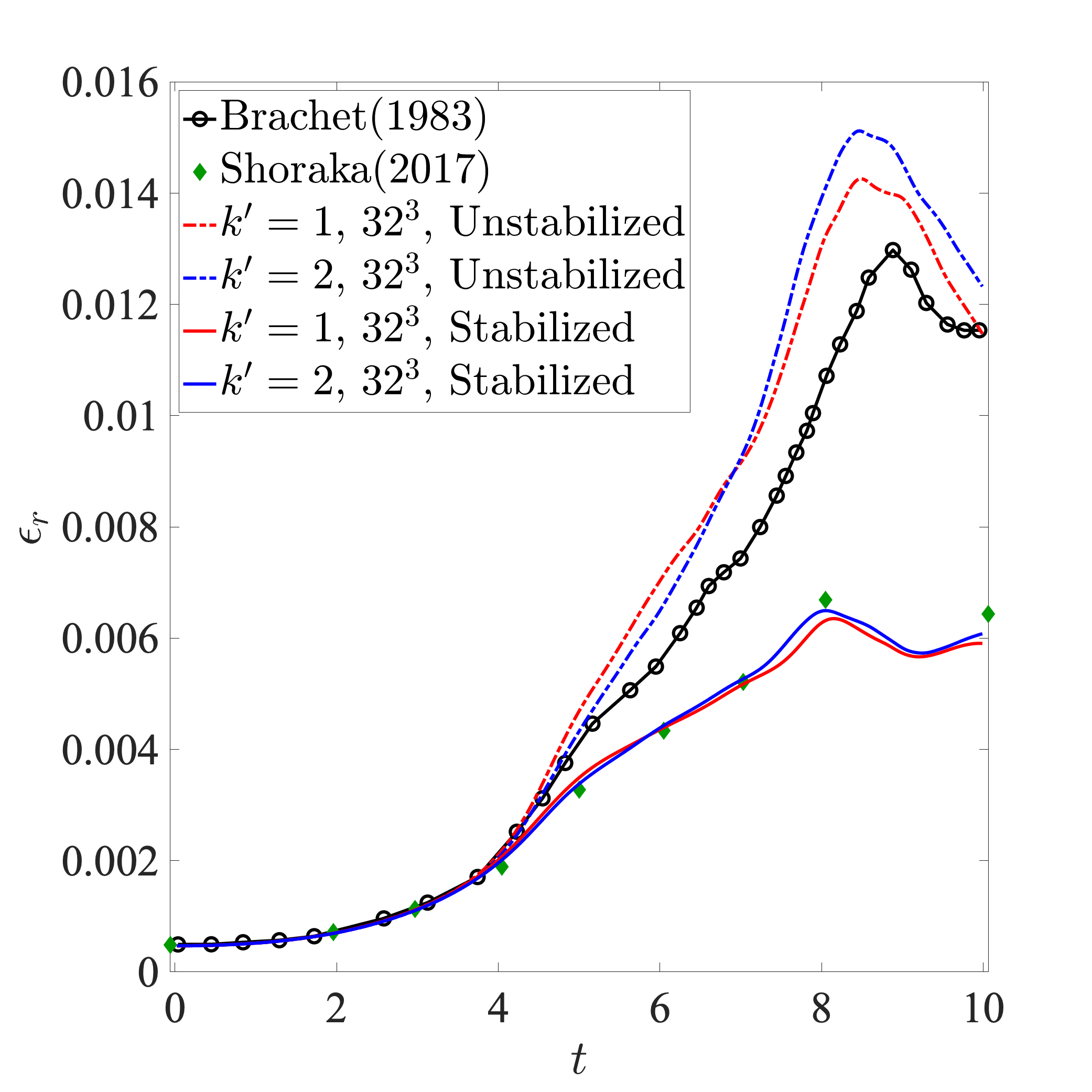}
\centering
\caption{Evolution of resolved dissipation rate for the 3D Taylor--Green vortex problem at $Re = 1600$.}
\label{f19}
\end{figure}

\subsubsection{Energy spectrum}
To attain a better understanding of how energy is distributed among the scales of motion for both the skeleton-stabilized and unstabilized divergence-conforming B-spline discretizations, we examine energy spectra resulting from our computations.  In particular, we examine the time-instantaneous energy spectra $\hat{E}_{\kappa}$ at time $t = 9$ as a function of wavenumber $\kappa$.  This time is near the time of peak dissipation.  Figure \ref{f20} shows the energy spectra attained using both the stabilized and unstabilized discretizations alongside the DNS spectra attained in \cite{BO83}.  The energy spectra attained using the stabilized discretizations are in good agreement with the DNS spectra.  The stabilized spectrum for degree $k' = 2$ is marginally more accurate than the stabilized spectrum for degree $k' = 1$ except for the very highest resolved wavenumbers.  There is a large pile-up of energy at the highest resolved wavenumbers for the unstabilized discretizations that is not seen for the stabilized discretizations.  This pile-up of energy, like the resolved dissipation plots discussed previously, suggests the presence of large spurious gradients in the unstabilized velocity fields.  This is not surprising, as there is no dissipation mechanism in the unstabilized discretizations except for viscous dissipation.  The unstabilized discretizations also underpredict the energy contained in the low wavenumbers, which is likely due to the fact that the unstabilized discretizations exhibit excessive total dissipation shortly before $t = 9$.  This dissipation results in a pronounced drop in total kinetic energy $E_k$ versus the stabilized discretizations, as seen in Figure \ref{f21}.

\begin{figure}[t!]
\includegraphics[height = 3in]{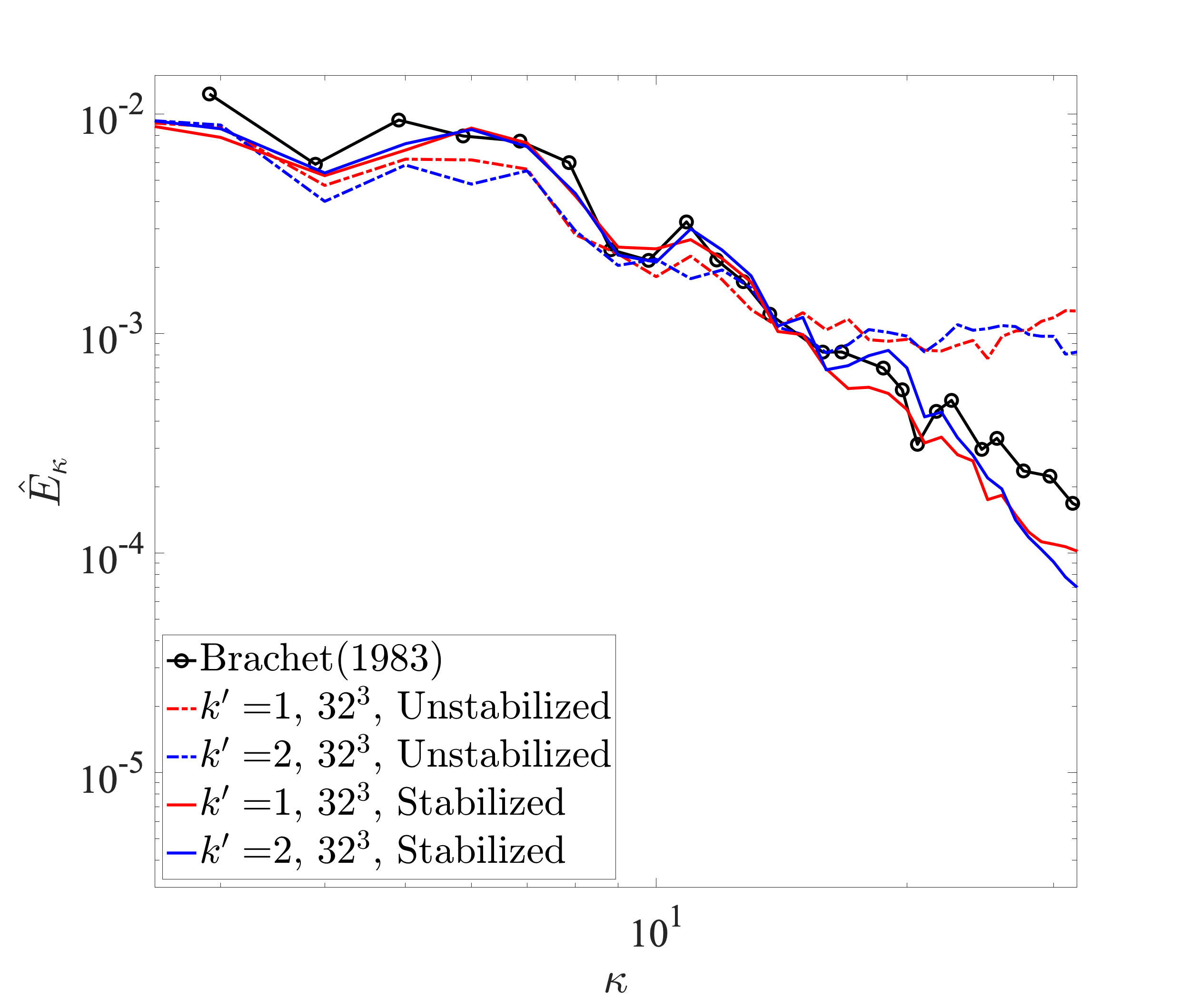}
\centering
\caption{Energy spectrum at $t = 9$ for the 3D Taylor--Green vortex problem at $Re = 1600$.}
\label{f20}
\end{figure}

\begin{figure}[t!]
\includegraphics[height = 3in]{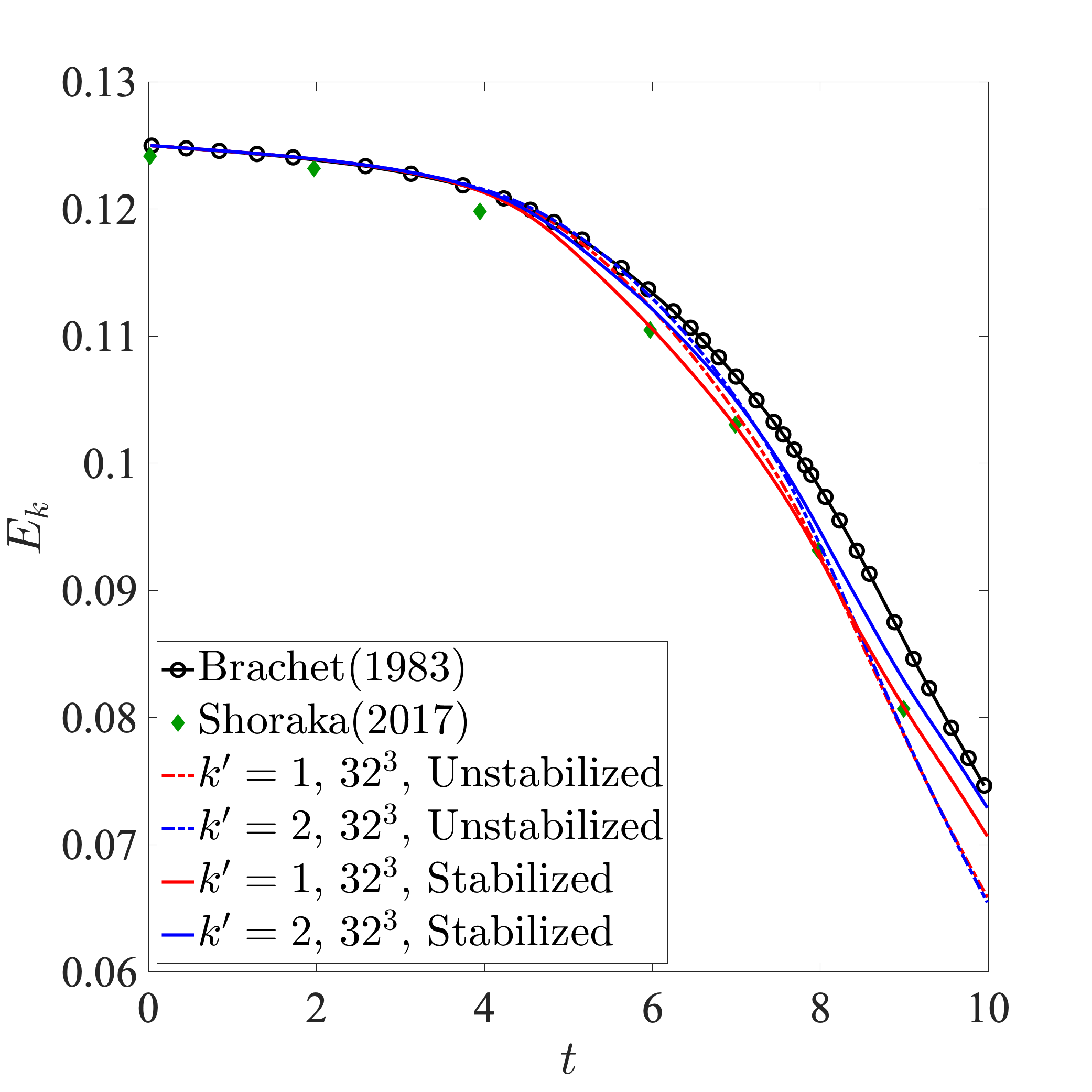}
\centering
\caption{Evolution of total kinetic energy for the 3D Taylor--Green vortex problem at $Re = 1600$.}
\label{f21}
\end{figure}

\subsubsection{$Q$-criterion}
To visualize the energy pile-up shown in the last section in an alternative way, we turn to the $Q$-criterion, defined as the second invariant of the velocity gradient tensor \cite{KO07}: 

\begin{equation}
Q = \frac{1}{2} \left( \boldsymbol{\Omega}:\boldsymbol{\Omega}  - \boldsymbol{S}:\boldsymbol{S} \right)\text{ .}
\end{equation}

\begin{figure}[t!]
\centering
   	 \begin{subfigure}[b]{0.32\textwidth}
   	 	\centering
		\includegraphics[height=2.1in, width = 0.9999\textwidth]{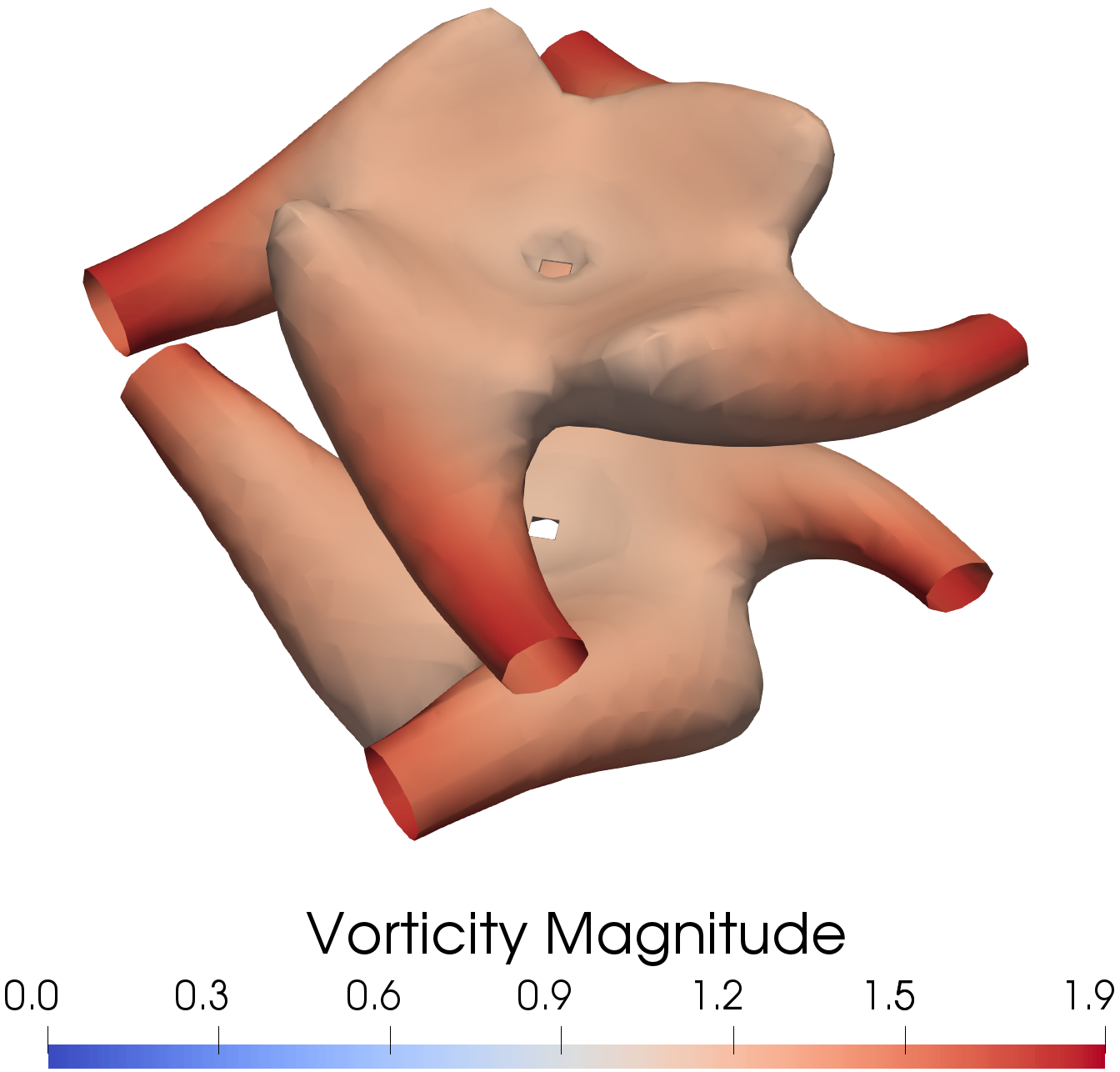}
		\caption{$t\sim2$}
	\end{subfigure}
	\begin{subfigure}[b]{0.32\textwidth}
   	 	\centering
		\includegraphics[height=2.5in]{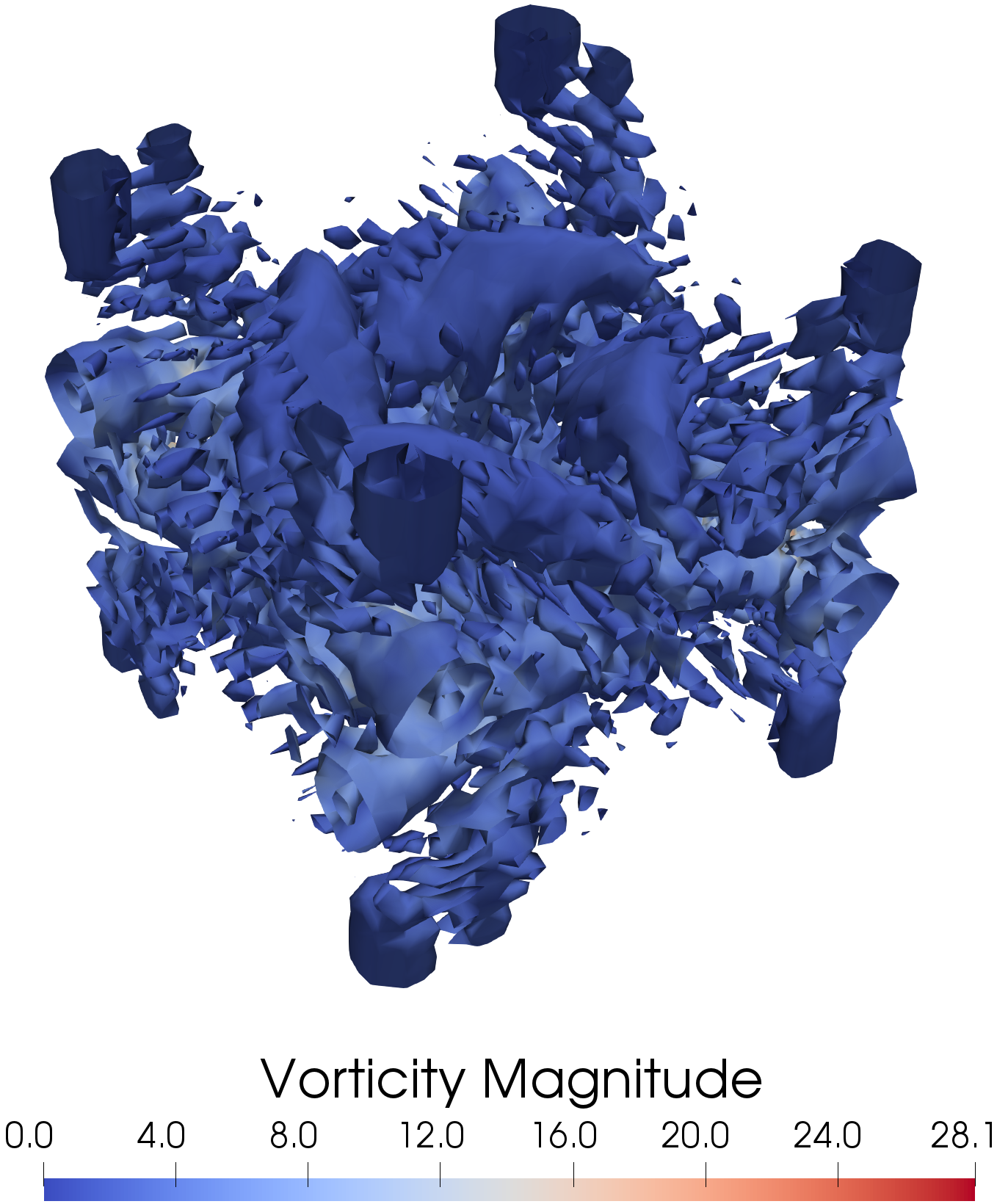}
		\caption{$t\sim6$}
	\end{subfigure}
	\begin{subfigure}[b]{0.32\textwidth}
   	 	\centering
		\includegraphics[height=2.5in]{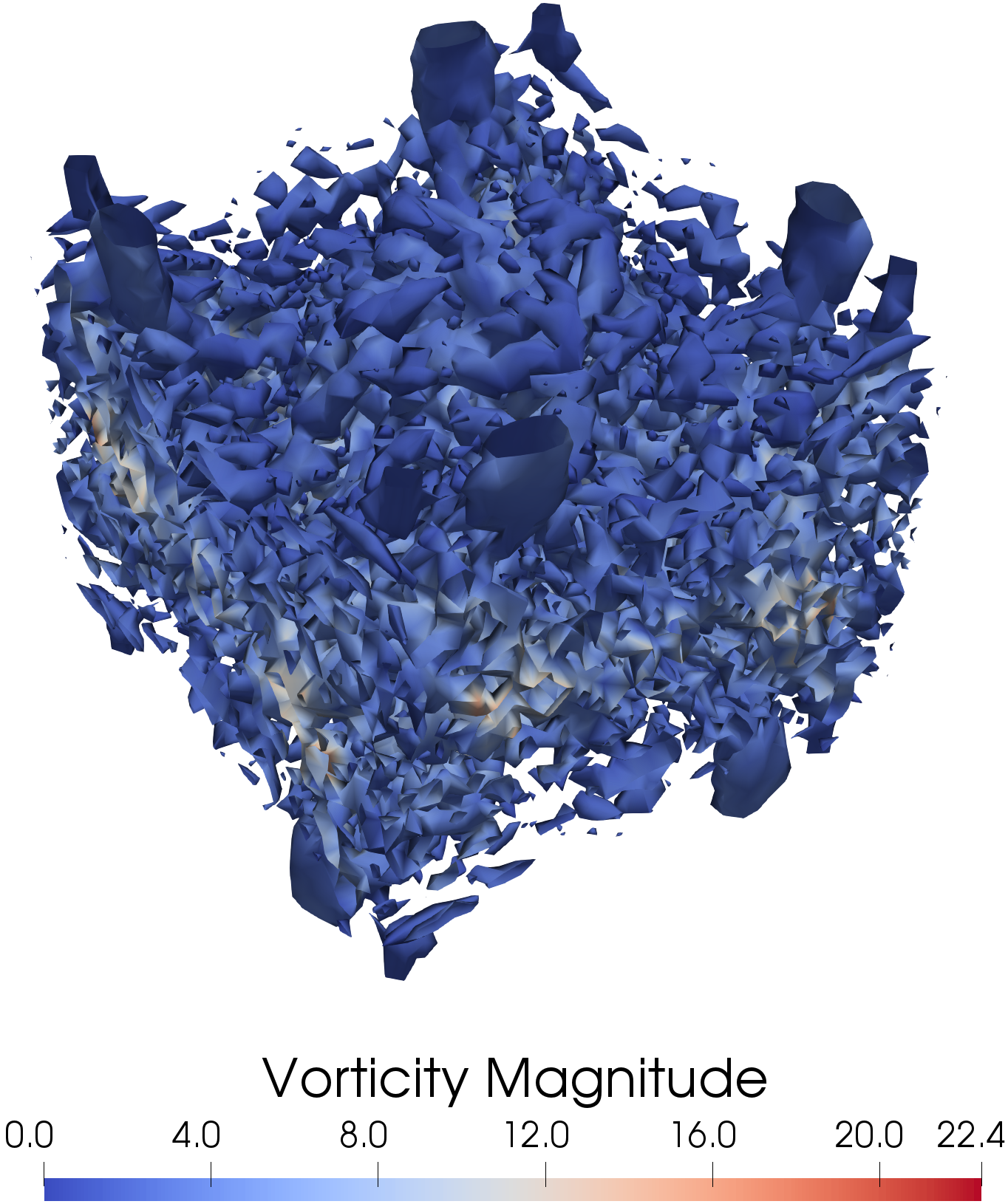}
		\caption{$t\sim10$}
	\end{subfigure}
	\caption{$Q$-criterion isocontours for $Q = 0.3$ attained using an unstabilized divergence-conforming discretization of polynomial degree $k' = 1$ on a $32^3$ element mesh for the 3D Taylor--Green vortex problem.}
\label{f22}	
\end{figure}
\begin{figure}[t!]
\centering
   	 \begin{subfigure}[b]{0.32\textwidth}
   	 	\centering
		\includegraphics[height=2.1in, width = 0.9999\textwidth]{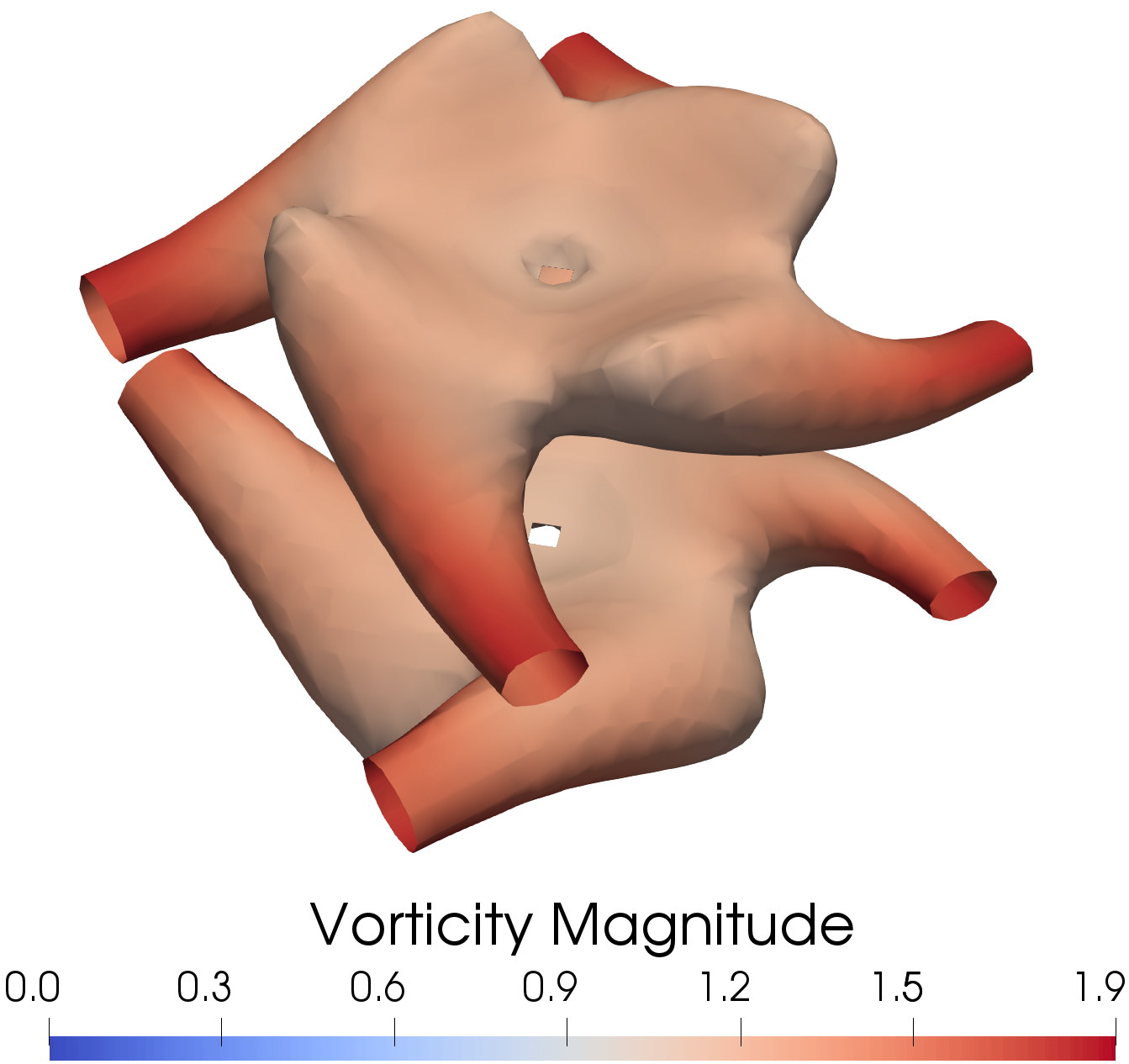}
		\caption{$t\sim2$}
	\end{subfigure}
	\begin{subfigure}[b]{0.32\textwidth}
   	 	\centering
		\includegraphics[height=2.5in]{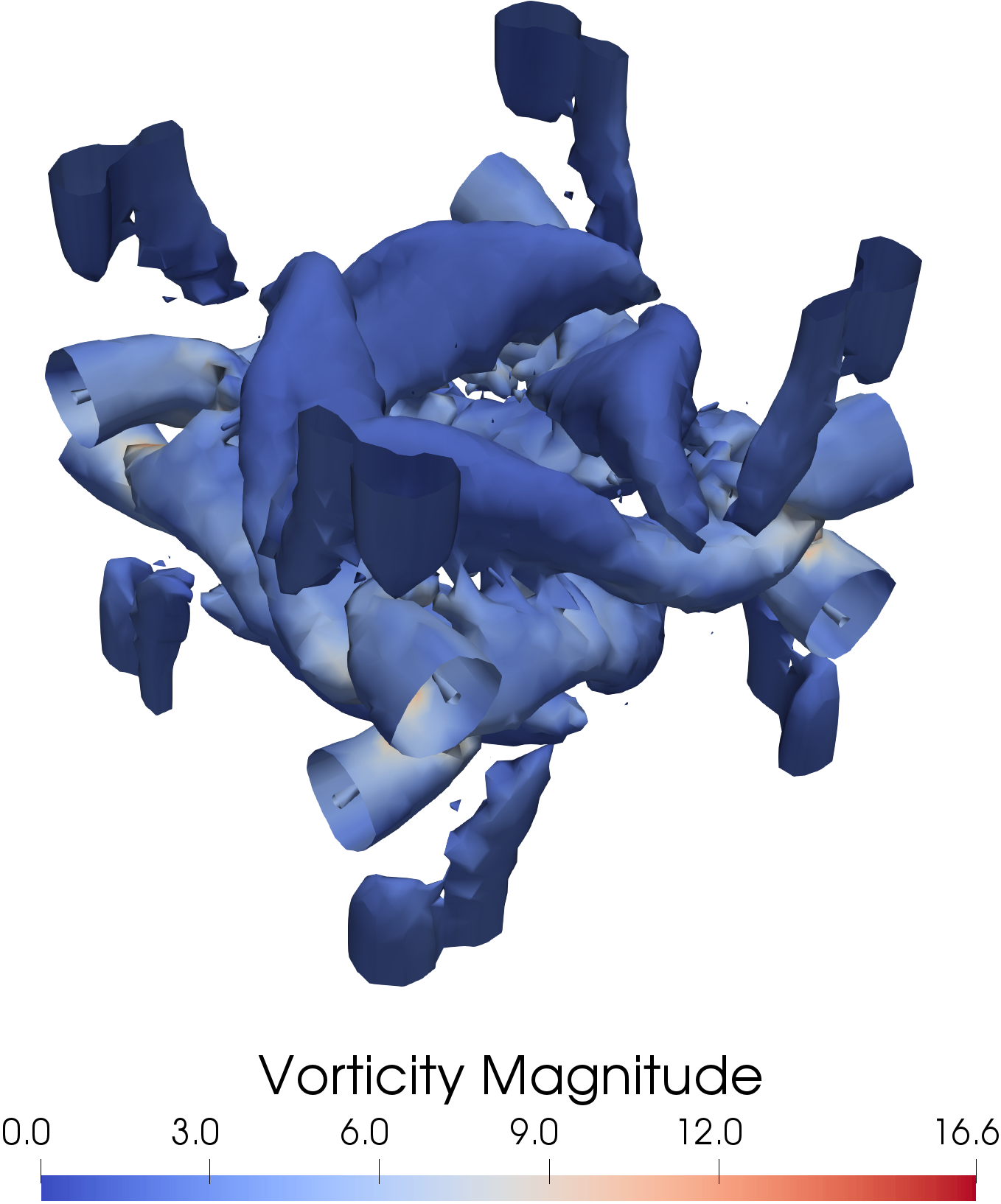}
		\caption{$t\sim6$}
	\end{subfigure}
	\begin{subfigure}[b]{0.32\textwidth}
   	 	\centering
		\includegraphics[height=2.5in]{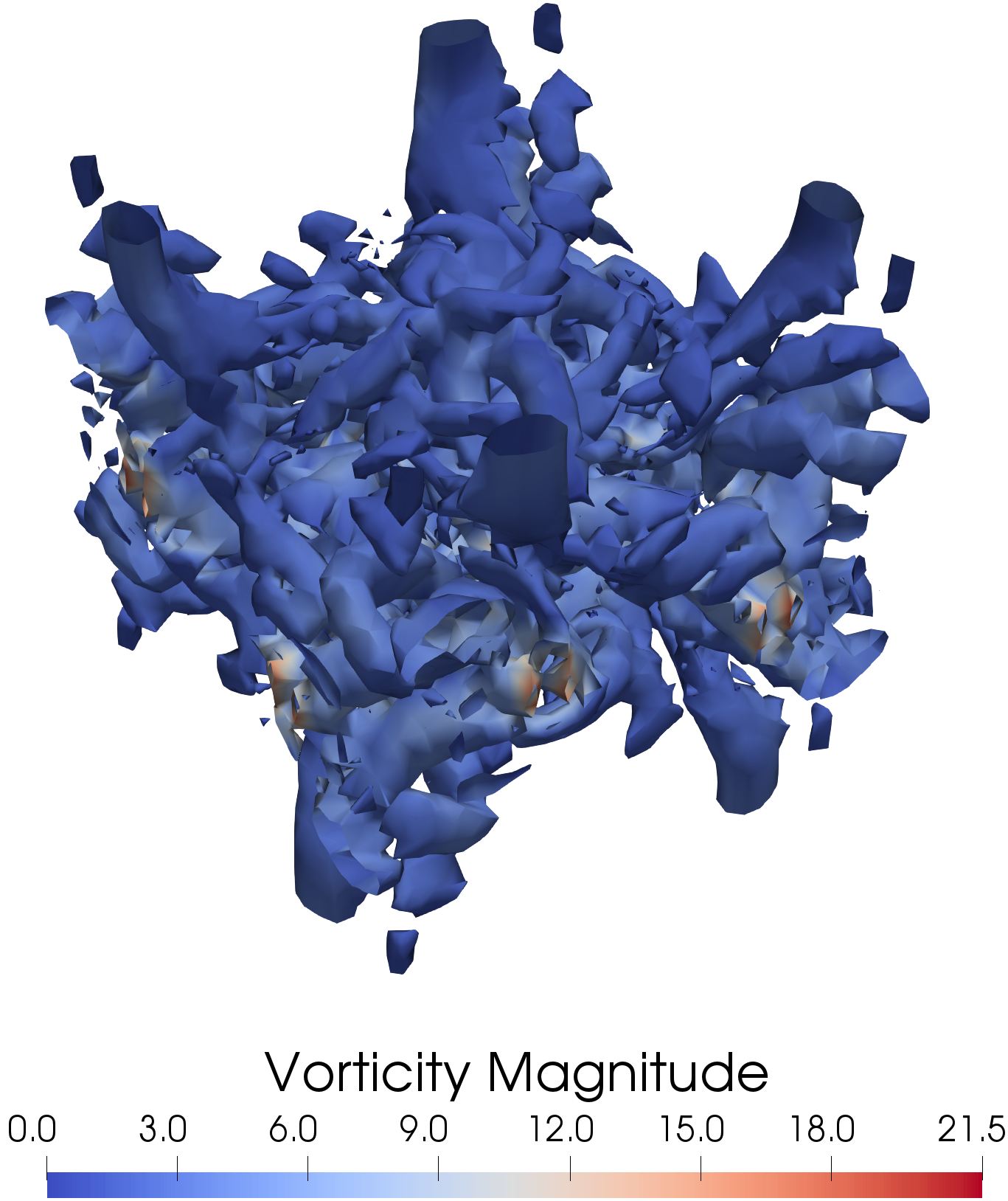}
		\caption{$t\sim10$}
	\end{subfigure}
	\caption{$Q$-criterion isocontours for $Q = 0.3$ attained using a stabilized divergence-conforming discretization of polynomial degree $k' = 1$ on a $32^3$ element mesh for the 3D Taylor--Green vortex problem.}
\label{f23}	
\end{figure}

\noindent Above, $\boldsymbol{S} = \nabla^s \boldsymbol{u}$ is the strain rate tensor and $\boldsymbol{\Omega} = \nabla \boldsymbol{u} - \nabla^s \boldsymbol{u}$ is the rotation rate tensor.  Consequently, regions of positive $Q$ are characterized by a local dominance of vorticity rate over strain rate.  In Figure \ref{f22} and \ref{f23}, we display the $Q$-criterion isocontours for $Q = 0.3$ for three different times corresponding to laminar flow ($t\sim2$), transitional flow ($t\sim6$), and turbulent flow ($t\sim 10$) for the unstabilized and stabilized discretizations of degree $k' = 1$.  The $Q$-criterion isocontours are colored by vorticity magnitude in the figures.  The unstabilized and stabilized $Q$-criterion isocontours look very similar at $t\sim2$.  However, the unstabilized and stabilized $Q$-criterion isocontours look pronouncedly different at $t\sim6$ and $t\sim10$.  In particular, there are far more small-scale structures at these times for the unstabilized discretization as compared with the stabilized discretization.  This is consistent with our previous observations.  Namely, as the unstabilized discretization has no dissipation mechanism besides viscous dissipation, spurious small-scale structures begin to accumulate throughout the domain during laminar-to-turbulent transition and this situation is exacerbated when the flow is fully turbulent.  The energy associated with these spurious small-scale structures results in the pile-up of energy seen at high wavenumbers in the unstabilized energy spectra.  These small-scale structures are not seen in the stabilized results as skeleton stabilization provides an effective model for the transfer of energy from resolved scales to unresolved scales.  This also explains why no pile-up of energy was seen in the stabilized energy spectra.  Similar conclusions may be drawn by comparing unstabilized and stabilized $Q$-criterion isocontours for degree $k' = 2$, but these results are not shown here for brevity.

\section{Conclusions}
\label{cl}
In this paper, we proposed a skeleton stabilization methodology for divergence-conforming B-spline discretizations of the incompressible Navier--Stokes problem.  Our method penalizes jumps in high-order normal derivatives of the velocity field across interior mesh facets to alleviate advective instabilities associated with Galerkin’s method.  Due to the special structure of divergence-conforming B-splines, our method acts only on the components of velocity tangential to interior mesh facets.  We showed that skeleton-stabilized divergence-conforming B-spline discretizations are consistent, pressure robust, and energy stable, and we also showed they admit global balance laws for linear momentum on rectilinear domains and axial angular momentum on cylindrical domains.  We illustrated how to select the stabilization parameter appearing in our method to avoid excessive dissipation in the diffusion-dominated limit and in the cross-wind direction.  We examined the accuracy and robustness of our proposed skeleton stabilization methodology using a suite of numerical experiments.  We first considered a manufactured solution, and from this study, we found our method yields optimal $\textbf{L}^2$- and $\textbf{H}^1$-norm convergence rates for the velocity field.  We next considered a lid-driven cavity problem and found skeleton stabilization removes spurious near-wall eddies that emerge from an application of Galerkin’s method.  We finally considered a 3D Taylor--Green vortex problem and discovered skeleton stabilization is an effective Implicit Large Eddy Simulation (ILES) strategy for this problem.  In particular, skeleton stabilization completely eliminates the pile-up of energy that occurs with Galerkin’s method.

There are many directions we plan to pursue in future work.  First, we plan to conduct a full stability and error analysis of our skeleton stabilization methodology, perhaps building upon similar analyses for edge stabilization of Scott--Vogelius discretizations \cite{BU08}.  Second, we plan to extend our skeleton stabilization methodology to multi-patch divergence-conforming B-spline discretizations as well as divergence-conforming discretizations based on T-splines \cite{buffa2014isogeometric}, LR-splines \cite{johannessen2015divergence}, and hierarchical B-splines \cite{evans2020hierarchical}.  Finally, we plan to extend our skeleton stabilization methodology to multi-physics applications where pointwise mass conservation is desirable, including coupled flow-transport \cite{galvin2012stabilizing}, incompressible fluid-structure interaction \cite{peskin1993improved}, and incompressible magnetohydrodynamics \cite{greif2010mixed}.  We are particularly interested in using skeleton stabilization to stabilize immersogeometric analysis of bio-prosthetic heart valves using divergence-conforming B-splines \cite{KA17}, as our previous stabilization strategy for this application severely limited the spatial accuracy of our resulting analysis scheme.

\section{Acknowledgement}
The first author worked on this project in pursuit of a Master’s thesis at the University of Colorado Boulder.  The second author was partially supported by National Science Foundation award number 2103939 and the third author was partially supported by National Science Foundation award number 2104106.  This work utilized resources from the University of Colorado Boulder Research Computing Group, which is supported by the National Science Foundation (awards ACI-1532235 and ACI-1532236), the University of Colorado Boulder, and Colorado State University.
\bibliographystyle{plain}
\bibliography{ref}{}
\end{document}